\documentclass[12pt,reqno]{amsart}

\usepackage{amsmath, amsthm, amssymb, stmaryrd}

\usepackage{enumerate}
\usepackage{url}
\usepackage{xcolor}  

\usepackage{parskip}

\usepackage{graphicx}

\makeatletter
\@namedef{subjclassname@2020}{%
  \textup{2020} Mathematics Subject Classification}
\makeatother

\let\OLDthebibliography\thebibliography
\renewcommand\thebibliography[1]{
  \OLDthebibliography{#1}
  \setlength{\itemsep}{0pt plus 0.3ex}
}

\usepackage{mathtools}

\usepackage{multirow, array}
\usepackage{placeins}
\usepackage{mathrsfs}

\usepackage{lipsum}
\usepackage{setspace}

\usepackage{hyperref}

\advance \topmargin by -\headheight
\advance \topmargin by -\headsep

\evensidemargin \oddsidemargin
\renewcommand{\baselinestretch}{1}

\newtheorem{thm}{Theorem}[section]
\newtheorem{lemma}[thm]{Lemma}
\newtheorem{prop}[thm]{Proposition}
\newtheorem{cor}[thm]{Corollary}
\newtheorem{conj}[thm]{Conjecture}
\newtheorem{problem}[thm]{Problem}
\newtheorem{qn}[thm]{Question}

\theoremstyle{definition}
\newtheorem{defn}[thm]{Definition}

\theoremstyle{remark}
\newtheorem{remark}[thm]{Remark}

\numberwithin{equation}{section}

\newcommand{\mmod}[1]{{{\,\,\mathrm{mod}\,\,#1}}}

\newcommand*\wrapletters[1]{\wr@pletters#1\@nil}
\def\wr@pletters#1#2\@nil{#1\allowbreak\if&#2&\else\wr@pletters#2\@nil\fi}

\def\set{\mathcal} 
\def\even{}

\def\numgam{c_t}
\def\demgam{q_t'}
\def\grow{\xi}

\def\alp{{\alpha}} \def\bfalp{{\boldsymbol \alpha}}
\def\bet{{\beta}}  
\def\gam{{\gamma}} \def\Gam{{\Gamma}}
\def\del{{\delta}} \def\Del{{\Delta}}

\def\tet{{\theta}}  \def\Tet{{\Theta}}
\def\kap{{\kappa}}
\def\lam{{\lambda}} \def\Lam{{\Lambda}}

\def\sig{{\sigma}} \def\Sig{{\Sigma}}

\def\ome{{\omega}}  
\def\eps{\varepsilon}

\def\le{\leqslant} \def\ge{\geqslant}  
\def \leq {\leqslant} \def \geq {\geqslant}

\def\d{{\,{\rm d}}}

\def \sig{{\sigma}}

\def \bN {\mathbb N}

\def \bQ {\mathbb Q}
\def \bR {\mathbb R}
\def \bZ {\mathbb Z}

\def \ba {\mathbf a}
\def \bb {\mathbf b}

\def \bl {\boldsymbol \ell}

\def \bn {\mathbf n}

\def \bt {\mathbf t}

\def \bzero {\mathbf 0}

\def \balp {{\boldsymbol{\alp}}}

\def \bgam {{\boldsymbol{\gam}}}
\def \bdel {{\boldsymbol{\del}}}
\def \brho {{\boldsymbol{\rho}}}

\def \un {\hat{n}}
\def \um {\hat{m}}

\def \uM {\hat{M}}
\def \uN {\hat{N}}

\def \cA {\mathcal A}
\def \cB {\mathcal B}

\def \cD {\mathcal D}
\def \cE {\mathcal E}

\def \cG {\mathcal G}

\def \cI {\mathcal I}
\def \cJ {\mathcal J}

\def \cL {\mathcal L}

\def \cP {\mathcal P}

\def \cR {\mathcal R}
\def \cS {\mathcal S}
\def \cT {\mathcal T}
\def \cU {\mathcal U}
\def \cV {\mathcal V}
\def \cW {\mathcal W}
\def \cX {\mathcal X}

\def \Nm {\Pi}

\def \det {\mathrm{det}}

\def \diam {\diamond}

\def \dist {{\mathrm{dist}}}

\global\long\def\psio{\psi_{\grow}}%

\global\long\def\Rone{\set{W}(N)}%

\global\long\def\uN{\hat{N}}%

\global\long\def\uM{\hat{M}}%

\global\long\def\A{\mathcal{A}}%

\global\long\def\W{\mathcal{W}}%

\begin{document}

\title[Littlewood and Duffin--Schaeffer-type problems]
{Littlewood and Duffin--Schaeffer-type problems in diophantine approximation}

\author[Sam Chow]{Sam Chow}
\address{Mathematics Institute, Zeeman Building, University of Warwick, Coventry CV4 7AL, United Kingdom}
\email{sam.chow@warwick.ac.uk}

\author[Niclas Technau]{Niclas Technau}
\address{School of Mathematical Sciences, Tel Aviv University, Tel Aviv 69978, Israel}
\address{Department of Mathematics, 
University of Wisconsin, 480 Lincoln Drive, 
Madison, WI, 53706, USA}
\email{technau@math.wisc.edu}

\thanks{}
\date{}

\maketitle

{\centering\footnotesize \em{Dedicated to Andy Pollington} \par}

\begin{abstract} 
Gallagher's theorem describes the multiplicative diophantine approximation rate of a typical vector. We establish a fully-inhomogeneous version of Gallagher's theorem, a diophantine fibre refinement, and a sharp and unexpected threshold for Liouville fibres. Along the way, we prove an inhomogeneous version of the Duffin--Schaeffer conjecture for a class of non-monotonic approximation functions.
\end{abstract}

\renewcommand\contentsname{}

\begin{center}
Table of Contents
\end{center}
\renewcommand{\baselinestretch}{0.75}\normalsize
\tableofcontents
\renewcommand{\baselinestretch}{1.0}\normalsize

\section{Introduction}

This manuscript concerns 
two fundamental problems in diophantine approximation.
We introduce a method to tackle, in a general context, 
inhomogeneous versions of Littlewood's conjecture which 
are metric in at least one parameter. Further, we 
prove an inhomogeneous version of the Duffin--Schaeffer conjecture
for a relatively large class of functions.

Let us begin by explaining the link to Littlewood's conjecture, 
and defer elaborating on the inhomogeneous 
Duffin--Schaeffer conjecture to Section \ref{Sec: Outline}. Around 1930, Littlewood raised the question of whether planar 
badly approximable vectors exist in a multiplicative sense:
That is, if for all $(\alpha_1,\alpha_2) \in \bR^2$, we have
\begin{equation}\label{eq: Littlewood}
\liminf_{n\rightarrow\infty} n \Vert n\alpha_1 \Vert \cdot
\Vert n\alpha_2 \Vert = 0,
\end{equation}
where $\Vert \cdot \Vert$ denotes the distance to the nearest integer. Until the time of writing, finding non-trivial examples of $(\alpha_1, \alpha_2)$ satisfying \eqref{eq: Littlewood}, 
barring rather special cases, evades
the best efforts of the mathematical community. For instance, the problem remains open even for $(\alpha_1, \alpha_2) = (\sqrt{2},\sqrt{3})$. Here non-trivial means that $\alpha_1,\alpha_2$ lie in the set
$$
\mathrm{Bad} = \{ \beta \in \mathbb{R}: 
\exists_{c>0} \quad n\Vert n \beta \Vert > c \quad
\mathrm{for}\,\mathrm{all}\, n \in \bN
\}
$$
of {\em badly approximable numbers}. This set has Lebesgue measure zero, by the Borel--Cantelli lemma, but full Hausdorff dimension, by the Jarn\'{i}k--Besicovitch theorem \cite[Theorem 3.2]{BRV2016}.

\begin{remark}
By Dirichlet's approximation theorem, the inequality $n\Vert n \beta \Vert<1$ holds infinitely often for {\em any} $\beta\in \bR$.
So badly approximable numbers are precisely those numbers
for which Dirichlet's approximation theorem is optimal, up to a constant.
\end{remark}

The study of the measure theory and fractal geometry 
centering around \eqref{eq: Littlewood} has turned out to be a fruitful endeavour. Indeed, it has led to various exciting developments 
in homogeneous dynamics and in diophantine approximation. 
We presently expound upon this.

\subsubsection*{Homogeneous dynamics}

Let $k \in \bN$, define 
\[
G = \mathrm{SL}_{k+1}(\bR),
\qquad \Gam = \mathrm{SL}_{k+1}(\bZ),
\]
and let $D$ be the subgroup of diagonal matrices in $G$. There is a classical correspondence---the \emph{Dani correspondence} \cite{Dan1985}---between 
the diophantine properties of 
in $(\alp_1,\ldots,\alp_k) \in \bR^k$ and the dynamical properties of 
the orbit of
\[
\begin{pmatrix}
1 & 0 &  \cdots &  0 &  \alp_1 \\
& 1 &  &  &\alp_2 \\ 
&& \ddots &  & \vdots \\
&&& 1& \alp_k \\ &&&& 1
\end{pmatrix} \in G/\Gam
\]
under the action of $D$. 
By virtue of Dani's correspondence, Littlewood's conjecture is closely
linked to a conjecture of Margulis \cite[Conjecture 1]{Mar2000}. Below we state a special case, namely \cite[Conjecture 9]{Mar2000}.
\begin{conj} \label{MargulisSpecial} Let $k \ge 2$, and let $G, \Gam, D$ be as above. Let $z \in G/\Gam$, and assume that $Dz \subseteq G/\Gam$ has compact closure. Then $Dz$ is closed.
\end{conj}
If true, Conjecture \ref{MargulisSpecial} would imply Littlewood's conjecture. See \cite{Mar2000} for further details.

With this dynamical perspective, 
Einsiedler, Katok, and Lindenstrauss \cite{EKL2006}
established, \emph{inter alia}, the striking result 
that the set of putative counterexamples to Littlewood's
conjecture \eqref{eq: Littlewood} has Hausdorff dimension zero. 
A crucial ingredient was a deep result of Ratner.

From a similar departure point, Shapira \cite{Sha2011} established
a measure-theoretic, uniform version 
of an inhomogeneous Littlewood-type problem, solving an old problem of Cassels.
To state it, we stress that `for almost all' 
(and similarly for `almost every')
in this manuscript always means with respect to the Lebesgue measure
on the ambient space, expressing that the complement 
of the set under consideration has Lebesgue measure zero. Shapira proved that for almost all $(\alpha,\beta)\in \mathbb{R}^2$ the relation
$$
\liminf_{n\rightarrow \infty} n \Vert n\alpha - \gamma \Vert \cdot 
\Vert n\beta - \delta \Vert =0
$$ 
holds for any $\gamma,\delta \in \bR$.
Gorodnik and Vishe \cite{GP2018} improved this by a factor of 
$(\log \log \log \log \log n)^\lambda$, 
for some constant $\lambda >0$. 

While the above results suggest that Littlewood's conjecture 
might be correct, there is an indication against it:
Adiceam, Nesharim, and Lunnon \cite{ANL2018}
proved that a certain function field analogue
of the Littlewood conjecture is false.

In what follows, the word metric is used to mean `measure-theoretic', as is customary in this area. The goal of metric number theory is to classify behaviour up to exceptional sets of measure zero.

\subsubsection*{Metric multiplicative diophantine approximation}
Alongside the theory of homogeneous dynamics linked to 
Littlewood's conjecture, 
there have been significant advances  
towards the corresponding metric theory. The first systematic result in this direction
is a famous theorem of Gallagher \cite{Gal1962}:\footnote{The convergence part was already known, see \cite[Remark 1.2]{BV2015}. For this reason, \emph{Gallagher's theorem} sometimes refers to the divergence part alone.}
For any non-increasing $\psi: \bN \rightarrow \bR_{>0}$
and almost all $(\alpha_1,\ldots,\alpha_k)\in \bR^k$, we have
\begin{equation}\label{eq: Gallagher inequality}
\Vert n \alpha_1 \Vert \cdots \Vert n \alpha_k \Vert < \psi(n)
\end{equation}
infinitely often if the relevant series of measures diverges,
that is, if 
\begin{equation}\label{eq: sum of measure diverges}
\sum_{n\geq 1} \psi(n) (\log n)^{k-1} = \infty.
\end{equation}
Consequently \eqref{eq: Littlewood} holds true for almost every $(\alpha_1,\alpha_2)\in \bR^2$ by a $\log$-squared margin:
$$
\liminf_{n\rightarrow \infty}  
n (\log n)^2 \Vert n \alpha_1 \Vert \cdot \Vert n \alpha_2 \Vert = 0.
$$
Continuing this line of research, Pollington and Velani \cite{PV2000}
showed the following fibre statement, by exploiting the Fourier decay property of a certain fractal measure: If $\alpha_1 \in \mathrm{Bad}$ then 
there is a set of numbers $\alpha_2 \in \mathrm{Bad}$, 
of full Hausdorff dimension, such that 
$$
n \log n\Vert n \alpha_1 \Vert  \cdot \Vert n \alpha_2 \Vert < 1
$$
holds for infinitely many $n \in \bN$.

A further fibre statement concerning \eqref{eq: Littlewood}
was established by Beresnevich, Haynes, and Velani
in \cite[Theorem 2.4]{BHV2016}.
To state it, we recall that 
{\em Liouville numbers}
are irrational real numbers $\alpha$ such that
for any $w>0$ the inequality
$$
\Vert n \alpha \Vert <n^{-w}
$$
holds infinitely often. We denote the set of Liouville numbers by $\cL$.

\begin{thm} [{\cite[Theorem 2.4]{BHV2016}}]
Let $\alpha_1\in \bR\setminus (\bQ \cup \cL)$ and $\gamma \in \bR$. If the Duffin--Schaeffer conjecture is true, 
then for almost all $\alpha_2\in \bR$ we have 
$$
\liminf_{n\rightarrow \infty} n (\log n)^2 
\Vert n\alpha_1 -\gamma \Vert \cdot \Vert n\alpha_2 \Vert = 0.
$$
\end{thm}
\noindent
When this theorem was proved, the Duffin--Schaeffer conjecture was 
still open. The former was then proved, 
without appealing to the Duffin--Schaeffer
conjecture, by the first named author \cite{Cho2018} in a stronger form:

\begin{thm} [{\cite{Cho2018}}]
Let $\alpha_1,\gamma \in \bR$ and assume that 
$\alpha_1\in \bR\setminus (\bQ \cup \cL)$. If $\psi:\bN \rightarrow \bR_{>0}$ 
is non-increasing and the series 
\begin{equation}\label{eq: divergence assumption plane}
\sum_{n\geq 1} \psi(n) \log n
\end{equation} 
diverges, then for almost all $\alpha_2\in \bR$ there exist infinitely many $n \in \bN$ such that
$$ 
\Vert n\alpha_1 -\gamma \Vert  \cdot \Vert n\alpha_2 \Vert < \psi(n).
$$
\end{thm}
The proof used the structural theory 
of Bohr sets (see Section \ref{Sec: Outline}), as well as
continued fractions and the geometry of numbers, in a crucial way.
This combinatorial--geometric method was further developed by the 
authors \cite{CT2019} to prove higher-dimensional results,
removing the reliance on continued fractions. Instead,
a more versatile framework from the geometry of numbers was brought to bear on the problem. Another interesting facet of the approach is the application of \emph{diophantine transference inequalities} \cite{BV2010, BL2005, BL2010, CGGMS, GM2019, Khi1926} to deal with the inhomogeneous shifts.

By fixing $\alp_1$ above, one considers pairs $(\alp_1, \alp_2)$ which lie on a vertical line in the plane. With Yang, the first named author showed in \cite{CY2020} that if $\cL_0$ is an arbitrary line in the plane then for almost all $(\alp_1,\alp_2) \in \cL_0$ we have
\[
\liminf_{n \to \infty} n(\log n)^2 \| n \alp_1 \| \cdot \| n \alp_2 \| = 0.
\]
Writing
\[
\cL_0 = \{ (\alp,\bet) \in \bR^2: \alp = a \bet + b \},
\]
and assuming that that pair $(a,b)$ satisfies the Lebesgue-generic condition
\begin{equation} \label{generic}
\sup \{ w \in \bR: \exists^\infty (x,y) \in \bZ^2 \quad \| x a + y b \| < (|x| + |y|)^{-w} \} < 5,
\end{equation}
it was also shown there that if $\psi: \bN \to [0,\infty)$ is non-increasing and 
\[
\displaystyle \sum_{n \ge 1} \psi(n) \log n < \infty
\]
then for almost all $(\alp_1, \alp_2) \in \cL_0$ the inequality 
\[
\| n \alp_1 \| \cdot \| n \alp_2 \| < \psi(n)
\]
has at most finitely many solutions $n \in \bN$. 
The divergence theory was attained via an effective asymptotic equidistribution theorem for unipotent orbits in $\mathrm{SL}_3(\bR) / \mathrm{SL}_3(\bZ)$, whilst the convergence statement involved the correspondence between Bohr sets and generalised arithmetic progressions. All of this sits within the broader context of \emph{metric diophantine approximation on manifolds}, for which there is a vast literature \cite{Ber2012, BDV2007, Hua2015, Hua2020, KM1998, VV2006}.

A common feature of these results is that they are homogeneous in the metric parameter, i.e. they involve $\Vert n \alpha_2 \Vert $
but not $\Vert n \alpha_2 - \gam \Vert $
with a general parameter $\gam$. Even a weak inhomogeneous version of
Gallagher's theorem, akin to Shapira's \cite[Theorem 1.2]{Sha2011},
remains completely open, despite numerous attempts. In light of this, Beresnevich, Haynes, 
and Velani \cite[Problem 2.3]{BHV2016} posed the following problem:

\begin{problem}[A fully-inhomogeneous version of Gallagher's theorem on vertical planar lines, weak form]\label{prob2}
Let $\alpha_1, \gamma_1, \gam_2 \in \bR$, and suppose that $\alpha_1 \not\in \cL \cup\bQ$.
Prove that
$$
\liminf_{n\rightarrow\infty}n\,(\log n)^2 
\|n\alpha_1-\gamma_1\| \cdot \|n\alpha_2-\gamma_2\|  
= 0  \quad \  \mbox{for almost all }  \alpha_2 \in \bR.
$$
\end{problem}

They write in the paragraph leading up to \cite[Problem 2.3]{BHV2016} concerning
this problem that it ``currently seems well out of reach''.  Nevertheless, a stronger conjecture was put forth by the first named 
author \cite[Conjecture 1.6]{Cho2018}:

\begin{conj}[A fully-inhomogeneous version of Gallagher's theorem on vertical planar lines, strong form]\label{conj: planar inhom. Gallagher}
Let $\alpha_1, \gamma_1, \gamma_2$ be as in Problem \ref{prob2}.
Suppose $\psi: \mathbb{N}\rightarrow \bR_{>0}$ is non-increasing 
and that the series \eqref{eq: divergence assumption plane} diverges.
Then for almost all $\alp_2 \in \bR$ there exist infinitely many $n \in \bN$ such that
\[
\|n\alpha_1-\gamma_1\|\cdot\|n\alpha_2-\gamma_2\|   < \psi(n).
\]
\end{conj}

This manuscript resolves Problem \ref{prob2}
{\em a fortiori} by proving Conjecture \ref{conj: planar inhom. Gallagher}. 
Additionally, our methods 
are capable of deducing a higher-dimensional generalisation,
as conjectured by the authors in \cite[Conjecture 1.7]{CT2019}.
Furthermore, we resolve
Conjecture 2.1 of Beresnevich, Haynes, and Velani \cite{BHV2016}:

\begin{conj}[A fully-inhomogeneous version of Gallagher's theorem in the plane] \label{PlanarFIG}
Let $\gamma_1,\gamma_2, \in \bR$, and let $\psi:\bN\rightarrow \bR_{>0}$ 
be a non-increasing function such that the series 
\eqref{eq: divergence assumption plane} diverges.  
Then for almost all $(\alpha_1, \alpha_2) \in\bR^2$ the inequality 
\[
\Vert n \alp_1 - \gam_1\Vert \cdot \Vert n \alpha_2 - \gam_2 \Vert < \psi(n)
\]
holds infinitely often.
\end{conj}

Note that Conjecture \ref{conj: planar inhom. Gallagher} implies Conjecture \ref{PlanarFIG}, since $\cL \cup \bQ$ has Lebesgue measure zero.
We now formulate our results in greater detail.

\subsection{Main results}

Recall that the 
{\em multiplicative exponent} $\omega^{\times}(\bfalp)$
of a vector $\bfalp=(\alpha_1,\ldots,\alpha_d)\in \bR^d$
is the supremum of all $w > 0$ such that
$$
\Vert n \alpha_1 \Vert \cdots \Vert n \alpha_d \Vert < n^{-w}
$$
infinitely often. The property of $\omega^{\times}(\bfalp)$ being finite 
can be interpreted as a higher-dimensional generalisation
of being an irrational, non-Liouville number.

\begin{thm} [A fully-inhomogeneous version of Gallagher's theorem on vertical lines, strong form]\label{thm2} 
Let $k\geq 2$. Fix 
$\balp=(\alp_1,\ldots,\alp_{k-1})\in \bR^{k-1}$
and $\gam_1,\ldots,\gam_k \in \bR$.
For $k=2$, suppose that $\alp_1$ is an irrational, non-Liouville number,
and for $k\geq 3$ suppose that
\begin{equation}\label{eq: small mult. exponent}
\omega^{\times} (\balp) < \frac{k-1}{k-2}.
\end{equation}
If $\psi: \bN \to \bR_{>0}$ is non-increasing 
and satisfies \eqref{eq: sum of measure diverges},
then for almost all $\alpha_k \in \bR$ 
there are infinitely many $n \in \bN$ for which
\begin{equation}\label{eq: inhomog Gallagher ineq}
\| n \alp_1 - \gam_1\| \cdots \| n \alp_k - \gamma_k\| < \psi(n).
\end{equation}
\end{thm}

\begin{remark} \label{exceptional}
The set of $\balp \in \bR^{k-1}$ for which 
$$
\omega^{\times} (\balp) \geq \frac{k-1}{k-2}
$$
has Lebesgue measure zero and, stronger still, has Hausdorff dimension strictly less than $k-1$. The former follows from the Borel--Cantelli lemma
and the latter from the work of Hussain and Simmons \cite{HS2018}. The set of Liouville numbers has Lebesgue measure zero and, stronger still, has Hausdorff dimension 0.
\end{remark}

Theorem \ref{thm2} is precisely \cite[Conjecture 1.7]{CT2019}, and we have the following noteworthy special cases.

\begin{cor} 
Conjectures \ref{conj: planar inhom. Gallagher} and \ref{PlanarFIG} are true. 
\end{cor}

We have the following generalisation of Conjecture \ref{PlanarFIG}, which includes its complementary convergence theory.

\begin{cor}[A fully-inhomogeneous version of Gallagher's theorem]
\label{FIG}
Let $\gam_1,\ldots,\gam_k \in \bR$, and let $\psi: \bN \to (0,\infty)$ be a non-increasing function. Write $\cW^\times = \cW^\times(\psi, \gam_1,\ldots,\gam_k)$ for the set of $(\alp_1,\ldots,\alp_k) \in [0,1]^k$ such that \eqref{eq: inhomog Gallagher ineq} has infinitely many solutions $n \in \bN$. Then 
\[
\mu_k(\cW^\times) = 
\begin{cases}
1, &\text{if } \displaystyle \sum_{n=1}^\infty \psi(n) (\log n)^{k-1} = \infty\\
0, &\text{if } \displaystyle \sum_{n=1}^\infty \psi(n) (\log n)^{k-1} < \infty,
\end{cases}
\]
where $\mu_k$ denotes $k$-dimensional Lebesgue measure.
\end{cor}

\noindent The divergence part of Corollary \ref{FIG} is a consequence of Theorem \ref{thm2} and Remark \ref{exceptional}. The convergence part requires only classical techniques, and will be proved in Section \ref{ConvergencePart}. Theorem \ref{thm2} also resolves Problem \ref{prob2} in the following stronger and more general form:

\begin{cor} 
Let $k\geq 2$. Fix a fibre vector 
$\balp=(\alp_1,\ldots,\alp_{k-1})\in \bR^{k-1}$,
and shifts $\gam_1,\ldots,\gam_k \in \bR$.
For $k=2$, suppose that $\alp_1$ is an irrational, non-Liouville number,
and for $k\geq 3$ assume
\eqref{eq: small mult. exponent}.
Then for almost all $\alpha_k \in \bR$ 
there are infinitely many $n \in \bN$ for which
\begin{equation*}
\| n \alp_1 - \gam_1\| \cdots \| n \alp_k - \gamma_k\|
< \frac{1}{n (\log n)^k \log \log n}.
\end{equation*}
\end{cor}

When $\gamma_k = 0$, the results above follow from \cite{Cho2018}. 
In the planar case,
we go beyond the scope of
Problem \ref{prob2}.
Indeed, we also solve it on fibres 
$(\alpha_1,\bR)$ where $\alpha_1$ is a Liouville number:

\begin{thm}\label{thm: Liouville} Let $\alpha_1 \in \cL$, and let $\gam_1, \gam_2 \in \bR$. Then, for almost all $\alpha_2\in\bR$, we have
\[
\liminf_{n\rightarrow \infty} 
n(\log n)^{2} 
\| n \alp_1 - \gam_1\| \cdot
\| n \alpha_2 - \gam_2\|
 = 0.
\]
\end{thm}

In view of Theorem \ref{thm2}, we see that this result holds for any irrational $\alp_1$, Liouville or not. To be clear, we obtain the following statement.

\begin{thm}\label{thm: definitive} Let $\alpha_1 \in \bR \setminus \bQ$, and let $\gam_1, \gam_2 \in \bR$. Then, for almost all $\alpha_2\in\bR$, we have
\[
\liminf_{n\rightarrow \infty} 
n(\log n)^{2} 
\| n \alp_1 - \gam_1\| \cdot
\| n \alpha_2 - \gam_2\|
 = 0.
\]
\end{thm}

However, if $\alp_1 \in \bQ$, then one can easily construct a counterexample by choosing any $\gamma_1 \notin \alp_1 \bZ$ and applying Sz\"usz's theorem \cite{Szu1958}, see Theorem \ref{Sz}. In this sense, and in the sense described in the next two paragraphs, Theorem \ref{thm: definitive} is definitive.

The analysis on Liouville fibres is delicate, owing to the erratic
behaviour of the arising sums of reciprocals of fractional parts \cite{BHV2016}. In light of our earlier discussion on
Problem \ref{prob2} and Conjecture \ref{conj: planar inhom. Gallagher}, 
one might expect Theorem \ref{thm: Liouville}
not to be sharp. Surprisingly, the result is sharp, as we now detail.

Let us ask for a strengthening of Theorem \ref{thm: Liouville}
by considering an approximation function with a faster decay, say
\begin{equation}
\psio(n)=\frac{1}{n(\log n)^{2}\grow(n)}\label{def: psi_1},
\end{equation}
where $\grow: \bN \to [1,\infty)$ is an unbounded and non-decreasing function. Then
the strengthened statement of Theorem \ref{thm: Liouville}, with $\psi(n) = \psi_\grow(n)$, is false:

\begin{thm}\label{thm: log square is sharp}
Let $\grow: \bN \to [1,\infty)$ be non-decreasing and unbounded. 
Then there are continuum many pairs $(\alpha_1, \gamma_1) \in \cL  \times \bR$ 
such that for any $\gamma_2\in\mathbb{R}$ and 
almost all $\alpha_2 \in \bR$ the inequality
\[
\| n \alp_1 - \gam_1\| \cdot
\| n \alpha_2 - \gam_2\|
< \psi_\xi(n)
\]
has at most finitely many solutions $n \in \bN$.
\end{thm}

\begin{remark}
\begin{enumerate}
\item One could regard this as a result `in the opposite direction' to Littlewood's conjecture, though there are several differences. A volume heuristic, and the works of Peck \cite{Peck1961} and Pollington--Velani \cite{PV2000}, suggest that \eqref{eq: Littlewood} can be strengthened by roughly a logarithm---see the discussion in \cite{BV2011}---and meanwhile Badziahin \cite{Bad2013} has shown that it cannot be strengthened much further than that. For this reason, we consider that results in the opposite direction to Littlewood's conjecture are worthy of further study.
\item In the course of our proof, we explicitly construct the pairs $(\alp_1,\gam_1)$. This feature is often not present in results of this flavour.
\end{enumerate}
\end{remark}

To conclude this discussion, Theorem \ref{thm: definitive} is sharp, and has no unnecessary restrictions on $\alp_1, \gam_1,$ and $\gam_2$.

\bigskip

By work of Beresnevich and Velani \cite[Section 1]{BV2015},
as well as Hussain and Simmons \cite{HS2018}, `fractal'
Hausdorff measures are known to be 
insensitive to the multiplicative nature of these types of problems. Fix $k\geq 2$. For $\psi: \bN \to \bR_{>0}$ non-increasing with $\lim_{n\to \infty}\psi(n) = 0$, and $\bgam = (\gam_1,\ldots,\gam_k) \in \bR^k$,
denote by $\cW_k^\times(\psi,\bgam)$ the set of 
$(\alp_1,\ldots,\alp_k) \in\bR^{k}$ satisfying \eqref{eq: inhomog Gallagher ineq}
for infinitely many $n$. Further,
denote by $\cW_k(\psi, \bgam)$ the set of 
$(\alp_1,\ldots,\alp_k) \in [0,1]^k$ 
for which 
\[
\max( \| n \alp_1 - \gam_1 \|, \ldots, \| n \alp_k - \gam_k \| ) < \psi(n)
\]
has infinitely many solutions $n \in \bN$. 
By
\cite[Corollary 1.4]{HS2018} and \cite[Theorem 6.1]{BRV2016}, for $\bgam \in \bR^k$
we have the Hausdorff measure identity
\begin{equation}\label{eq: Hausdorff measure identity}
H^s( {\set W}_k^\times (\psi, \bgam)) = H^{s- (k-1)} ({\set W}_1(\psi,\bgam)) 
\qquad (k-1 < s < k).
\end{equation}
We interpret from this that multiplicatively approximating 
$k$ reals using the same denominator is no different to approximating one of the $k$ numbers,
except possibly for a set of zero Hausdorff $s$-measure. 
This behaviour differs greatly from that of the Lebesgue case $s=k$, 
where there are extra logarithms for the multiplicative problem.
As explained in \cite{BV2015, HS2018}, 
in the remaining ranges for $s$ the Hausdorff theory 
trivialises: 
if $s > k$ then $H^s( {\set W}_k^\times (\psi),\bgam) = 0$, 
irrespective of $\psi$, whereas if $s \le k-1$ 
then $H^s( {\set W}_k^\times (\psi,\bgam)) = \infty$.

\subsection{Key ideas and further results}\label{Sec: Outline}

\subsubsection*{A fully-inhomogeneous fibre refinement of Gallagher's theorem}

Owing to the robustness of our method, much of the argument for Theorem \ref{thm2} is
transparent already in the planar case $k=2$. As it is simpler from a technical point of view,
we shall outline the proof only in this case, and indicate in passing how to generalise to higher dimensions.
To begin, let us isolate the metric parameter $\alpha =\alpha_2$
in \eqref{eq: inhomog Gallagher ineq}
on the left hand side of the inequality. 
As $\Vert \cdot\Vert$ is $1$-periodic,
it suffices if we show 
\begin{equation}\label{eq: Phi dim 2}
\Vert n \alpha - \gamma \Vert <
\Phi(n) :=\frac{ \psi(n)} { \| n \alp_1 - \gam_1 \| }, 
\qquad \gamma := \gamma_2,
\end{equation}
holds for almost every $\alpha \in [0,1]$ infinitely often.
If $\Phi$ were a non-increasing function, then we could utilise 
Sz\"usz's extension of Khintchine's theorem, which grants a sharp description of the inhomogeneous approximation rate of a generic real number: 
\begin{thm}[Sz\"usz \cite{Szu1958}]
\label{Sz}
If $\Psi: \bN \rightarrow \bR_{>0}$ is non-increasing
and $\gam \in \bR$, then the Lebesgue measure of the set of $\alp \in [0,1]$ for which
\begin{equation}\label{eq: inhom Khintchine}
\Vert n \alp -\gamma \Vert < \Psi(n)
\end{equation}
holds infinitely often is 1 (resp. 0) if 
\begin{equation}\label{eq: sum of measures in Khintchine div}
\sum_{n\geq 1} \Psi(n)
\end{equation}
diverges (resp. converges).
\end{thm}
Since $\Phi$ is very much not monotonic,
deducing \eqref{eq: Phi dim 2} infinitely often for almost all $\alp$,
from the divergence of $\sum_n \Phi(n)$,
is far more demanding. 
In fact, it is known that this naive condition is
insufficient, as Duffin and Schaeffer \cite{DS1941}
pointed out: There exists $\Psi: \bN \to \bR_{>0}$ such that 
\eqref{eq: sum of measures in Khintchine div} diverges but for $\gam = 0$ and almost all $\alpha \in [0,1]$ the inequality
\eqref{eq: inhom Khintchine} holds at most finitely often. This was generalised by Ram\'irez in \cite{Ram2016}.

To circumvent their counterexamples, Duffin and Schaeffer restricted attention to reduced fractions, and correspondingly imposed the condition that the series
\begin{equation}\label{eq: Duffin--Schaeffer condition}
\sum_{n\geq 1} \Psi(n) \frac{\varphi(n)}{n} 
\end{equation}
diverges, where $\varphi$ is Euler's totient function. The Duffin--Schaeffer conjecture 
was a major open problem in diophantine approximation
since the 1940s. Over the course of the nearly eight decades,
various partial results towards the Duffin--Schaeffer conjecture were obtained, by 
\begin{itemize}
\item Erd\H{o}s \cite{Erd1970} in 1970, Vaaler \cite{Vaa1978} in 1978
\item Pollington--Vaughan \cite{PV1990} in 1990
\item Harman \cite{Har1990} in 1990, Haynes--Pollington--Velani \cite{HPV2012} in 2012,
\end{itemize}
and several other authors \cite{Aist2014, Fufu, BHHV2013}. Recently, Koukoulopoulos and Maynard (2019) broke through with a complete proof of
the Duffin--Schaeffer conjecture:
 
\begin{thm}[Koukoulopoulos--Maynard \cite{KM2020}] \label{DS}
If $\Psi: \bN \rightarrow \bR_{>0}$ is such that
\eqref{eq: Duffin--Schaeffer condition} diverges then, for almost all $\alp \in [0,1]$, the inequality
\[
\vert n \alp  - a \vert < \Psi(n)
\]
holds for infinitely many coprime $a,n\in \bN$.
\end{thm}
A natural generalisation
would be an inhomogeneous version 
of the Duffin--Schaeffer conjecture:

\begin{conj}[Inhomogeneous Duffin--Schaeffer, see Ram\'irez \cite{Ram2016}]
\label{Felipe}
Let $\gamma \in \mathbb{R}$.
If $\Psi: \bN \rightarrow \bR_{>0}$ is such that
\eqref{eq: Duffin--Schaeffer condition} diverges, then for almost all $\alp \in [0,1]$ the inequality
\[
\vert n \alp  - a - \gamma \vert < \Psi(n)
\]
holds for infinitely many 
coprime $a,n\in \bN$.
\end{conj}

For us it is enough to establish a similar result for 
a concrete class of functions of the shape 
\eqref{eq: Phi dim 2}. 
We consider sets $\cA_n$ that are roughly of the form
$$
\{ \alpha \in [0,1]: 
\Vert n \alpha - \gamma \Vert < \Phi(n)\}.
$$
By standard probabilistic arguments, it suffices to show that the measures of the sets $\cA_n$
are not summable, and that the sets
are quasi-independent in an averaged sense.
The latter property is, as always, the crux of the matter,
and involves {\em overlap estimates} that quantify
how large $\mu({\set A}_n \cap {\set A}_m)$ is compared 
to $\mu({\set A}_n)\mu({\set A}_m)$ on average.

To this end, we may confine our analysis 
to a reasonably large set ${\set G}$ of `good' indices $n$. 
To simplify matters, we 
decompose $\bN$ into dyadic ranges, wherein $n\asymp N$,
and in addition the oscillating factor
$\| n \alp_1 - \gam_1 \|$ has a fixed order of magnitude.
Sets of such integers $n$ are essentially \emph{Bohr sets}
\begin{equation} \label{Bohr1}
{\set B} = {\set B}_{\alp_1}^{\gam_1} (N; \rho_1) := \{ n \in \bZ: |n| \le N, \| n \alp_1 - \gam_1 \| \le \rho_1 \},
\end{equation}
which appear in many areas of mathematics. A novelty of this paper is to show how to handle 
the overlap estimates via congruences in Bohr sets. Here the structural theory from our previous work \cite{CT2019}, constructing the correspondence between Bohr sets and generalised arithmetic progressions---a central pillar of additive combinatorics \cite{TV2006}---in the present context, plays a pivotal role. Previously there was progress made in this direction by Tao--Vu \cite{TV2008}, and by the first named author \cite{Cho2018}. Further, it will be helpful to group $m,n$ according to the size of the greatest common divisor $d$ of $m$ and $n$.

We handle the overlap estimates by averaging over indices 
$m, n$ from different Bohr sets of the shape \eqref{Bohr1}. 
After summing over different dyadic ranges, 
we can then infer the required quasi-independence on average. In the course of our analysis, we need to count solutions to congruences in generalised arithmetic progressions that are essentially Bohr sets. The range in which $d$ is large requires extra care: For $\gamma \notin \cL \cup \bQ$, a repulsion stemming from this diophantine assumption enables us to treat this challenging case. For $\gam \in \cL \cup \bQ$, an additional argument enables us to crack this devilish final case; the idea is to introduce a counterpart to reduced fractions, which we call `shift-reduced'. 

\begin{defn} \label{defn: shift-reduced}
Let $\gam \in \bR$, $\eta \in (0,1)$, and $n \in \bN$. Let $c_0/q'_0, c_1/q'_1, \ldots$ be the continued fraction convergents of $\gam$, see Subsection \ref{CFbit} for what these are.\footnote{In this paper, we will employ the continued fraction expansions of quantities denoted by $\alp$ and $\gam$. We reserve the more common notation $p_0/q_0, p_1/q_1, \ldots$ for the continued fraction convergents of a quantity denoted by $\alp$, see Sections 4 and 5.} Denote by $\numgam/\demgam$ the continued fraction convergent of $\gamma$ for which $t$ is maximal satisfying 
\begin{equation} \label{def: LiouvilleThing}
q_t' \leq n^{\eta}.
\end{equation}
The pair $(a,n) \in \bZ \times \bN$ is
\emph{$(\gam, \eta)$-shift-reduced} 
if $ (\demgam a + \numgam, n) = 1$.
\end{defn}

\noindent To our knowledge, this notion has not hitherto appeared in the diophantine 
approximation literature.

\begin{remark}
Note that $\gamma$ can be a rational in the above definition, in which case the sequence $(q_t')_t$ terminates. Indeed, in the case $\gam = 0 $, 
we recover the traditional notion of 
reduced fractions: Letting
$\numgam=0$ and $\demgam=1$,
the fraction $a/n$ is reduced if and only if the pair $(a,n)$ is 
$(0,1/2)$-shift-reduced. Moreover, if $\gam \in \bQ$ and $n^\eta$ is greater than or equal to the denominator of $\gam$, then $\gam = c_t/q'_t$. We provide some background on continued fractions in Section \ref{CFbit}.
\end{remark}

Our definition of shift-reduced fractions is sensitive to the shift $\gam$. This finesse slightly complicates matters, because we lose measure by not using all fractions. However, by sieve theory we are able to show that shift-reduced fractions are at least as prevalent as reduced fractions, and we are then able to establish the divergence of the relevant series. We suspect that the idea of using shift-reduced fractions will be useful for other arithmetic problems, including perhaps a version of the inhomogeneous Duffin--Schaeffer conjecture, as we will shortly discuss further.

Yet all of this combined yields only  
that \eqref{eq: inhomog Gallagher ineq} holds infinitely often on a set of positive measure.
In the absence of a zero--one law for inhomogeneous diophantine approximation,
we need to carefully `localise' the overlap estimates
to indeed deduce that \eqref{eq: inhomog Gallagher ineq}
holds for a set of $\alp_k$ of full measure; this machinery, though not confined to the realm of metric diophantine approximation, is well-explained in the monograph of Beresnevich, Dickinson, and Velani \cite{BDV2006}. This subtlety complicates the analysis non-trivially,
as it requires us to keep hold of a factor of $\mu(\cI)$ 
throughout. In the regime of $m,n$ in which $d$ is essentially constant 
it turns out to be surprisingly difficult to do so.
To avoid this scenario, we introduce 
artificial powers of $4$ as divisors, which guarantees us that $d$ is not too small. This is an unorthodox manoeuvre, but one that we found useful in practice, and one that may find other uses.

In the course of our proof, we establish the following weakened version of Conjecture \ref{Felipe} for a class of non-monotonic approximating functions generalising $\Phi$ as defined in \eqref{eq: Phi dim 2}.

\begin{conj} [Weak inhomogeneous Duffin--Schaeffer conjecture]
\label{WeakInhomogeneous}
Let $\gamma \in \mathbb{R}$.
If $\Psi: \bN \rightarrow \bR_{>0}$ is such that
\eqref{eq: Duffin--Schaeffer condition} diverges, then for almost all $\alp \in [0,1]$ the inequality
\[
\| n \alp  - \gamma \|< \Psi(n)
\]
has infinitely many solutions $n\in \bN$.
\end{conj}

The $\gam = 0$ case of this was coined the \emph{weak Duffin--Schaeffer conjecture} by Ram\'irez \cite[Conjecture 9.1]{Ram2017}. The word `weak' is used for the following reason: Euler's totient function is present in the divergence hypothesis despite there being no coprimality aspect.

\begin{thm} [Special case of weak inhomogeneous Duffin--Schaeffer]
\label{thm: SpecialCaseWeak}
Let $k\geq 2$. Fix 
$\balp=(\alp_1,\ldots,\alp_{k-1})\in \bR^{k-1}$
and $\gam_1,\ldots,\gam_k \in \bR$.
For $k=2$, suppose that $\alp_1$ is an irrational, non-Liouville number,
and for $k\geq 3$ assume \eqref{eq: small mult. exponent}.
Let $\psi: \bN \to \bR_{>0}$ be a non-increasing 
function satisfying \eqref{eq: sum of measure diverges}, and let 
\begin{equation} \label{PhiDef}
\Phi(n) = \frac{\psi(n)}{\| n \alp_1 - \gam_1\| \cdots \| n \alp_{k-1} - \gam_{k-1}\|} \qquad (n \in \bN).
\end{equation}
Then Conjecture \ref{WeakInhomogeneous} holds for $\Psi = \Phi$.
\end{thm}

We wonder if there is a sharp dichotomy along the lines of Conjecture \ref{WeakInhomogeneous}.

\begin{qn} [Inhomogeneous Duffin--Schaeffer dichotomy]
\label{InhomogDichotomy}
Let $\gam \in \bR$ and $\Psi: \bN \to \bR_{>0}$. Does there exist $\eta \in (0,1)$ with the following property? Denote by $\cW(\Psi; \gamma, \eta)$ the set of $\alp \in [0,1]$ such that
\[
| n \alp - \gam - a| < \Psi(n), \qquad (a,n) \text{ is } (\gam,\eta)\text{-shift-reduced}
\]
has infinitely many solutions $(a,n) \in \bZ \times \bN$. Then
\[
\mu(\cW(\Psi; \gamma, \eta))
= \begin{cases} 1,&
\text{if }
\displaystyle \sum_{n \ge 1} \frac{\varphi_{\gam,\eta}(n)}n \Psi(n) = \infty \\
 0,&
\text{if }
\displaystyle \sum_{n \ge1} \frac{\varphi_{\gam,\eta}(n)}n \Psi(n) < \infty,
\end{cases}
\]
where
\[
\varphi_{\gam,\eta}(n) = \# \{ a \in \{1,2,\ldots,n\}: (a,n) \text{ is } (\gam,\eta)\text{-shift-reduced} \}.
\]
\end{qn}

\begin{remark} 
\label{rem: QuestionRemark}
It is not clear at this stage whether $\eta$ should need to depend on $\gam$ or $\Psi$. Moreover, it could be that the property holds for all $\eta \in (0,\eta_0)$, for some $\eta_0 \in (0,1]$. This question has the appeal of a matching convergence theory, {\em unlike} Conjecture  \ref{WeakInhomogeneous}.
\end{remark}

We are also able to answer Question \ref{InhomogDichotomy} positively for $\Psi = \Phi$, subject to natural assumptions:

\begin{thm} [Special case of inhomogeneous Duffin--Schaeffer dichotomy]
\label{thm: SpecialCase}
Let $k\geq 2$. Fix 
$\balp=(\alp_1,\ldots,\alp_{k-1})\in \bR^{k-1}$
and $\gam_1,\ldots,\gam_k \in \bR$.
For $k=2$, suppose that $\alp_1$ is an irrational, non-Liouville number,
and for $k\geq 3$ assume \eqref{eq: small mult. exponent}.
Let $\psi: \bN \to \bR_{>0}$ be a non-increasing 
function satisfying \eqref{eq: sum of measure diverges}, and let $\Phi$ be as in \eqref{PhiDef}. Then, with $\gam = \gam_k$ and the notation of Question \ref{InhomogDichotomy}, there exists $\eta \in (0,1)$ such that
\begin{equation} \label{eq: Positive1}
\sum_{n=1}^\infty \frac{\varphi_{\gam,\eta}(n)}{n} \Phi(n) = \infty
\end{equation}
and 
\begin{equation}
    \label{eq: Positive2}
\mu(\cW(\Phi;\gam,\eta))=1.
\end{equation}
In particular, Question \ref{InhomogDichotomy} has a positive answer for $\Psi = \Phi$.
\end{thm}

Our method comes close to directly establishing Theorem \ref{thm: SpecialCase}. However, it misses a pathological case where an auxiliary function exceeds $1/2$ infinitely often, causing the relevant intervals to overlap and their union to have smaller measure. We are able to circumvent this by ad-hoc means, and provide the details in an appendix.

Inhomogeneous, non-monotonic diophantine approximation is also discussed in Harman's book \cite[Chapter 3]{Har1998}, as well as in recent work of Yu \cite{YuFourier, YuEV}.

Our method is robust 
with respect to the dimension $k$.
If $k\geq 3$, the combinatorics for controlling 
the overlap estimates is relatively similar 
to the planar case, $k=2$. The notable differences are that 
there are more dyadic ranges to sum over and the choice
of cutoff parameters needs to be adapted.

\subsubsection*{An inhomogeneous version of Gallagher's theorem on Liouville fibres}

For Theorem \ref{thm: Liouville}, the overall structure of our proof parallels that of the Duffin--Schaeffer theorem \cite[Theorem 2.3]{Har1998}, which is a special case of the Duffin--Schaeffer conjecture. This naturally leads us to count solutions to congruences in Bohr sets. In this setting, the latter are rather sparse in the set of positive integers, which presents new difficulties. The partial quotients, see Section \ref{Sec: Preliminaries},
grow extremely rapidly infinitely many times, and we use this together with 
classical continued fraction analysis to deal with the combinatorial aspects
of the overlap estimates.

\subsubsection*{Liouville fibres: sharpness}

The prove the sharpness result, Theorem \ref{thm: log square is sharp}, we construct pairs 
$(\alpha,\gamma) \in \cL \times \bR$ in such a way as to keep $\Vert n \alpha - \gam \Vert$ away from zero. We achieve this via the Ostrowski expansion \cite[Section 3]{BHV2016}, by choosing each Ostrowski coefficient of $\gam$ with respect to $\alp$ to be roughly half times the corresponding partial quotient of $\alp$, and by choosing $\alp$ to have extremely rapidly-growing partial quotients.

\subsection{Open problems}

We have already discussed a few open questions. Here
are some that we have yet to cover.

\subsubsection*{Relaxing the diophantine condition}

We expect that the assumption \eqref{eq: small mult. exponent}
in Theorem \ref{thm2} can be somewhat relaxed.
This likely requires different methods.

\subsubsection*{Convergence theory}

It would be desirable to advance the 
convergence theory complementing Theorem \ref{thm2}, especially for $k\geqslant 3$.
For $k=2$, progress has been made by Beresnevich, Haynes, and Velani \cite{BHV2016}.

By partial summation and the Borel--Cantelli lemma, 
proving the desired convergence statement can be reduced to showing that
$$
\sum_{n\leq N} \frac{1}{\Vert n \alpha_1 - \gamma_1 \Vert \cdots 
\Vert n \alpha_{k-1} - \gamma_{k-1} \Vert} \ll N (\log N)^{k-1}.
$$
In the case $k=2$ and $\gamma_1 = 0$, for a generic choice of 
$\alpha_1\in \bR$ this bound is false,
see \cite[Example 1.1]{BHV2016}. 
However, the logarithmically averaged sums 
$$
S_{\balp}^{\bgam}(N)= 
\sum_{n\leq N} \frac{1}{n\Vert n \alpha_1 - \gamma_1 \Vert \cdots 
\Vert n \alpha_{k-1} - \gamma_{k-1} \Vert},
\qquad \bgam = (\gam_1,\ldots,\gam_{k-1})
$$
could be better behaved, and accurately bounding these would lead to a similar outcome. 
On probabilistic grounds, we expect that 
\[
S_{\balp}^{\bgam}(N) \ll_{\balp,\bgam} (\log N)^{k}.
\]
Subject to a generic diophantine condition on $\balp$,
it is less difficult to prove matching lower bounds via dyadic pigeonholing and estimates for the cardinality of a Bohr set, see \cite[Lemma 3.1]{CT2019}. So the task is to determine the order of magnitude of $S_{\balp}^{\bgam}(N)$. 
Beresnevich, Haynes and Velani \cite[Theorem 1.4]{BHV2016}
showed that if $k=2$ and $\gam\in\bR$
then $S_{\alp}^{\gam}(N) \ll_{\alp,\gam} (\log N)^2$
for almost every $\alp\in\bR$.
With this in mind, we pose the following problem which, if resolved in a sufficiently positive manner,
would entail a coherent convergence theory. 
\begin{problem}
Let $k\geq 3$, and let $\mathcal{C}_k$ be the set
of $(\balp,\bgam)\in \bR^{k-1}\times \bR^{k-1}$ for which
$$
S_{\balp}^{\bgam}(N)\ll_{\balp,\bgam} (\log N)^{k}.
$$ 
Is the set $\mathcal{C}_k$ non-empty? What is its Lebesgue measure?
\end{problem} 

Empirical evidence mildly supports the assertion that $S_{\balp}^{\bzero}(N) \ll_{\balp} (\log N)^3$ holds for generic values of $\balp=(\alp_1,\alp_2)$. We randomly generated 
\[
\alp_1 \approx 0.957363115715396, \qquad
\alp_2 \approx 0.3049448415027476.
\]
With
\[
H = 10^6, \qquad c = \frac{S_{\balp}^{\bzero}(H)}{(\log H)^3} \approx 1.73475,
\]
we plotted $S_{\balp}^{\bzero}(N)$ and $c(\log N)^3$ against $N$ for $N = 2,3,\ldots,H$, see Figure \ref{LogAvg}. There are `jumps' when $\|n \alp_1 \| \cdot \| n \alp_2 \|$ is very small, but these do not appear to affect the order of magnitude of $S_{\balp}^{\bzero}(N)$. For further discussion, we refer the reader to the article by L\^{e} and Vaaler \cite{LV2015}.

\begin{figure}[!ht] 
    \centering
\includegraphics[height=7cm]{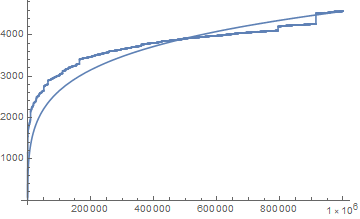}
\caption{$S^{\bzero}_{\balp}(N)$ and $c(\log N)^3$ against $N$}
\label{LogAvg}
\end{figure}

\subsubsection*{Dual approximation}
There are natural dual versions of Gallagher's theorem. 
Loosely speaking, in the dual framework one studies
how close a given vector is to a hyperplane of bounded height,
as the height bound increases.
In fact, we intend to address the next problem in a future work.

\begin{conj} Let $k\geq 2$, and
$(\alp_1, \ldots, \alp_{k-1}) \in \bR^{k-1}$.
Then for almost all $\alp_k \in \bR$ 
there exist infinitely many $(n_1,\ldots,n_k) \in \bZ^k$ such that
\begin{equation*}
 \| n_1 \alp_1 + \cdots + n_k \alp_k \| < 
\frac{1}{H(\bn) (\log H(\bn))^{k}},
\end{equation*}
where $H(\bn) = H(n_1,\ldots,n_k)=
n_1^+ \cdots n_k^+$ and $n^+ = \max(|n|, 2)$.
\end{conj}

This is also \cite[Conjecture 1.9]{CT2019}, which comes with some further discussion.

The dual convergence theory is also of interest, that is, to show that the convergence of the series in
\eqref{eq: sum of measure diverges}
implies that for almost almost all $\alpha_k$ the inequality
\begin{equation*}
 \| n_1 \alp_1 + \cdots + n_k \alp_k \| < 
\psi(H(\bn))
\end{equation*}
holds at most finitely often.
As with the usual multiplicative approximation problems described above, 
the convergence theory in the dual setting is very much open. 
There is a discussion of the relevant sums, as well as a reference 
to a possible departure point, in the paragraphs surrounding Conjecture 1.1 
of Beresnevich, Haynes, and Velani \cite{BHV2016}.

\subsubsection{A Hausdorff dimension problem}

In Theorem \ref{thm: log square is sharp}, we did not estimate the Hausdorff dimension of the set of pairs $(\alp_1,\gam_1) \in \bR^2$ 
such that for any $\gamma_2\in\mathbb{R}$ and 
almost all $\alpha_2 \in \bR$ the inequality
\[
\| n \alp_1 - \gam_1\| \cdot
\| n \alpha_2 - \gam_2\|
< \psi_\xi(n)
\]
has at most finitely many solutions $n \in \bN$. It would be of interest to do so, for $\xi$ increasing slowly to infinity, or even for $\xi(n) = \log \log n$. This problem can be further simplified by fixing $\gam_2 = 0$.

\subsection{Organisation and notation} \label{OrgNot}

\subsubsection*{Organisation of the manuscript}

In Section \ref{Sec: Preliminaries}, we collect
tools and technical lemmata 
that would otherwise disrupt the course 
of the main arguments. In
Section \ref{Sec: Inhomog. Gallagher}, we prove Theorems \ref{thm2} and \ref{thm: SpecialCaseWeak} together, followed by the convergence part of Corollary \ref{FIG}.
Thereafter, we prove Theorems \ref{thm: Liouville}
and \ref{thm: log square is sharp} in Sections \ref{Sec: Gallagher on Liouville fibres} and \ref{Sec: log square sharp}, respectively. Appendix \ref{pathology} describes how the proof of Theorem \ref{thm: SpecialCaseWeak} can be adapted to give Theorem \ref{thm: SpecialCase}.

\subsubsection*{Notation}

We use the Vinogradov and Bachmann--Landau  notations: 
For functions $f$ and positive-valued functions $g$, we write $f \ll g$ 
or $f = O(g)$ if there exists a constant $C$ such that 
$|f(x)| \le C g(x)$ for all values of $x$ under consideration. We write $f \asymp g$ or $f = \Tet(g)$ if $f \ll g \ll f$.
Throughout this manuscript, the implied constants 
are allowed to depend on:
\begin{itemize}
\item The approximation function $\psi$
\item A fixed vector $\balp \in \bR^{k-1}$, and $\gam_1,\ldots,\gam_k \in \bR$
\item A constant $C \ge 2$, which 
in turn only needs to depend on $\balp,\gam_1,\ldots,\gam_k$,
specifying the ranges $[C^j,C^{j+1}]$
to which we localise various parameters
\item A small positive constant $\eps_0$, which only needs to depend on $\balp, \gam_1,\ldots,\gam_k$.
\end{itemize} 
These dependencies shall usually not be indicated by a subscript.
To be explicit, we do consider the dimension $k$ of the ambient 
space to be data of the vector $\balp$ and hence shall not indicate
its dependency.
If any other dependence occurs, we
record this using an appropriate subscript. 
If $\cS$ is a set, we denote the cardinality of $\cS$ by $|\cS|$ or $\# \cS$. 
The symbol $p$ is reserved for primes. The pronumeral $N$ 
denotes a positive integer, sufficiently large in terms of 
$\balp, \gam_1,\ldots,\gam_k$, the approximation function $\psi$, and a bounded interval $\cI$ that will arise in the course of some of the proofs. For a vector $\balp \in \bR^{k-1}$, we abbreviate 
the product of its coordinates to
\begin{equation} \label{eq: ProductNotation}
\Nm (\balp) = \alpha_1 \cdots \alpha_{k-1}.
\end{equation}
When $x \in \bR$, 
we write $\| x \|$ for the distance from $x$ to the nearest integer.
Furthermore, $\mu$ denotes one-dimensional Lebesgue measure.
Given  $\cS\subseteq\mathbb{R}$, we write $\cS_{\leq X}$ for $\{ x \in \cS: x \le X \}$.

Finally, 
we often have to deal with expressions such as
$$
\sum_{n\leq N} \frac{1}{(\log n) \log \log n},
$$
which are not always well-defined because of finitely many small 
$n \in \bN$. To deal with this, we write $\ln(x)$ for the natural logarithm of a positive real number $x$, and put $\log(x) =\max (\ln(x), 1)$ to ensure that these logarithms and their iterates are indeed well-defined and positive.

\addtocontents{toc}{\protect\setcounter{tocdepth}{1}}

\subsection*{Funding and acknowledgements}

SC was supported by EPSRC Fellowship Grant EP/S00226X/2.
NT was supported by the European Research Council (ERC) 
under the European Union’s Horizon 2020
research and innovation programme (Grant agreement No. 786758),
as well as by the Austrian Science Fund (FWF) grant J 4464-N,
and is grateful to Rajula Srivastava for her eternal encouragement.
We thank Sanju Velani for his enthusiasm towards this topic, Andy Pollington for many enthralling conversations, Victor Beresnevich for further encouragement, and Zeev Rudnick for comments on a draft. Last, but not least, we thank the referee for a thorough reading and for helpful remarks, including a nice proof Lemma \ref{detd} that we have now included.

\addtocontents{toc}{\protect\setcounter{tocdepth}{2}}

\section{Preliminaries}\label{Sec: Preliminaries}

In this subsection, we gather together a panoply of tools. We begin with the theory of continued fractions, before continuing to that of Bohr sets. Then we present some standard measure theory and real analysis. The latter will enable us to introduce the artificial divisors to which we alluded earlier, at essentially no cost. After that we discuss some estimates from the geometry of numbers, whose raison d'\^etre is to count solutions to congruences in generalised arithmetic progressions. We then review some basic prime number theory and sieve theory, culminating in the fundamental lemma of sieve theory, which we will later use to count shift-reduced fractions.

\subsection{Continued fractions}
\label{CFbit}

The material here is standard; see for instance \cite[Chapter 1, Section 2]{Bugeaud2004}.
For each irrational number $\alpha$ there exists a unique
sequence of integers $a_0,a_1,a_2, \ldots$, the {\em partial quotients
of $\alpha$}, such that $a_j\ge 1$ for $j \ge 1$ and
\[
\alp = a_0 + \dfrac{1}{a_1 + \dfrac{1}{a_2 +
\raisebox{-2ex}{$\ddots$}}}.
\]
For $j \ge 0$, let $[a_0; a_{1},\ldots,a_{j}]$ denote the truncation the above infinite continued
fraction at place $j$, and let $p_{j} \in \bZ$ and $q_{j} \in \bN$
be the coprime integers satisfying 
\[
[a_0; a_{1},\ldots,a_{j}]=\frac{p_{j}}{q_{j}}.
\]
This rational number is the called the {\em $j$-th convergent to $\alpha.$}
Continued fractions enjoy an impressive portfolio of beautiful properties. One particular
feature for which we have ample need is the following recursion:
For $j\geq 1$, we have 
\begin{equation}\label{eq: continued fraction recursion}
q_{j+1}=a_{j+1}q_{j}+q_{j-1}\quad \mathrm{and}\quad p_{j+1}=a_{j+1}p_{j}+p_{j-1}.
\end{equation}
The initial values are
\[
p_0 = a_0, \qquad q_0 = 1, \qquad p_1 = a_1 a_0 + 1, \qquad q_1 = a_1.
\]
For each $j \ge 0$, the quantity
\begin{equation} \label{def: Dk}
D_{j}=q_{j}\alpha-p_{j}
\end{equation}
satisfies the bound
\begin{equation}
\label{eq: standard estimate}
\frac{1}{2}\leq\left|D_{j}\right|q_{j+1}\leq1.
\end{equation}
Furthermore, it is well-known that the signs $D_{j}/\vert D_{j} \vert$
are alternating. 
For an irrational number $\alpha \in \bR$,
the {\em diophantine exponent} of $\alp$ is 
$$
\omega (\alp) = \sup\{w>0: \Vert q \alp \Vert < q^{-w}\, 
\mathrm{for\, infinitely\,many\,} q\geq 1\}.
$$
Note that $\ome(\alp) = \ome^\times(\alp)$.
We require the following well-known fact
characterising diophantine numbers in terms of the growth 
of the consecutive continued fraction denominators \cite[Lemma 1.1]{BHV2016}: 

\begin{lemma} \label{BHV11}
Let $\alpha \in \bR \setminus \bQ$,
and $(q_k)_k$ be the sequence 
of its continued fraction denominators. 
Then
\[
\omega (\alpha) = \limsup_{k\rightarrow \infty} 
\frac{\log q_{k+1}}{\log q_k}.
\]
In particular $\alpha \not\in {\set L}$ if and only if 
$\log q_{k+1} \ll \log q_k $.
\end{lemma}

\bigskip

Continued fractions have been used to prove a highly aesthetic result called the \emph{three gap theorem}. For $\alp \in \bR \setminus \bQ$ and $m \in \bN$, this asserts that there are at most three distinct gaps $d_{i+1} - d_i$, where
\[
\{ d_1, \ldots, d_m \} = \{ i \alp - \lfloor i \alp \rfloor: 1 \le i \le m \}
\]
and
\[
0 = d_0 < \cdots < d_{m+1} = 1.
\]
Many find this surprising at first. The sizes of the gaps can be computed and described in terms of the continued fraction expansion. We will need only the size of the largest gap. The following theorem combines parts of Theorem 1 and Corollary 1 of \cite{MK1998}, and adopts the typical convention that $q_{-1} = 0$.

\begin{thm} \label{LargestGap}
Let $\alp \in \bR \setminus \bQ$ and $m \in \bN$. Then:
\begin{enumerate}[(a)]
\item There is a unique representation
\[
m = r q_k + q_{k-1} + s,
\]
for some
\[
k \ge 0, \qquad 1 \le r \le a_{k+1}, \qquad 0 \le s \le q_k - 1.
\]
\item If $s < q_k - 1$ then
\[
\max \{ d_{i+1} - d_i: 0 \le i \le m \}
=
\begin{cases}
|D_{k+1}| + |D_k|, &\text{if } r = a_{k+1} \\
|D_{k+1}| + (a_{k+1} - r + 1) |D_k|, &\text{if } r < a_{k+1}.
\end{cases}
\]
If $s = q_k - 1$ then
\[
\max \{ d_{i+1} - d_i: 0 \le i \le m \}
=
\begin{cases}
|D_k|, &\text{if } r = a_{k+1} \\
|D_{k+1}| + (a_{k+1} - r) |D_k|, &\text{if } r < a_{k+1}.
\end{cases}
\]
\end{enumerate}
\end{thm}

\bigskip

Rational numbers also have continued fraction expansions, however they are finite, taking the form
\[
a_0 + \dfrac{1}{a_1 + \dfrac{1}{a_2 +
\raisebox{-2ex}{$\ddots$}\raisebox{-4ex}{$+\dfrac1{a_t}$}}} \quad.
\]

The partial quotients $a_0 \in \bZ$ and $a_1,\ldots,a_t \in \bN$, as well as the convergents $p_j/q_j$ ($0 \le j \le t$), are defined in the same way is in the irrational case, except that we impose the additional constraint $a_t > 1$ to be sure that the expansion is unique.

\bigskip

In the next subsection, we fix an irrational number
$\alpha$ and describe an expansion, the \emph{Ostrowski expansion}, that
allows us to accurately read off for each $n\in\mathbb{N}$ the size of
$\left\Vert n\alpha\right\Vert $, and more generally $\| n \alp - \gam \|$ for $\gam \in \bR$.

\subsection{Ostrowski expansions}

Let $n\in\mathbb{N}$, and let $K \geq 0$ be such that 
\begin{equation}
q_{K}\leq n<q_{K+1}.\label{eq: n sandwiched}
\end{equation}
It is known \cite[p. 24]{RS1992} that there exists a uniquely determined
sequence of non-negative integers $c_{k}=c_{k}(n)$ such that
\[
n=\sum_{k\geq0}c_{k+1}q_{k},
\]
satisfying 
$$c_{k+1}=0\,\, \mathrm{for}\,\, \mathrm{all}\, k>K,
$$
as well as the following additional constraints:
\begin{align*}
0\leq c_{1}<a_{1},  
\qquad 
0\leq c_{k+1}\leq a_{k+1} \quad (k \in \bN),
\qquad
\text{if } c_{k+1} = a_{k+1} \text{ then } c_k = 0.
\end{align*}
The structure of the set of integers whose initial Ostrowski digits are prescribed is well-understood; this set is referred to as a \emph{cylinder set}. We require information concerning the size of the gap between consecutive elements, which we retrieve from \cite[Lemma 5.1]{BHV2016}. 

\begin{lemma} [Gaps lemma]
\label{lem: gaps}  Let $m \ge 0$, and let
${\set A}\left(d_{1},\ldots,d_{m+1}\right)$
denote the set of positive integers whose initial Ostrowski
digits are $d_{1},\ldots,d_{m+1}$. Let $n_{1}<n_{2}<\ldots$ be
the elements of the set 
${\set A}\left(d_{1},\ldots,d_{m+1}\right)$, and let $i \in \bN$.
If $d_{m+1}>0$ then 
\[
n_{i+1}-n_{i}\geq q_{m+1},
\]
and if $d_{m+1}=0$ then $n_{i+1}-n_{i}\in\left\{ q_{m+1},q_{m}\right\} $.
Further, if $d_{m+1}=0$ and $n_{i+1}-n_{i}=q_{m}$ then $c_{m+2}\left(n_{i}\right)=a_{m+2}$
and the gap $n_{i+1}-n_{i}$ is preceded by $a_{m+2}$ gaps of size
$q_{m+1}$.
\end{lemma}

With reference to (\ref{def: Dk}), 
we will apply the theory of this subsection to pairs $(\alp,\gam)$, where
\[
\gamma=\sum_{k\geq0}b_{k+1}D_{k},
\]
such that
\begin{equation} \label{eq: DandyAndy}
a_{k}\geq 64, \quad
\frac{a_k}{4} \leq b_{k} \leq \frac{a_k}{2} \quad (k \ge 1) , \qquad a_0 = 0.
\end{equation}
Before proceeding in earnest, we confirm some technical conditions that will put us into a standard setting.

\begin{lemma} \label{lem: DandyAndy}
If we have \eqref{eq: DandyAndy}, then 
\[
0 < \alp < \frac1{64}, 
\qquad 0 \le \gam < 1 - \alp,
\]
and
\begin{equation}
\left\Vert n\alpha-\gamma\right\Vert \neq 0 \qquad (n\in \bN)
\label{eq: non-vanishing}.
\end{equation}
\end{lemma}

\begin{proof}
The first inequality follows from $a_0 = 0$ and $a_1 \ge 64.$

We compute that 
$$
\vert b_{k+1} D_k\vert 
\leq \frac{a_{k+1}}{q_{k+1}}
\leq \frac{1}{q_{k}}
,\quad
\vert b_{k+1} D_k\vert 
\geq \frac{a_{k+1}/4}{2q_{k+1}}
\geq \frac{1}{16 q_{k}}
\qquad (k \ge 0).
$$
As $a_{k+1} \ge 64$ for all $k \ge 0$, we see from these inequalities 
that $\vert b_{k+1} D_k\vert$ is a monotonically decreasing 
sequence which converges to zero as $k\rightarrow \infty$.
Using that the signs of the $D_k$ alternate and that
$D_{0} = \{\alp\} > 0$, we conclude that $ b_{k+1} D_k +  b_{k+2} D_{k+1} \geq 0$ if $k$ is even and
$ b_{k+1} D_k +  b_{k+2} D_{k+1}\leq 0$ if $k$ is odd.
Therefore
$$
0 
\leq 
\sum_{k\geq 0} (b_{2k+1} D_{2k} +  b_{2k+2} D_{2k+1})
= \gam
$$
supplies the lower bound in the second inequality. For the upper bound, note that
$$
\gam = 
b_1 D_{0} +
\sum_{k\geq 0} (b_{2k+2} D_{2k+1}+ b_{2k+3} D_{2k+2})
\leq b_1 D_{0} \leq \frac{b_1}{a_1} \leq 1/2 < 1-\alp.
$$

Finally, we turn our attention towards \eqref{eq: non-vanishing}. Observe that
\[
n \alp - \gam = \sum_{k \ge 0} (c_{k+1} q_k \alp - b_{k+1} (q_k \alp - p_k)) \in \Sigma + \bZ,
\]
where 
\begin{equation} \label{SigDef}
\Sigma=\sum_{k\geq0}\delta_{k+1}D_{k},\qquad\delta_{k+1}=c_{k+1}-b_{k+1}\quad (k\geq0).
\end{equation}
Using \eqref{eq: standard estimate} and \eqref{eq: continued fraction recursion}, as well as \eqref{eq: DandyAndy}, we compute that
\[
|\Sigma| \ge |D_0| - \frac34 \sum_{k \ge 1} a_{k+1} |D_k| \ge \{ \alp \} - \frac34 \sum_{k \ge 1} \frac1{q_k},
\]
where $\{\alp\}$ denotes the fractional part of $\alp$, and
\[
\sum_{k \ge 1} \frac1{q_k}
\le \frac1{q_1} \sum_{r \ge 0} 64^{-r} = \frac{64}{63q_1}.
\]
Since
\[
\frac{1} {\{ \alp \}} < a_1 + 1 = q_1 + 1 \le \frac{65}{64} q_1,
\]
we have 
\[
|\Sig| > \left( \frac{64}{65} - \frac{16}{21} \right) q_1^{-1} > 0.
\]
Moreover
\[
|\Sigma| \le |D_0| + \sum_{k \ge 1} a_{k+1} |D_k| 
\le \{ \alp \} +  \sum_{k \ge 1} \frac1{q_k} \le \frac1{64} + \frac1{63} < 1.
\]
Verily we have \eqref{eq: non-vanishing}.
\end{proof}

It turns out that we can quantify
the size of $\left\Vert n\alpha-\gamma\right\Vert $ in terms of the
quantity $\Sigma$ from \eqref{SigDef},
as the next lemma details. The assumption \eqref{eq: DandyAndy} simplifies several technicalities. 
The following combines \cite[Lemmata 4.3, 4.4, and 4.5]{BHV2016}, in this special case.

\begin{lemma}
\label{lem: size of inhomogeneous distance}If we have \eqref{eq: DandyAndy},
then 
\[
\left\Vert n\alpha-\gamma\right\Vert =\left\Vert \Sigma\right\Vert =\min(\left|\Sigma\right|,1-\left|\Sigma\right|).
\]
Furthermore, let $m=m(n)$ be the smallest $i\geq 0$ such that $\delta_{i+1}\neq0$,
and let $K$ be as in (\ref{eq: n sandwiched}). Then we have the following estimates for
 $\left|\Sigma\right|$ and $1-\left|\Sigma\right|$.

\begin{enumerate} [(1)]
\item 
\[
\left|\Sigma\right|=(\left|\delta_{m+1}\right|-1)\left|D_{m}\right|+u_{m+2}\left|D_{m+1}\right|+u_{m+3}\left|D_{m+2}\right|+\varUpsilon,
\]
where $u_{m+2},u_{m+3},\varUpsilon$ are non-negative real numbers
constrained by 
\[
u_{m+2}\asymp a_{m+2},\qquad u_{m+3}\asymp a_{m+3},\qquad\varUpsilon\ll |D_{m+2}|.
\]
\item 
\[
1-\left|\Sigma\right|=u_{1}\left|D_{0}\right|+u_{2}\left|D_{1}\right|+\tilde{\varUpsilon},
\]
where $u_{1},u_{2},\tilde{\varUpsilon}$ are non-negative, and constrained
by 
\[
u_{1}\asymp a_{1},\qquad u_{2}\asymp a_{2},\qquad
\tilde \varUpsilon\ll |D_{1}|.
\]
\end{enumerate}
\end{lemma}

\subsection{Bohr sets} 

For $\balp, \bgam \in \bR^{k-1}$, we have ample need for bounds on the cardinality of Bohr sets
\begin{equation} \label{BohrDef}
{\set B} = {\set B}_\balp^\bgam (N; \brho) := \{ n \in \bZ: |n| \le N, \| n \alp_i - \gam_i \| \le \rho_i \quad (1 \le i \le k-1) \},
\end{equation}
that are sharp up to multiplication by absolute constants.
Further, it turns out to be crucial 
to have precise control over the ranges in which the Bohr sets have a sufficiently nice structure. If the width parameters $\delta_i$ and the length parameters $N_i$
are in a suitable regime, then the Bohr set $\cB$
is enveloped---efficiently, as we detail soon---in a {\em $k$-dimensional generalised arithmetic progression}
\begin{equation}\label{def: generalised arithmetic progression}
{\set P}(b; A_1, \ldots, A_k; N_1, \ldots, N_k) = 
\{ b + A_1 n_1 + \cdots + A_k n_k: |n_i| \le N_i \},
\end{equation}
where $b, A_1, \ldots, A_k, N_1, \ldots, N_k \in \bN$. 
The thresholds for the admissible regimes depend
naturally on the diophantine properties of $\balp$.

Inside ${\set B}$, 
we will find a large (asymmetric) 
generalised arithmetic progression
\begin{equation} \label{def: Pplus}
{\set P}^+(b; A_1, \ldots, A_k; N_1, \ldots, N_k) = 
\{ b + A_1 n_1 + \cdots + A_k n_k: 1 \le n_i \le N_i \}.
\end{equation}
This is \emph{proper} 
if for each $n \in {\set P}^+(b, A_1, \ldots, A_k, N_1, \ldots, N_k)$ there is a unique vector $(n_1, \ldots, n_k) \in \bN^k$ 
for which
\[
n = b + A_1 n_1 + \cdots + A_k n_k, \qquad n_i \le N_i \quad (1 \le i \le k).
\]
Throughout this subsection we operate under the assumption 
\eqref{eq: small mult. exponent}, which in the case $k=2$
simply means that $\alpha$ is irrational and non-Liouville.
Set
$$
\eta(\bfalp) = \frac1{\ome^\times(\balp)} \: - \: \frac{k-2}{k-1} \in (0,1].
$$
In \cite[Section 3]{CT2019}, we exploited the strict positivity of $\eta(\balp)$ to describe regimes in which the Bohr sets \eqref{BohrDef}
contain and are contained in generalised arithmetic progressions 
of the expected size. We presently outline the key findings from that investigation, as far as they are needed here.

\begin{lemma} [Inner structure] \label{lem: inner} 
Let $\vartheta \geq 1$. 
Then there exists $\tilde{\eps} > 0$
with the following property.
If $\eps \in (0, \tilde{\eps}]$ is fixed, and $N$ is large in terms of $\vartheta, \eps, \balp$, and
\begin{equation} \label{eq: regime}
N^{-\eps} \le \rho_i \le 1 \quad (1 \le i \le k-1),
\end{equation}
then there exists a proper generalised arithmetic progression 
\[
{\set P} = {\set P}^+(b; A_1, \ldots, A_k; N_1, \ldots, N_k) 
\]
contained in ${\set  B}$, for which 
\[
|{\set P}| \gg \rho_1 \cdots \rho_{k-1} N, 
\qquad \min_{i\le k} N_i \ge N^{\vartheta\eps}, 
\qquad N^{\sqrt \eps} \le b \le \frac N{10},
\]
and
\begin{equation*}
\gcd(A_1,\ldots,A_k) = 1.
\end{equation*}
\end{lemma}

\begin{proof}
Observe that the statement gets weaker as $\eps$ decreases, in the sense that if it holds for $\eps = \tilde \eps$ then it holds for any $\eps \in (0 , \tilde \eps]$. It is almost identical to the statement in \cite[Lemma 3.1]{CT2019}, but there the variable $\vartheta$ was equal to $1$. By replacing $\tilde \eps$ by $\tilde \eps/\vartheta$, thereby shrinking the range of admissible values of $\eps,$ we obtain the statement here.
\end{proof}

Next, we quantify the range of the width vector $\brho$ for which the Bohr sets are efficiently contained in generalised arithmetic progressions. The outer structure lemma as stated in \cite[Lemma 3.2]{CT2019} 
is homogeneous and fails to record the feature 
that $\gcd(A_1,\ldots,A_k)=1$. Here we require an inhomogeneous version as well as the latter feature, and fortunately we can extract this 
additional information from the proof of \cite[Lemma 3.2]{CT2019}, as we now explain.

\begin{lemma} [Outer structure] \label{lem: outer} 
Let $\vartheta \geq 1$. 
Then there exists $\tilde{\eps} > 0$
with the following property. 
If $\eps \in (0, \tilde{\eps}]$ is fixed and $N$ is
sufficiently large in terms of $\eps, \vartheta$,
and we have \eqref{eq: regime}, then there exists a generalised arithmetic progression
\[
{\set P} = {\set P}(b; A_1, \ldots, A_k; N_1, \ldots, N_k),
\]
containing $B_\balp^\bgam(N;\brho)$,
for which
\[
\min_{i\le k} N_i \ge N^{\vartheta \eps}, 
\qquad |{\set P}| \ll N_1 \cdots N_k \ll \rho_1 \cdots \rho_{k-1}N
\]
and
\[
\gcd(A_1,\ldots,A_k)=1.
\]
\end{lemma}

\begin{proof}
Lemma \ref{lem: inner} implies that 
the Bohr set ${\set B}_\balp^\bgam(N;\brho)$ 
is non-empty, so choose 
$b \in {\set B}_\balp^\bgam(N;\brho)$ arbitrarily.
For any $n\in {\set B}_\balp^\bgam(N;\brho)$,
the triangle inequality yields 
$n-b \in {\set B}_\balp^\bzero(N;2\brho)$.
Now \cite[Lemma 3.2]{CT2019} assures us that
\[
n-b \in {\set P}(0; A_1, \ldots, A_k; N_1, \ldots, N_k).
\]
The coprimality property $\gcd(A_1,\ldots,A_k)=1$ 
comes out of the proof, but was not recorded in
\cite[Lemma 3.2]{CT2019} because it was not needed there. Similarly, the inequality $\min_{i \le k}N_i \ge N^\eps$ arises in the proof. Moreover, it is stated in \cite[Lemma 3.2]{CT2019} that $|\cP| \ll \rho_1 \cdots \rho_{k-1} N$, but its proof contains the more refined inequalities
\[
|{\set P}| \ll N_1 \cdots N_k \ll \rho_1 \cdots \rho_{k-1}N.
\]
By replacing $\tilde \eps$ by $\tilde \eps/\vartheta$, thereby shrinking the range of admissible values of $\eps$, we are able to bootstrap the inequality to $\min_{i \le k}N_i \ge N^{\vartheta \eps}$.
\end{proof}

Finally, recall that Beresnevich, Haynes, and Velani \cite[Lemma 6.1]{BHV2016} determined a convenient condition under which 
a rank $1$ Bohr set has the expected size.
\begin{lemma}\label{lem: BHV size bound for Bohr sets}
Let $\alpha \in \bR \setminus \bQ$, let $(q_k)_k$
be its sequence of denominators of convergents of continued fractions, 
and let
$ \del \in (0,\Vert q_2 \alpha \Vert /2 )$.
If there exists $\ell \in \bZ$ such that 
$$
\frac{1}{2\del} \leq q_\ell \leq N,
$$
then 
\[
\del N -1 \leq 
\# {\set B}_{\alpha}^{0}(N; \del) \leq 32 \del N.
\]
\end{lemma}

\subsection{Measure theory}

To bound from below the measure of the limit superior sets of interest, we deploy the `divergence' Borel--Cantelli lemma \cite[Lemma 2.3]{Har1998}.

\begin{lemma}\label{lem: Borel-Cantelli}
Let ${\set E}_1, {\set E}_2, \ldots$ be a sequence of Borel subsets of $[0,1]$ such that
\[
\sum_{n=1}^\infty \mu({\set E}_n) = \infty,
\]
and let ${\set E} = \displaystyle \limsup_{n \to \infty} 
{\set E}_n$. Then 
\begin{equation*}
\mu ({\set E}) \geq \limsup_{N \rightarrow \infty} 
\frac{\left(\displaystyle \sum_{n\leq N} \mu 
({\set E}_n)\right)^2}
{\displaystyle \sum_{n, m \leq N} \mu \left
({\set E}_n \cap {\set E}_m \right)}.
\end{equation*}
\end{lemma}

In many applications in metric diophantine approximation, 
it suffices to establish that a limit superior set ${\set E}$
has positive measure in order to conclude that 
it has full measure: Often one of the classical
zero--one laws, such as Cassels' or Gallagher's 
\cite[Section 2.2]{Har1998},
rules out the possibility that $\mu({\set E})\in(0,1)$.  
However, for our purposes no such zero--one law is available.
To establish full measure, we turn instead to another device, which follows from Lebesgue's density theorem.
Intuitively, the next lemma---sometimes referred to as Knopp's lemma---states that any set whose local densities
are uniformly and positively bounded from below must have full measure. This is a special case of \cite[Proposition 1]{BDV2006}.

\begin{lemma}\label{lem: density}
If ${\set E} \subseteq [0,1]$ is Borel set and
$$
\mu({\set E} \cap {\set I}) \gg \mu ({\set I}) 
$$
for any interval ${\set I}\subseteq [0,1]$, then 
$\mu({\set E})=1$.
\end{lemma}

\subsection{Real analysis}

The next lemma will later enable us to eschew 
a certain small-GCD regime.
In the first instance, it asserts that if 
$\psi(n)\log n$ is not summable then neither is 
$\psi(an)\log n$ for any fixed $a\in\bN$. 
In fact, we can allow $a$ to increase very slowly with $n$.

\begin{lemma}\label{lem: Fourier}
Let $d\in \bN$. Let 
$\psi: \bN \rightarrow \bR_{>0}$ be non-increasing such that 
$$
\sum_{n\geq 1} \psi(n) (\log n)^d
$$
diverges. Then, there exists a strictly increasing sequence $(K_{i})_{i}$
of positive integers satisfying
$ K_i\geq  \exp(\exp i)$, for all $i\geq 2$, such that $K_1 = 1$, $K_i = o(K_{i+1})$, and:
\begin{enumerate}
\item The map $f$ defined by 
$$
f(n)=i \qquad (K_{i}\leq n<K_{i+1})
$$ satisfies $f(n)\ll \log \log n$.
\item If $\hat{\psi}(n)=\psi(4^{f(n)}n)$, for all $n$,
then the series 
$$ \sum_{n\geq1}\hat{\psi}(n)(\log n)^d
$$ diverges.
\end{enumerate}
\end{lemma}

\begin{proof}
We construct $(K_{i})_{i}$ recursively. Write
\[
S_{\psi}(N)=\sum_{n\leq N}\psi(n) (\log n)^d,
\qquad
S_{\hat \psi}(N)=\sum_{n\leq N}\hat \psi(n) (\log n)^d.
\]
Set $K_1 = 1$, let $i \ge 1$, and suppose that $K_{i}$ has already been constructed.
If $a \in \bN$ and $N \ge N_a$, where $N_a = N(a,\psi)$ is large, then
\begin{align*}
S_{\psi}(aN) &=
S_\psi(a^2) + \sum_{a \leq j< N}\sum_{r \le a}
\psi\left(aj+r\right) (\log\left(aj+r\right))^d
\\
 & \ll S_\psi(a^2) + \sum_{a \leq j< N}\sum_{r \le a}\psi\left(aj\right)(\log j)^d
 \ll a\sum_{j\leq N}\psi\left(aj\right)(\log j)^d.
\end{align*}
Next, consider $a = 4^i$ and $N \ge K_i + N_a$, for some $i \in \bN$. Supposing we choose $K_{i+1} = N+1$, then as $i \ge f(j)$ for any $j \le N$, we have
\[
\psi(aj) = \psi(4^i j) \le \psi(4^{f(j)}j) = \hat \psi(j) \qquad (j \le N),
\]
and so
\[
S_{\hat \psi}(N) \gg a^{-1} S_\psi(aN) = 4^{-i} S_\psi(4^i N).
\]
We define $K_{i+1}$ to be one more than the smallest positive integer 
\[
N \geq \exp (\exp (2 K_{i} + N_{4^{i}}))
\]
for which $4^{-i} S_\psi(4^i N) \ge i$. 
By construction, we have 
\[
S_{\hat{\psi}}(K_{i+1}) \gg i,
\qquad K_{i+1} \geq \exp( \exp (i+1)).
\]
The latter implies that $f(n) \leq \log \log n$
for $n \ge K_2$, completing the proof.
\end{proof}

The following lemma helps us with dyadic pigeonholing.

\begin{lemma}\label{lem: truncation}
Let $h:\bN \rightarrow \bR_{>0}$ be non-increasing, and fix $C\geq 2$, $\kap > 0$, as well a positive integer $J_0 \ll 1$.
Then, for $N \in \bN$ and $J = \lfloor \log N / \log C \rfloor$, we have
\begin{equation}\label{eq: doubling on average}
\sum_{C^{J_0} \le n \le N} h(n) (\log n)^{\kappa}
\asymp \sum_{j=J_0}^{J} j^{\kappa}C^{j} h(C^j).
\end{equation}
\end{lemma}

\begin{proof}
Since $h$ is non-increasing, for $j =1,2,\ldots,J$ we have
\begin{align*}
\sum_{C^{j} < n \le C^{j+1}} h(n) (\log n)^\kappa
&\ge C^{j+1}  (1 - 1/2) h(C^{j+1}) (j \log C)^\kappa \\
&\gg C^{j+1} h(C^{j+1})
((j+1) \log C)^\kappa
\end{align*}
and
\begin{align*}
\sum_{C^{j} \leq n \le C^{j+1}} h(n) (\log n)^\kappa \le C^{j+1} h(C^{j}) ((j+1) \log C)^\kappa
\ll C^j h(C^j) (j \log C)^\kap.
\end{align*}
Therefore
\begin{align*}
\sum_{C^{J_0} \le n \le N} h(n) (\log n)^{\kappa}
&\gg h(C^{J_0}) (J_0 \log C)^\kap +  \sum_{j=J_0}^{J-1} C^{j+1} h(C^{j+1}) ( (j+1) \log C)^\kap \\
&\gg \sum_{j=J_0}^{J} C^{j} h(C^{j}) ( j \log C)^\kap
\end{align*}
and
\begin{align*}
\sum_{C^{J_0} \le n \le N} h(n) (\log n)^{\kappa}
&\ll \sum_{j=J_0}^{J}
C^{j} h(C^{j}) ( j \log C)^\kap.
\end{align*}
\end{proof}

\subsection{Geometry of numbers}

The following lattice point counting theorem originates from the work of Davenport \cite{D1951}, see also \cite{BW2014} and \cite[p. 244]{T1993}. 
Our precise statement follows from \cite[Lemmata 2.1 and 2.2]{BW2014}.

\begin{thm} [Davenport]\label{DavenportBound}

Let $d,h \in \bN$, and let ${\set S}$ be a compact subset of $\bR^{d}$.
Assume that the two following conditions are met:

\begin{enumerate}[(i)]
\item Any line intersects ${\set S}$ in a set of points which, if non-empty,
comprises at most $h$ intervals.
\item The first condition holds, with $j$ in place of $d$, for 
the projection of ${\set S}$ onto any $j$-dimensional subspace. 
\end{enumerate}

Let $\lam_1 \le \cdots \le \lam_d$ be the successive minima, 
with respect to the Euclidean unit ball, 
of a (full-rank) lattice $\Lambda$ in $\bR^{d}$. Then
\[
\biggl\vert \vert {\set S} \cap \Lambda \vert 
- \frac{\mathrm{vol (\cS)}}{\det \: \Lambda}
\biggr\vert 
\ll_{d,h} \sum_{j=0}^{d-1} \frac{V_{j}({\set S})}{\lam_1 \cdots \lam_{j}},
\]
where $V_{j}({\set S})$ is the supremum of the $j$-dimensional volumes of the projections 
of $ {\set S}$ onto any $j$-dimensional subspace.
We adopt the convention that $V_{0}({\set S})=1$.
\end{thm}

We have not encountered a reference for the following classical result, so we provide a proof.

\begin{lemma} \label{detd}
Let $k \in \bN$, and let $A_1, \ldots, A_k, d \in \bN$ with 
\[
\gcd(A_1,\ldots,A_k,d) = 1.
\]
Then the congruence
\begin{equation} \label{eq: LatticeCongruence}
A_1 n_1 + \cdots + A_k n_k \equiv 0 \mmod d
\end{equation}
defines a full-rank lattice $\Lam$ of determinant $d$.
\end{lemma}

\begin{proof} As $\gcd(A_1,\ldots,A_k,d) = 1$, there exist integers $m_1,\ldots,m_{k+1}$ such that
\[
A_1 m_1 + \cdots + A_k m_k + d m_{k+1} = 1,
\]
and particular
\[
A_1 m_1 + \cdots + A_k m_k \equiv 1 \mmod d.
\]
As $1$ generates the cyclic group $\bZ/d\bZ$, it follows that the homomorphism
\begin{align*}
\bZ^k &\to \bZ/d\bZ \\ (n_1,\ldots,n_k)  &\mapsto A_1 n_1 + \cdots + A_k n_k
\end{align*}
is surjective. Further, its kernel is $\Lambda \subseteq \bZ^k$. By the first isomorphism theorem, the group $\bZ^k/\Lam$ has finite order $d$, and consequently $\Lam$ has full rank. An application of \cite[Subsection I.2.2, Lemma 1]{Cas1997} completes the proof.
\end{proof}

We need to be able to count elements of a generalised arithmetic progression divisible by a given positive integer $d$. The previous two facts enable us to accurately do so, provided that $d$ is not too large.

\begin{lemma}\label{lem: divisible numbers in GAPS}
\:
\begin{enumerate}[(i)]
    \item 
Let $d \in \bN$, and let $ {\set P}$
be generalised arithmetic progression given by \eqref{def: generalised arithmetic progression}, where $\gcd(A_1,\ldots,A_k)=1$.
Then 
\begin{equation}\label{eq: divisibilty in Bohr sets}
\frac{\# \{ n\in {\set P}: d \mid n\}}{N_1\cdots N_k} \ll d^{-1} + (\min_{i \le k} N_i)^{-1}.
\end{equation}
\item
Let ${\set P}$ be a proper, asymmetric 
generalised arithmetic progression 
given by \eqref{def: Pplus}, 
where $\gcd(A_1,\ldots,A_k)=1$, and let $d \in \bN$. Then
\begin{equation} \label{eq: DivisibilityInBohrSets2}
\frac{\# \{ n\in {\set P}: d \mid n\}}{N_1\cdots N_k} 
= d^{-1} + O(1/\min_{i \le k} N_i).
\end{equation}
\end{enumerate}
\end{lemma}

\begin{proof}
\begin{enumerate}[(i)]
    \item The quantity $\# \{ n\in {\set P}: d \mid n\}$ is bounded above by the number of integer solutions 
    \[
    (n_1,\ldots,n_k) \in [-N_1,N_1] \times \cdots \times [-N_k, N_k]
    \] to
\[
b + A_1 n_1 + \cdots + A_k n_k \equiv 0 \mmod d.
\]
We may assume that this has a solution
\[
(n_1^*, \ldots, n_k^*) \in [-N_1,N_1] \times \cdots \times [-N_k, N_k],
\]
and then
\[
(n'_1, \ldots, n'_k)
= (n_1,\ldots,n_k)
- (n_1^*, \ldots, n_k^*)
\]
lies in $[-2N_1,2N_1] \times \cdots \times [-2N_k, 2N_k]$ and satisfies
\[
A_1 n'_1 + \cdots + A_k n'_k \equiv 0 \mmod d.
\]
By Lemma \ref{detd}, this defines a full-rank lattice of determinant $d$, and this is a sublattice of $\bZ^n$ so the successive minima are greater than or equal to $1$. By Theorem \ref{DavenportBound}, we now have
\[
\# \{ n\in {\set P}: d \mid n\}
\le  \frac{2^k N_1 \cdots N_k}{d} 
+ O_k(N_1 \cdots N_k / \min_{i \le k} N_i),
\]
giving \eqref{eq: divisibilty in Bohr sets}.
\item As ${\set P}$ is proper, 
the quantity $\# \{ n\in {\set P}: d \mid n\}$ 
counts integer solutions 
    \[
    (n_1,\ldots,n_k) \in [1,N_1] \times \cdots \times [1, N_k]
    \] to
\[
b + A_1 n_1 + \cdots + A_k n_k \equiv 0 \mmod d.
\]
Since $\gcd(A_1,\ldots,A_k) = 1$, there exist integers $n^*_1, \ldots, n^*_k$ such that
\[
b + A_1 n^*_1 + \cdots + A_k n^*_k = 0.
\]
Now
\[
(n'_1, \ldots, n'_k)
= (n_1,\ldots,n_k)
- (n_1^*, \ldots, n_k^*)
\]
lies in $[1-n^*_1, N_1-n^*_1] \times \cdots \times [1-n^*_1, N_k - n^*_1]$ and satisfies
\[
A_1 n'_1 + \cdots + A_k n'_k \equiv 0 \mmod d.
\]
By Lemma \ref{detd}, this defines a full-rank lattice of determinant $d$, and this is a sublattice of $\bZ^n$ so the successive minima are greater than or equal to $1$. By Theorem \ref{DavenportBound}, we now have
\[
\# \{ n\in {\set P}: d \mid n\}
= \frac{(N_1-1) \cdots (N_k-1)}{d} + O_k(N_1 \cdots N_k / \min_{i \le k} N_i),
\]
giving \eqref{eq: DivisibilityInBohrSets2}.
\end{enumerate}
\end{proof}

\subsection{Primes and sieves}

We require one of Mertens' three famous, classical estimates \cite[Theorem 3.4(c)]{DimitrisBook}.

\begin{thm}[Mertens' third theorem]\label{thm: Mertens}
For $x \ge 2$, we have
\[
\prod_{p \le x} 
\left( 1- \frac1p \right)
= \frac{e^{-\gam}}{\log x} \left(1 + O \left( \frac1{\log x} \right) \right),
\]
where $\gam$ is the Euler--Mascheroni constant.
\end{thm}

\bigskip

Sieve theory is a powerful collection of techniques used to study prime numbers. Let $\cP$ be set of primes, and let 
\[
\cA = (a_n)_{1 \le n \le x}
\]
be a finite sequence of non-negative real numbers. The main object of interest is the \emph{sifting function}
\[
S({\set A}, z) = \sum_{(n,P(z))=1}a_n,
\]
where
\[
P(z)= \prod_{\substack{p\in {\set P} \\ p < z}} p,
\]
for a parameter $z \ge 2$. For example, if $\cJ \subseteq [1,x] \cap \bZ$ and $a_n = 1$ for $n \in \cJ$ and $a_n = 0$ for $n \notin \cJ$, then $S(\cA, z)$ counts elements of $\cJ$ not divisible by any prime $p \in \cP$ for which $p < z$.

In this subsection only, the letter $\mu$ denotes the M\"obius function, and not a Lebesgue measure. Evaluating
the sifting function
using the inclusion--exclusion principle yields
\[
S(\cA, z) = \sum_{d \mid P(z)} \mu(d) A_d(x),
\]
where
\[
A_d(x) = \sum_{\substack{n \le x \\ n \equiv 0 \mmod d}} a_n \qquad (d \in \bN).
\]
This gives
rise to 
the main term $XV(z)$,
where
\begin{itemize}
\item $X$ is typically chosen to approximate $A_1(x)$
\item The \emph{density function}, $g$, is a multiplicative arithmetic function satisfying
\[
0 \le g(p) < 1 \qquad (p \in \cP),
\]
and $g(p)$ is typically chosen to approximate $A_p(x)/X$
\item
\[
V(z) = \prod_{\substack{p\in {\set P} \\ p < z }} (1-g(p)).
\]
\end{itemize}

However, if $z$ is large then $P(z)$ could have many prime factors, and the error terms may combine to overwhelm the main term. This impasse stood for a long time before Viggo Brun was able to overcome it in many situations.  Brun's idea was to approximate 
$\mu$ by a function of smaller support, thereby reducing the number of error terms. One particularly useful outcome is the fundamental lemma, which involves the following additional objects:
\begin{itemize}
\item Remainders
\[
r_d = A_d(x) - g(d) X \qquad (d \in \bN)
\]
\item The \emph{dimension} $\kap \ge 0$, typically chosen to approximate a suitable average of $g(p)p$ over $p \in \cP$
\item The \emph{level} $D > z$,
and the \emph{sifting variable}
$s = \log D/ \log z$.
\end{itemize}
The result, for which the lower bound is stated below, is taken from Opera de Cribro \cite[Theorem 6.9]{opera}. 
In principle there is a great deal of flexibility, however in practice there is often a natural choice of parameters that reflects the nature of the problem, and any less principled choice tends to produce weaker information. 

\begin{thm}[Fundamental lemma of sieve theory] \label{thm: sieve bound}
Let $\kappa \geq 0$, $z\geq 2$, $D\geq z^{9\kappa +1}$,
$K > 1$, and assume that 
\begin{equation} \label{DimensionCondition}
    \prod_{w\leq p < z} (1-g(p))^{-1} 
    \leq K \Big(\frac{\log z}{\log w} \Big)^\kappa 
    \qquad (2 \leq w < z).
\end{equation}
Then 
$$
S({\set A}, z) \geq X V(z)(1-e^{9\kappa - s}K^{10})
- \sum_{\substack{d\mid P(z)\\ d< D}} |r_{d}|.
$$
\end{thm}

\addtocontents{toc}{\protect\setcounter{tocdepth}{2}}

\section{A fully-inhomogeneous version of Gallagher's theorem}\label{Sec: Inhomog. Gallagher}

In this section, we prove Theorem \ref{thm2}, along the way establishing Theorem \ref{thm: SpecialCaseWeak}. At the end we prove the convergence part of Corollary \ref{FIG}. The reasoning presented here is sensitive to the diophantine nature of the shift $\gamma_k$. We begin by introducing some notation, and reducing the statements of Theorems \ref{thm2} and \ref{thm: SpecialCaseWeak}
to proving statistical properties of certain sets. 

\subsection{Notation and reduction steps}

For ease of exposition, we use the abbreviations
$$
\alpha=\alpha_k, \qquad \gamma=\gamma_k
$$ throughout the present section. 
With $f$ as in Lemma \ref{lem: Fourier}, specialising $d=k-1$ therein, we set
\begin{equation*}
\hat{\psi}(n) = \psi(\un),
\quad \mathrm{where} \quad \un = 4^{f(n)}n.
\end{equation*}
In light of \eqref{eq: sum of measure diverges}, we then have
\begin{equation} \label{eq: HatDivergence}
\sum_{n=1}^\infty \psi(\un) (\log n)^{k-1} = \infty.
\end{equation}
For $\eps_0 > 0$, we define
\begin{align*}
\hat  {\set G}_\diam &=  \{ h \in \bN: h^{-4 \eps_0} 
\le \| h \alp_i - \gam_i \| \le h^{-2 \eps_0}\quad (1\leq i\leq k-1)\}, \\
{\set G}_\diam &= \{ n \in \bN: \hat n \in \hat {\set G}_\diam \},
\end{align*}
and
\[
{\set G} = \left \{ n \in {\set G}_\diam: \psi(\hat n) \ge \frac1{n(\log n)^{k+1}}\right \}.
\]

\begin{remark}
\label{epsSmall}
The constant $\eps_0$ is sufficiently small depending on the diophantine nature of the fibre vector $\balp$ and the final shift $\gam$. Since, throughout this section,
we operate under the assumption 
\eqref{eq: small mult. exponent} for $k\geq 3$
and for $k=2$ that $\alp \notin (\cL \cup \bQ)$,
the constant $\eps_0$ will always be small enough so that 
the structural theory of Bohr sets applies. In particular, we will have $\eps_0 \le \tilde \eps$ when
Lemmata \ref{lem: inner} and \ref{lem: outer} are applied with $\vartheta = 20k$. We will also always assume that $\eps_0 < (99k)^{-1}$.
\end{remark}

Let us now also introduce a parameter $\eta = \eta(\gam) \in (0,1)$.
We shall in due course be more specific about $\eps_0$ and $\eta$, if $\gam$ is diophantine, rational, or Liouville. 
The reader seeking these details instantly may consult 
\eqref{def: eps}, \eqref{def: approximation sets rational case},
and \eqref{def: choice of eta}.

For $n \in \bN$, let
\[
\Psi(n) = \frac{ \psi(\hat n)}
{ \| \hat n \alp_1 - \gam_1 \| \cdots \| \hat n \alp_{k-1} - \gam_{k-1} \| } 1_{{\set G}}(n),
\]
where $1_{{\set G}}$ is the indicator function of 
${\set G}$. Fix a non-empty interval
${\set I}\subseteq [0,1]$ and $\gam \in \bR$, and for $n \in \bN$ let
\begin{equation}\label{def: approximation sets for diophantine shifts}
{\set E}_n^{{\set I}, \gam} := 
\displaystyle
\left \{ \alpha \in [0,1]: 
\exists a\in \bZ\,
\mathrm{s.t.} \quad
\displaystyle \substack{
\displaystyle
a+\gamma \in \hat n {\set I}, \\
\displaystyle
\vert \hat n \alpha - \gamma - a \vert < \Psi(n), \\
\displaystyle
(a,\un) \text{ is } (\gam,\eta)\text{-shift-reduced}} \right \}.
\end{equation}
We call these
\emph{(localised) approximation sets}.
Note that if $n$ is large in terms of $\cI$ then
\begin{equation} \label{eq: BoundMeasureOfApprox}
\mu(\cE_n^{\cI,\gam}) 
\ll \mu(\cI) \Psi(n).
\end{equation}
Observe that
$$
{\set G} = \{ n \in \bN: \mu({\set E}_n) >0 \}.
$$
We will often suppress the dependence on $\gam$ and ${\set I}$ in the notation by writing 
${\set E}_n$ in place of ${\set E}_n^{{\set I}, \gam}$.

We say that $\gam$ is {\em admissible}
if there exists $\eta \in (0,1)$ such that for any interval ${\set I} \subseteq [0,1]$
there are infinitely positive integers $X$ for which
the following two properties hold:
\begin{enumerate}
\item We have
\begin{equation}\label{eq: divergence of approx. localised sets}
\sum_{n \le X} \mu ({\set E}_{n}) 
\asymp \mu({\set I}) \sum_{n \le X} \psi(\un) (\log n)^{k-1}.
\end{equation}
\item The sets ${\set E}_n$ are quasi-independent on average 
for those $X$, i.e.
\begin{equation}\label{eq: indepen of localised sets}
\sum_{\substack{m,n \leq X}}  
\mu({\set E}_{n} \cap {\set E}_{m}) 
\ll \mu({\set I}) \Biggl(
\sum_{n \le X} \psi (\un) 
(\log n)^{k-1} \Biggl)^2.
\end{equation}
\end{enumerate} 

By \eqref{eq: HatDivergence}, the right hand side of \eqref{eq: divergence of approx. localised sets} is unbounded as a function of $X$.
It is worth stressing that
the implicit constants in 
\eqref{eq: divergence of approx. localised sets} and \eqref{eq: indepen of localised sets} 
are only allowed to depend on $\balp,\gam_1,\ldots,\gam_k$, and are uniform in ${\set I}$.

Now we show that Theorems \ref{thm: SpecialCaseWeak} and \ref{thm2} 
can be reduced to showing that every $\gam \in \bR$ is admissible. We assume throughout this section that \begin{equation} \label{eq: PsiSmall}
\Psi(n) \le 1/2 \qquad (n \text{ large}),
\end{equation}
as we may because otherwise these two theorems are trivial.

\begin{remark}
Unless otherwise specified, `(sufficiently) large' 
means large 
in terms of $\balp, \gam_1,\ldots,\gam_k, \psi, {\set I}$. Similarly, unless otherwise specified, a positive real number is `(sufficiently) small' if it is small in terms of $\balp, \gam_1,\ldots,\gam_k, \psi, {\set I}$.
\end{remark}

\begin{prop} \label{prop: pot}
If every $\gam \in \bR$ is admissible,
then Theorems \ref{thm: SpecialCaseWeak} and \ref{thm2} are true.
\end{prop}

\begin{proof}
Let $\gam \in \bR$, and let $\eta$ be as in the definition of admissibility. We first show that
$
\cW(\Psi;\gam,\eta)
:=  \limsup_{n \to \infty} \cA_n
$
has full measure, where
\[
\cA_n = 
\left \{ \alpha \in [0,1]: 
\exists a\in \bZ\,
\mathrm{s.t.} \quad
\displaystyle \substack{
\displaystyle
\vert \hat n \alpha - \gamma - a \vert < \Psi(n), \\
\displaystyle
(a,\un) \text{ is } (\gam,\eta)\text{-shift-reduced}} \right \}.
\]
Fix a non-empty subinterval ${\set I}'$ of $[0,1]$, and let ${\set I}$ 
be its dilation by $1/2$ about its centre.
Observe using the triangle inequality that 
if $n \ge n_0({\set I})$ then 
$
{\set E}_n^{\cI, \gam} \subseteq \cI'.
$
Let
$
{\set R}^{\cI, \gam} = \displaystyle \limsup_{n \to \infty} 
{\set E}_n^{\cI, \gam}.
$
Inserting \eqref{eq: divergence of approx. localised sets} 
and \eqref{eq: indepen of localised sets} 
into Lemma \ref{lem: Borel-Cantelli},
we obtain
\[
\mu(\cW(\Psi;\gam,\eta) \cap \cI') \ge \mu({\set R}^{\cI, \gam}) \gg \mu({\set I}) \gg \mu(\cI'),
\]
where the implied constant is independent of ${\set I'}$.
Lemma \ref{lem: density} 
now grants us that $\cW(\Psi;\gam,\eta)$
has full measure in $[0,1]$.
As $\Psi(n) \le \Phi(\hat n)$ for all $n$, we obtain Theorem \ref{thm: SpecialCaseWeak}.
By $1$-periodicity of $\| \cdot \|$, 
we thereby obtain 
Theorem \ref{thm2}.
\end{proof}

For the remainder of this section, we establish that any 
$\gam \in \bR$ is admissible.
In fact, we will verify \emph{a fortiori} that
the properties \eqref{eq: divergence of approx. localised sets} 
and \eqref{eq: indepen of localised sets} 
of the approximation sets hold
for all sufficiently large values of $X$. Moreover, the first property will hold for all $\eta \in (0,1)$.

\subsection{Divergence of the series}

Let $\gam \in \bR$ and $\eta \in (0,1)$, and let $X \in \bN$ be large. Let $\eps_0$ be small, as in Remark \ref{epsSmall}. In this subsection, we
establish the property \eqref{eq: divergence of approx. localised sets}. We begin by estimating $\mu(\cE_n)$.

\begin{lemma}\label{lem: lower and upper bound on measure of approx sets for Liouville shifts}
Let $n \in \cG$ be large.
Then 
\begin{equation}
\label{eq: assumed measure bound approximation set}
 \frac{\varphi(\un)}{\un} \mu({\set I}) \Psi(n) 
 \ll \mu({\set E}_{n}) \ll
 \mu({\set I}) \Psi_{\even}(n).
\end{equation}
\end{lemma}

\begin{proof}
The enunciated upper bound follows at once by observing that 
${\set E_{n}}$ is contained in $O(\un \mu ({\set I}))$ 
many intervals of length $2\Psi_{\even}(n)/\un$.
On the other hand, it
contains 
\[
\varphi_0(\un) := \sum_{\substack{a + \gam \in \un {\set I} \\ 
(q'_t a + c_t, \un) = 1
}}1
\]
many open intervals of length $2\Psi_{\even}(n)/\un$ and, by \eqref{eq: PsiSmall}, these are disjoint.
We proceed to show that
\begin{equation}\label{eq: lower bound of fake phi}
    \varphi_0(\un) \gg \mu({\set I}) \varphi(\un),
\end{equation}
which would complete the proof. We will find that the implied constant in \eqref{eq: lower bound of fake phi} is absolute.

To infer \eqref{eq: lower bound of fake phi},
we apply Theorem \ref{thm: sieve bound}
with the following specifications.
The set $\cP$ of relevant primes is 
the set of primes that divide $\un$ but not $q'_t$, so that
\[
P(z) = \prod_{p \in \cP_{<z}} p,
\]
and the sifting sequence is
\[
a_m = 
\begin{cases}1, &\text{if } 
\exists a \in (\un {\set I} - \gam) \cap \bZ \quad q'_t a + c_t = m \\
0, &\text{otherwise}.
\end{cases}
\]
We choose
\[
x = q'_t(\un+1), \qquad w = 2,
\qquad z = (\log \un)^2.
\]
The other relevant data are
\[
g(p) = 1/p \quad (p \in \cP),
\qquad
\kap = 1,
\qquad X = \un \mu({\set I}) = O(1) + \sum_{m \le x} a_m,
\]
$K > 1$ is an absolute constant, and
\[
s = 10(1 + \log K), \qquad D = z^s.
\]
The dimension condition \eqref{DimensionCondition} follows from Mertens' third theorem (Theorem \ref{thm: Mertens}). With
\[
V(z) = \prod_{p \in \cP_{<z}} (1-p^{-1}),
\]
we obtain
\[
\sum_{(m,P(z))=1} a_m \ge (1 - e^{9-s}K^{10}) XV(z) - \sum_{\substack{d \mid P(z) \\d < D}} |r_d|, 
\]
where
\[
r_d + d^{-1}X = \sum_{m \equiv 0 \mmod d} a_m.
\]

For $d < D$ dividing $P(z)$, we have $(d,q'_t) = 1$, so
\[
\sum_{m \equiv 0 \mmod d} a_m
= d^{-1}X + O(1),
\]
ergo $r_d \ll 1$. Therefore
\[
\sum_{\substack{d \mid P(z) \\d < D}} |r_d|
\ll D = (\log n)^{O(1)},
\]
the upshot being that
\[
\sum_{(m,P(z))=1} a_m \ge
(1-e^{-1}) XV(z) + o(XV(z)) \gg XV(z).
\]
Finally, the union bound gives
\[
\sum_{(m,P(x)/P(z)) > 1} a_m
\ll X \sum_{\substack{p \mid n \\ p \ge z}} p^{-1}
\ll \frac{X \log n} {z \log z} = o(XV(z)),
\]
and so
\[
\varphi_0(\un) = \sum_{(m,P(x))=1} a_m
\gg XV(z) \gg \mu({\set I}) \varphi(\un).
\]
\end{proof}

Let $C\geq 4$ be an integer constant, large in terms of the implied constants in Lemmata \ref{lem: inner} and \ref{lem: outer}. We assume for a purely technical reason that $C$ is a perfect square, and let $N$ be large in terms of $C$ and the other constants. To estimate $\sum_{n \le X} \frac{\varphi(\un)}{\un} \mu({\set I}) \Psi(n)$, it will be useful to gather all $n$ on a scale $N$ 
such that, additionally, the $\| \hat n \alp_i - \gam_i \|$ 
are in prescribed $C$-adic ranges;
here and in the sequel, a $C$-adic range is a subinterval of $(0,\infty)$ whose right endpoint is $C$ times the left endpoint.

Define
\begin{align} \notag
\hat{\cB}_{\mathrm{loc}}(N; \brho)
&= \displaystyle \left \{ h \in \bN:
\begin{array}{ll}
     &  \hat N < h \leq \widehat{CN},\\
     & \rho_j < \| h \alp_j - \gam_j \| \le C \rho_j
\quad (1 \le j \le k-1) 
\end{array}
\right \} \\
\label{def: B loc hat}
&= \cB(\widehat{CN}; C\brho) \setminus \left(
\cB(\hat N; C\brho) \cup \bigcup_{i=1}^{k-1} 
\cB(\widehat{CN}; \brho_i) \right),
\end{align}
where $\brho_i= (\rho_{i,1},\ldots,\rho_{i,k-1})$, and 
$\rho_{i,j}$ is defined to be $\rho_j$ if $i=j$ and $C\rho_j$ otherwise.
For this section $\brho$
denotes a parameter
in the hyperrectangle
\begin{equation}\label{def: widened eps range}
\Rone := [\hat N^{-4.1 \eps_0}, \hat N^{-1.9 \eps_0}]^{k-1}.
\end{equation}
Similarly we will have a large parameter $M \le N$, and $\bdel$ will denote a parameter in $\cW(M)$. We also write
\begin{equation}\label{def: localised Bohr set}
{\set B}_{\mathrm{loc}}(N; \brho)
=\{n \in \bN: \hat{n}\in \hat{{\set B}}_{\mathrm{loc}}(N; \brho)\}.
\end{equation}
By the construction of $f$,
there is a uniquely determined integer $u=u(N)$
with 
\begin{equation}\label{eq: f is locally constant}
f(n)\in \{u,u+1\} \qquad (N \le n \le CN).
\end{equation}
Furthermore, since $f(n)\ll \log \log n$, we know that
if $n \asymp N$ then $\un/n = N^{o(1)}$.
We have ample use for this estimate,
and shall use it without further mention.

We proceed to study these localised Bohr sets, first deriving a cardinality estimate
in the range of interest.
Heuristically, one might expect 
that each of the $\Tet(\hat N)$ many integers
in the interval $[\hat N ,\widehat{CN}]$ lies in $\{ \hat n: n \in \cB_{\mathrm{loc}}(N; \brho) \}$ with probability roughly $\Nm(\brho)/4^{f(N)}$, recalling the notation \eqref{eq: ProductNotation}.
For the ranges occurring implicitly in the set ${\set G}$,
this heuristic correctly predicts the order of magnitude
of $\# \cB_{\mathrm{loc}}(N; \brho)$:

\begin{lemma}\label{lem: size of localised Bohr sets}
We have
\begin{equation*}
\# {\set B}_{\mathrm{loc}}(N; \brho) \asymp \Nm (\brho)  N,
\end{equation*}
uniformly for $\brho \in \Rone$. Here `uniformly' means that the implied constants are the same for all $\brho \in \Rone$.
\end{lemma}

\begin{proof}
In light of \eqref{eq: f is locally constant},
we have
\begin{equation} \label{eq: TwoValues}
4^{f(N)}, 4^{f(n)}\in \{2^\nu, 2^{\nu+2}\}
\qquad (n \in {\set B}_{\mathrm{loc}}( N;\brho)),
\end{equation}
for some positive integer $\nu$, where
\[
2^{\nu} \le 4^{O(\log \log N)} = (\log N)^{O(1)}.
\]

For the upper bound, each $n \in {\set B}_{\mathrm{loc}}(N;\brho)$ gives rise to a different element of
$\hat {\set B}_{\mathrm{loc}}(N;\brho) = \cB(\widehat{CN};C\brho)$ that is divisible by $2^\nu$. We can count the latter using Lemmata \ref{lem: outer} 
and \ref{lem: divisible numbers in GAPS} therein. Indeed, the former furnishes a generalised arithmetic progression
\[
\cP = \cP(b;A_1,\ldots,A_k; N_1, \ldots,N_k)
\]
containing $\cB(\widehat{CN};C\brho)$, where
$$
N_1\cdots N_k \ll \widehat{CN} \cdot \Nm (\brho) \ll \hat N \cdot \Nm (\brho), \qquad 
\min_{i \le k}N_i \ge N^{^{20k\eps_{_0}}},
$$
and $\gcd(A_1,\ldots,A_k) = 1$. Therefore 
\[
\# {\set B}_{\mathrm{loc}}(N;\brho) \le \# \{ h \in \cP: 2^\nu \mid h \},
\]
whereupon \eqref{eq: divisibilty in Bohr sets} yields
\[
\# {\set B}_{\mathrm{loc}}(N;\brho)
\ll \frac{N_1\cdots N_k}{2^\nu}
\ll \frac{\hat N \cdot \Pi(\brho)}{ 2^\nu}
\ll \frac{\hat N \cdot \Pi(\brho)} {4^{f(N)}} = N \cdot \Pi(\brho).
\]

For the lower bound, we divide the interval $(N,CN]$ 
into two subintervals $(N,\sqrt C N]$ and $(\sqrt CN, CN]$, 
and observe that $f$ must be constant on 
at least one of these subintervals. 
We assume that $4^{f(n)} = 2^\nu$ on $(N, \sqrt CN]$; 
the cases involving $(\sqrt C N, CN]$ and/or $2^{\nu+2}$ 
can be dealt with in the same manner. 
A lower bound is then given by the number 
of elements of \begin{align*}
&\hat {\set B}_{\mathrm{loc}}(N;\brho) 
\cap (2^\nu N, \sqrt C 2^\nu N]
\\
& \qquad = {\set B}(\widehat{\sqrt{C}N}; C\brho) \setminus
\left(
{\set B}(\hat N; C\brho) \cup \bigcup_{i=1}^{k-1} 
{\set B}(\widehat{\sqrt{C}N}; (\rho_{i,1},\ldots,\rho_{i,k-1})) \right)
\end{align*}
that are divisible by $2^\nu$, which we can estimate using Lemmata \ref{lem: inner}, \ref{lem: outer},
and \ref{lem: divisible numbers in GAPS}. The point is that $C$ is large, so the count for ${\set B}(\widehat{\sqrt{C}N}; C\brho)$ dominates, as we explain further in the next paragraph.

For the remainder of the proof, our implied constants do not depend on $C$. The argument that we used for the upper bound yields
\[
\# \{ h\in \cB(\hat N; C \brho): 2^\nu \mid h \} \ll C^{k-1} \Nm(\brho) N,  
\]
as well as
\[
\# \{ h \in {\set B}(\widehat{\sqrt{C}N}; (\rho_{i,1},\ldots,\rho_{i,k-1})): 2^\nu \mid h \} \ll \sqrt C N \prod_{j \le k-1} \rho_{i,j} = C^{k-3/2} \Nm(\brho) N
\]
for $i=1,2,\ldots,k-1$. On the other hand, Lemma \ref{lem: inner} furnishes a proper, asymmetric generalised arithmetic progression
\[
\cP' = \cP^+(b'; A'_1, \ldots, A'_k; N'_1, \ldots,N'_k)
\]
contained in ${\set B}(\widehat{\sqrt{C}N}; C\brho)$, where
\[
N'_1\cdots N'_k \gg \widehat{\sqrt C N} \cdot C^{k-1} \Nm (\brho) \gg \hat N \cdot C^{k-1/2} \Nm (\brho), \qquad 
\min_{i \le k}N'_i \gg  N^{^{20k\eps_{_0}}},
\]
and $\gcd(A'_1,\ldots,A'_k) = 1$. Thus, by \eqref{eq: DivisibilityInBohrSets2}, we have
\[
\# \{ h \in {\set B}(\widehat{\sqrt{C}N}; C\brho): 2^\nu \mid h \} \ge \# \{ h \in \cP': 2^\nu \mid h \} \gg  C^{k-1/2} \Nm(\brho) N.
\]
As $C$ is large, we conclude that
\[
\# \cB_{\mathrm{loc}}(N;\brho) \gg C^{k-1/2} \Nm(\brho) N,
\]
which completes the proof.
\end{proof}

Before we demonstrate
\eqref{eq: divergence of approx. localised sets}, we prepare one more lemma.
 Since \eqref{eq: divergence of approx. localised sets} 
requires us to determine the order of magnitude of 
$\sum_{n\leq N} \mu({\set E_{n}})$,
it is helpful to know that the totient weights on the left hand side of
\eqref{eq: assumed measure bound approximation set} `average well' in suitable $C$-adic ranges, i.e.
that for $n$ in a localised Bohr set 
${\set B}_{\mathrm{loc}}(N; \brho)$,
the average of $\varphi(\un)/\un$
has order of magnitude $1$.
\begin{remark}

A similar strategy for proving this point
was employed by the first named author in 
\cite[Lemma 3.1]{Cho2018},
and then by both authors in \cite[Lemma 4.1]{CT2019}.
The additional difficulty here is a mild one:
Roughly speaking, we need to average only over elements 
of the Bohr sets which are divisible by 
an appropriate power of four. 
\end{remark}

\begin{lemma} [Good averaging]\label{averagegood}
We have
\[
\sum_{n \in {\set B}_{\mathrm{loc}}(N;\brho) } 
\frac{\varphi(\un)}\un \gg \# {\set B}_{\mathrm{loc}}(N;\brho),
\]
uniformly for all $\brho \in \Rone$.
\end{lemma}

\begin{proof} 
By the AM--GM inequality, we have
\[
\frac{1}{\# {\set B}_{\mathrm{loc}}(N;\brho)}
\sum_{n \in {\set B}_{\mathrm{loc}}(N;\brho) } 
\frac{\varphi(\un)}\un 
\geq
\left(
\prod_{n \in {\set P} } \frac{\varphi(\un)}\un
\right)^{\frac{1}{ \# {\set B}_{\mathrm{loc}}(N;\brho)}}
= \prod_{p \le \widehat{CN}} \left(1 - \frac1p\right)^{\tau_p},
\]
where 
\[
\tau_p = \frac{ \# \{ n \in 
{\set B}_{\mathrm{loc}}(N;\brho) : p \mid \un \}} {\# {\set B}_{\mathrm{loc}}(N;\brho) }.
\]
Now
\[
-\ln \left(
\frac{1}{\# {\set B}_{\mathrm{loc}}(N;\brho)}
\sum_{n \in {\set B}_{\mathrm{loc}}(N;\brho) } 
\frac{\varphi(\un)}\un \right) \ll
\sum_{p \le \widehat{CN}} \frac{\tau_p}{p},
\]
so as $\tau_2 \le 1$ it suffices to show that
\[
\tau_p \ll p^{-\eps_0} \qquad (3 \le p \le \widehat{CN}).
\]

Suppose $3 \le p \le \widehat{CN}$. By Lemma 
\ref{lem: size of localised Bohr sets}, we have
\[
\tau_p 
\ll \frac{\# \{ h \in \cB(\widehat{CN}; C\brho):  h \equiv 0 \mmod 2^\nu p \} }{\rho N},
\]
where $\nu$ is as in \eqref{eq: TwoValues} and $\rho = \Pi(\brho)$.
We can estimate the numerator via Lemmata \ref{lem: outer} and
\ref{lem: divisible numbers in GAPS}, noting that
$$
N_1\cdots N_k \asymp \widehat{CN} \cdot \rho, \qquad 
\min_{i \le k}N_i \ge N^{^{20k\eps_{_0}}}
$$
therein.
We obtain
\[
\tau_p \ll 2^\nu ((2^\nu p)^{-1} + N^{-20 k \eps_0}) \ll p^{-\eps_0}.
\]
\end{proof}

We can now estimate the normalised-totient-weighted
sums of measures, thereby attaining the main result of this subsection. For $j \in \bN$, define
\begin{equation}\label{def: Gj and dj} {\set D}_{j}=(C^j , C^{j+1}] \cap  {\set G}_\diam, \qquad 
\cG_j = (C^j , C^{j+1}] \cap  {\set G}.
\end{equation}

\begin{lemma}
\label{lem: divergence lemma}
The sets $\cE_n$
satisfy 
\eqref{eq: divergence of approx. localised sets} for all sufficiently large $X$. 
\end{lemma}

\begin{proof}
For $j$ large, observe that $\cD_j$ is contained in a union of $O(j^{k-1})$ 
sets
\[
{\set B}_{\mathrm{loc}}(C^j; (C^{-t_1}, \ldots , C^{-t_{k-1}}))
\]
satisfying the hypotheses of 
Lemma \ref{lem: size of localised Bohr sets}, and so
\begin{equation}\label{eq: size of recip sum}
T_j := \sum_{n \in {\set D}_j}  
\frac{1}{ \| \hat{n} \alp_1 - \gam_1 \| 
\cdots \| \hat{n} \alp_{k-1} - \gam_{k-1} \|  } 
\ll j^{k-1} C^j.
\end{equation}
Let $J$ be the largest integer $j$ such that $C^j \leq X$. 
By \eqref{eq: BoundMeasureOfApprox},
we have
$$
\sum_{C^j < n\leq C^{j+1}} \mu({\set E}_n) \ll 
\sum_{n\in {\set D}_{j}}  \mu(\cI) \Psi(n)
\le
\mu(\cI) \psi(\widehat{C^{j}}) T_j
$$
for $j$ large.
Applying \eqref{eq: size of recip sum} and summing over $j$ yields
\begin{align*}
\sum_{j\leq J}  \psi(\widehat{C^{j}}) T_j
\ll 1 +
\sum_{j\leq J}  j^{k-1} C^j \psi(\widehat{C^{j}}).
\end{align*}
By \eqref{eq: doubling on average},
we now have
\begin{equation*}
\mu(\cI)^{-1} \sum_{n\leq X} \mu({\set E}_n) 
\ll
1 +
\sum_{j\leq J}  j^{k-1} C^j \psi(\widehat{C^{j}})
\ll \sum_{n \le X} \psi(\un) (\log n)^{k-1},
\end{equation*}
recalling that $X$ is large and recalling from
\eqref{eq: HatDivergence} that the right hand side diverges 
as $X \to \infty$. We have established the upper bound in \eqref{eq: divergence of approx. localised sets}.

For the lower bound, we begin by applying \eqref{eq: assumed measure bound approximation set} to give
\[
\sum_{n\leq X} \mu({\set E}_n)
\gg
\mu(\cI) \sum_{C_1 < n\leq X} \frac{\varphi(\un)}{\un} \Psi(n),
\]
where $C_1$ is a large, positive constant. Our strategy is to show that
\begin{equation} \label{eq: show1}
\sum_{\substack{n \in \cG_\diam \\ n \le X}}
\frac{\varphi(\un)}{\un} \:
\frac{ \psi(\hat n)}
{ \| \hat n \alp_1 - \gam_1 \| \cdots \| \hat n \alp_{k-1} - \gam_{k-1} \| } \gg \sum_{n\leq X} \psi(\un) (\log n)^{k-1}
\end{equation}
and
\begin{equation} \label{eq: show2}
\sum_{n \in \cG_\diam \setminus \cG}
\frac{\varphi(\un)}{\un} \:
\frac{ \psi(\hat n)}
{ \| \hat n \alp_1 - \gam_1 \| \cdots \| \hat n \alp_{k-1} - \gam_{k-1} \| } \ll 1.
\end{equation}
By \eqref{eq: HatDivergence}, the right hand size of \eqref{eq: show1} diverges as $X \to \infty$, and so these two bounds will suffice.

We proceed to bound from below the contribution from $n\in {\set D}_j$, for $j$ large. To this end, consider the localised Bohr sets
$
{\set B}_{\mathrm{loc}}(C^j; \brho(\bt)),
$
where 
\[
\brho(\bt) = (C^{-t_1}, \ldots , C^{-t_{k-1}}).
\]
For $n \in {\set B}_{\mathrm{loc}}(C^j;  \brho(\bt))$,
where
$\bt=(t_1,\ldots,t_{k-1})$ satisfies
\[
2.1 \eps_0 (j+1) \le  t_r \le 3.9 \eps_0 j \quad (1\leq r \leq k-1),
\] 
we deduce that
\[
\hat n^{-4 \eps_0} \le C^{-t_r} \le C^{1-t_r}
\le \hat n^{-2 \eps_0} \qquad (1 \le r \le k-1).
\]
Hence $n \in {\set D}_j$, and so we see that
\[
{\set B}_{\mathrm{loc}}(C^j; \brho(\bt)) \subseteq {\set D}_j.
\]

Next, observe that 
$$
\sum_{n\in{\set D}_j} \frac{\varphi(\un)}{\un} \:
\frac{ \psi(\hat n)}
{ \| \hat n \alp_1 - \gam_1 \| \cdots \| \hat n \alp_{k-1} - \gam_{k-1} \| }
\gg
\psi(\widehat{C^{j+1}})
\sum_{\bt}
\sum_{n\in {\set B}_{\mathrm{loc}}(C^{j}; \brho(\bt))}
\frac{\varphi (\un)}{\un \rho(\bt)},
$$
where
$\bt$ runs through all the integer vectors as above
and $\rho(\bt)= \Nm(\brho(\bt))$.
Lemma \ref{lem: size of localised Bohr sets} 
implies
$\# {\set B}_{\mathrm{loc}}(C^j; \brho(\bt)) \gg C^{j} \rho(\bt)$, and so Lemma \ref{averagegood} assures us that 
$$
\sum_{n\in {\set B}_{\mathrm{loc}}(C^j; \brho(\bt))}
\frac{\varphi (\un)}{\un \rho(\bt)} \gg C^{j},
$$
uniformly for each of the $\Tet(j^{k-1})$ 
many choices of $\bt$.
Whence
\begin{equation} \label{DjLow}
\sum_{n\in{\set D}_j} \frac{\varphi(\un)}{\un} \:
\frac{ \psi(\hat n)}
{ \| \hat n \alp_1 - \gam_1 \| \cdots \| \hat n \alp_{k-1} - \gam_{k-1} \| }
\gg 
\psi(\widehat{C^{j+1}}) C^{j+1} (j+1)^{k-1}.
\end{equation}
Let $J_0$ be large, positive constant. Summing \eqref{DjLow} over the range
\[
J_0 \le j \leq \frac{\log X}{\log C}
\]
yields
\[
\sum_{\substack{n \in \cG_\diam \\ n \le X}}
\frac{\varphi(\un)}{\un} \:
\frac{ \psi(\hat n)}
{ \| \hat n \alp_1 - \gam_1 \| \cdots \| \hat n \alp_{k-1} - \gam_{k-1} \| } \gg 
\sum_{J_0 \le j \le \frac{\log X}{\log C}} \psi(\widehat{C^{j+1}}) C^{j+1} (j+1)^{k-1}.
\]
Lemma \ref{lem: truncation} now
delivers the inequality \eqref{eq: show1}.

Finally, by \eqref{eq: size of recip sum} we have
\begin{align*}
&\sum_{n \in \cG_\diam \setminus \cG}
\frac{\varphi(\un)}{\un} \:
\frac{ \psi(\hat n)}
{ \| \hat n \alp_1 - \gam_1 \| \cdots \| \hat n \alp_{k-1} - \gam_{k-1} \| } \\
&\le 
\sum_{j=1}^\infty
\sum_{n \in {\set D}_j \setminus {\set G}}
\frac{\psi(\un)}{ \| \hat{n} \alp_1 - \gam_1 \| 
\cdots \| \hat{n} \alp_{k-1} - \gam_{k-1} \|  } 
\\ &\ll 1+ \sum_{j=1}^\infty \frac{j^{k-1} C^j} 
{C^j (j \log C)^{k+1}}
\ll 1+ \sum_{j=1}^\infty j^{-2} \ll 1,
\end{align*}
which is \eqref{eq: show2}.
\end{proof}

Having established \eqref{eq: divergence of approx. localised sets}, our final task for this section is to prove \eqref{eq: indepen of localised sets} for some $\eps_0 > 0$ and $\eta \in (0,1)$ depending on $\gam$.

\subsection{Overlap estimates, localised Bohr sets, and the small-GCD regime}

In the present subsection, we reduce the task of proving
\eqref{eq: indepen of localised sets} to 
demonstrating a uniform estimate in localised Bohr sets, 
see Lemma \ref{lem: bound on dyadic overlaps}. 
Thereafter, we recast the desired estimate 
in terms of counting solutions to diophantine inequalities in localised Bohr sets, see Lemma \ref{lem: overlaps and inequalities}. 
At the end of this subsection, we establish such a counting result
in a regime where the arising GCDs 
are relatively small, see Proposition \ref{prop: overlap lemma}.

Let $N \geq M$, where $M$ is large. 
The next lemma asserts that if we have a good bound on
\begin{equation*}
R^{{\set I}, \gam}(M,N; \brho, \bdel) :=
\sum_{\substack{ 
n\in {\set B}_{\mathrm{loc}}(N;\brho)\\
m\in {\set B}_{\mathrm{loc}}(M; \bdel)\\
n\neq m}}
\mu({\set E}_n^{{\set I}, \gam} \cap {\set E}_m^{{\set I}, \gam}),
\end{equation*}
sufficiently uniform in  $\brho$ and $\bdel$,
then we can deduce
\eqref{eq: indepen of localised sets}. 
As before, we shall drop the dependency on $\gam$ and ${\set I}$
for most of the time, and 
simply write $R(M,N; \brho, \bdel) $ instead.

\begin{lemma} \label{lem: bound on dyadic overlaps}
If 
\begin{equation}\label{eq: bound on dyadic overlaps}
R(M,N; \brho, \bdel) \ll \mu({\set I}) MN \psi(\uM)\psi(\uN), 
\end{equation}
uniformly for $M$ large and $N \ge M$, $\brho \in {\set W}(N)$, and $\bdel \in {\set W}(M)$,
then \eqref{eq: indepen of localised sets} holds.
\end{lemma}

\begin{proof}
Let $J$ be the least integer $j$ such that $C^j \geq N $, and recall \eqref{def: Gj and dj}. 
Then
$$
\sum_{\substack{m,n \leq X}}  
\mu({\set E}_{n} \cap {\set E}_{m}) 
\ll 1 +
\sum_{\substack{j \le i \leq J}}  
\sum_{\substack{n\in {\set G}_i\\
m\in {\set G}_j}}
\mu({\set E}_{n} \cap {\set E}_{m}).
$$
Let $X_0 \in \bN$ be a large constant. 
By Lemma \ref{lem: divergence lemma}, 
the contribution from $j < X_0$ is bounded by a constant times
\[
\sum_{n \le X} \mu({\set E}_n) \ll 
 \mu({\set I}) \sum_{n \le X} \psi(\un) (\log n)^{k-1}.
\]
Thus, recalling from \eqref{eq: HatDivergence} that the completed series diverges, it remains to show that
\begin{equation} \label{OverlapGoal}
\sum_{X_0 \le j \le i \le J} \sum_{\substack{n\in {\set G}_i\\
m\in {\set G}_j}} 
\mu({\set E}_{n} \cap {\set E}_{m})
\ll \mu({\set I}) \left(\sum_{n \le X} 
\psi(\un) (\log n)^{k-1} \right)^2.
\end{equation}

Presently, we fix $i,j$ with $X_0 \le j \le i \le J$. Consider the vectors
\begin{equation} \label{eq: RhoDelShape}
\brho (\bt) =(C^{-t_1},\ldots,C^{-t_{k-1}}), \quad
\bdel (\boldsymbol \ell)= (C^{-\ell_1},\ldots,C^{-\ell_{k-1}}),
\end{equation}
wherein ranges for the
exponents will be prescribed shortly.
Let us account for the contribution of the diagonal first. 
By \eqref{eq: HatDivergence}, the summands for which $n=m$ contribute 
$$
\sum_{X_0 \le i \le J} \sum_{n\in {\set G}_i} 
\mu({\set E}_{n})
\ll 
\mu({\set I}) \sum_{n \le X} 
\psi(\un) (\log n)^{k-1}
\ll \mu({\set I}) \left(\sum_{n \le X} 
\psi(\un) (\log n)^{k-1} \right)^2.
$$

Next, we account for the off-diagonal contribution.
To this end, we approximately decompose
$$
S(i,j) := \sum_{\substack{ 
n \in {\set G}_i \\ 
m \in {\set G}_j\\
n\neq m}}
\mu({\set E}_n\cap {\set E}_m)
$$
into sums of the shape 
$$
S(i,j;\bt, \boldsymbol \ell) := \sum_{\substack{ 
n\in {\set B}_{\mathrm{loc}}(C^i;\brho(\bt) )\\
m\in {\set B}_{\mathrm{loc}}(C^j; \bdel(\bl))\\
n\neq m}} 
\mu({\set E}_n\cap {\set E}_m).
$$
Specifically, the sum $S(i,j)$ is bounded above by the sum of 
$S(i,j; \bt, \boldsymbol \ell)$ over the integer vectors $\bt$, $\bl$ 
within the ranges
\[
1.9 \eps_0 i \le t_1, \ldots, t_{k-1} \le 4.1 \eps_0 i,
\qquad
1.9 \eps_0 j \le \ell_1, \ldots, \ell_{k-1} \le 4.1 \eps_0 j.
\]
By the assumed estimate \eqref{eq: bound on dyadic overlaps},
we have
$$
S(i,j;\bt, \bl) \ll 
\mu({\set I}) C^{i}C^{j} \hat{\psi}(C^i) \hat{\psi}(C^j),
$$
uniformly in $\bt$ and $\bl$ as above.
Summing over the $O(i^{k-1})$ many choices for $\bt$
and the $O(j^{k-1})$ many choices for $\bl$,
we see that 
$$
S(i,j) \ll
\mu({\set I}) i^{k-1}C^{i}j^{k-1}C^{j} 
\hat{\psi}(C^i) \hat{\psi}(C^j).
$$
Summing over $i$ and $j$ gives
$$
\sum_{X_0 \le j \le i \le J} \sum_{\substack{n\in {\set G}_i\\
m\in {\set G}_j}} 
\mu({\set E}_{n} \cap {\set E}_{m})
\ll 
\mu({\set I})
\bigg( \sum_{\substack{X_0 \le j \leq J}}  
 j^{k-1} C^{j} \hat{\psi}(C^j)
\bigg)^2.
$$
Now Lemma \ref{lem: truncation} 
delivers \eqref{OverlapGoal}, completing the proof.
\end{proof}

Denote by $N_1,\ldots, N_{k}$ the length parameters,
arising from Lemma \ref{lem: outer}, 
associated to the outer structure of 
$\hat {\set B}_{\mathrm{loc}}(N;\brho)$. 
Specifically, we apply the lemma to 
${\set B}(\widehat{CN};C\brho)$, 
with $\eps = \eps_0$ and $\vartheta = 20k$. 
By symmetry, we may assume that
\begin{equation} \label{eq: N1min}
N_1 = \min\{ N_i: 1 \le i \le k\},
\end{equation}
and we do so in order to simplify notation. Note that
\[
N_1 \ge N^{20 k \eps_0}.
\]
Let $M$ be large, let $N \ge M$, and let
\begin{equation}\label{def: Delta}
\Del =  \widehat{CM} \frac{\psi(\uN)}{\Nm(\brho)} + 
\widehat{CN} \frac{\psi(\uM)}{\Nm (\bdel)}.
\end{equation}
Importantly, if there exists $m\in \cB_{\mathrm{loc}}(M;\bdel) \cap \cG$ then we have a lower bound 
on $\Del$ of the strength
\begin{align*}
\Del &\geq \widehat{CN} \frac{\psi(\uM)}{\Nm(\bdel)}
\ge \widehat{CN} \frac{\psi(\hat m)}{\Nm(\bdel)}
    \geq \frac{\widehat{CN}} {\Nm(\bdel) m( \log m)^{k+1}}\\&
    \geq \frac{\widehat{CN} }{CM}  \: 
    \frac{1}{\Nm(\bdel) (\log (CM))^{k+1}}
    \geq \frac{1}{\Nm (\bdel) (\log (CM))^{k+1}}.
\end{align*}
For $\bdel \in {\set W}(M)$, 
recalling \eqref{def: widened eps range},
we find that for all sufficiently large $M$ we have
\begin{equation}\label{eq: lower bound on Delta}
\Del > M^{(k-1)\eps_0} >1.
\end{equation}

Let $(q'_\ell)_\ell$ be the sequence of continued fraction denominators of $\gam$.
We estimate the contribution to the quantity 
$R(M,N; \brho, \bdel)$
from the small-GCD regime (including the diagonal), i.e. 
\begin{equation}\label{eq: R small gcd}
R_{\gcd \leq }(M,N; \brho, \bdel) := 
\sum_{\substack{ 
n\in {\set B}_{\mathrm{loc}}(N;\brho)\\
m\in {\set B}_{\mathrm{loc}}(M; \bdel)\\
\gcd(\un,\um)\leq \max\{3\Del/\|q'_2 \gam \|, \, N_1\} }}
\mu({\set E}_n \cap {\set E}_m),
\end{equation}
and the contribution from large-GCD regime, i.e. 
\begin{equation}\label{eq: R large gcd}
R_{\gcd >}(M,N; \brho, \bdel) := 
\sum_{\substack{ 
n\in {\set B}_{\mathrm{loc}}(N;\brho)\\
m\in {\set B}_{\mathrm{loc}}(M; \bdel)\\
\gcd(\un,\um)> \max\{3\Del/\|q'_2 \gam \|, \, N_1\}\\
n\neq m}}
\mu({\set E}_n \cap {\set E}_m),
\end{equation}
separately. When bounding these quantities, we may assume that $\cB_{\mathrm{loc}}(M;\bdel)$ intersects $\cG$, and so we have \eqref{eq: lower bound on Delta}.

Let us write 
\[
\rho = \Pi (\brho),
\qquad
\del = \Pi (\bdel).
\]
As a first step towards estimating
$R_{\gcd \leq }(M,N; \brho, \bdel)$,
we record a useful
relation between the size of 
$R_{\gcd \leq }(M,N; \brho, \bdel)$
and the number of solutions to 
a diophantine inequality with various
constraints.

\begin{lemma}\label{lem: overlaps and inequalities}
Let $M$ be large, let $N \ge M$, and let $\brho \in \cW(N)$, and $\bdel \in \cW(M)$.
Denote by $D_{\gcd \leq}$ the number of quadruples
$(n,m,a,b) \in \bN^2 \times \bZ^2$
for which 
\begin{equation}\label{eq: constraints numer in overlap est}
\frac{a+\gam}{\un}, \frac{b+\gam}{\um}\in {\set I},
\quad
n \in {\set B}_{\mathrm{loc}}(N; \brho) \cap {\set G}, 
\quad
m\in {\set B}_{\mathrm{loc}}(M; \bdel) \cap {\set G},
\end{equation}
as well as 
\begin{equation}\label{eq: gcd small}
\gcd(\un,\um) \le \max (3\Del/\|q'_2 \gam \|, N_1)
\end{equation}
and
\begin{equation}\label{eq: overlap inequality}
\vert (\un-\um)\gam - (\um a-\un b) \vert  \le \Del,
\end{equation} 
and finally
\begin{equation} \label{eq: shifty}
(a, \hat n), (b, \hat m)  \text{ are } (\gam,\eta)\text{-shift-reduced}.
\end{equation}
Furthermore, let $D_{\gcd >}$ be the number of quadruples 
$(n,m,a,b) \in \bN^2 \times \bZ^2$ which satisfy 
\eqref{eq: constraints numer in overlap est} 
, \eqref{eq: overlap inequality}, and \eqref{eq: shifty}, but instead of 
\eqref{eq: gcd small} the reversed inequality 
$\gcd(\un,\um) > \max (3\Del/\|q'_2 \gam \|, N_1)$
and the constraint $n\neq m$.
If 
\begin{equation}\label{eq: bound on DNM}
D_{\gcd \leq } \: = O(\mu({\set I})\rho N \del M \Del),
\end{equation}
then
$ R_{\gcd \leq }(M,N; \brho, \bdel)  \ll 
\mu({\set I}) MN \psi(\uM)\psi(\uN)$.
Moreover, if we have 
\begin{equation}\label{eq: gcd large}
D_{\gcd > } \: = O(\mu({\set I})\rho N \del M \Del),
\end{equation}
then $R_{\gcd > }(M,N; \brho, \bdel)  \ll 
\mu({\set I}) MN \psi(\uM)\psi(\uN)$.
\end{lemma}

\begin{proof}
We detail the proof only in the case of $D_{\gcd \leq  }$
since the case $D_{\gcd >}$ can be dealt with in the same way.
We begin by observing that each ${\set E}_n$ (resp. ${\set E}_m$) 
is contained in a union of finitely many intervals 
\[
{\set I}_{n,a} = \left(
\frac{ a + \gam - \Psi(n) } \un,
\frac{ a + \gam + \Psi(n) } \un \right)
\]
(resp. ${\set I}_{m,b}$), and so
$$
\mu({\set E}_n \cap {\set E}_m) \leq \max\{
\min(\mu({\set I}_{n,a}),
\mu({\set I}_{m,b})) : a,b \in \bZ \} \cdot \#\{(a,b):
{\set I}_{n,a}\cap {\set I}_{m,b} \neq \emptyset \}.
$$
The first factor is bounded above by
$$2\min \Big( \frac{\Psi(n)}{\un}, \frac{\Psi(m)}{\um} \Big) 
\le
2\min \Big( \frac{\psi(\uN)}{\rho \uN},
\dfrac{\psi(\uM)}{\del \uM} \Big). $$
To bound the second factor, 
we have to count how often the centre
$(a+\gam)/\un$ of 
${\set I}_{n,a}$ is `sufficiently close' to the centre $(b+\gam)/\um$ of an interval ${\set I}_{m,b}$.
Here sufficiently close means that the distance
of the centres is less than sum of the radii of the intervals, 
i.e.
$$   
\Big\vert \frac{b+\gam}{\um}
-  \frac{a+\gam}{\un} \Big\vert 
\le
\frac{\Psi(n)}{\un} + \frac{\Psi(m)}{\um}.
$$
Multiplying by $\um \un$, we see that 
\[
\vert \un (b+\gam)
-  \um (a+\gam) \vert 
\le
\widehat{CM} \Psi(n) + \widehat{CN} \Psi(m)
\leq \Del,
\]
and the number of integer solutions $(n,m,a,b)$ to the above inequality, subject to our constraints, is at most $D_{\gcd \leq}$.
Thus, if \eqref{eq: bound on DNM} holds then
\begin{align*}
R_{\gcd \leq }(M, N; \brho, \bdel) \ll 
\min \Big( \frac{\psi(\uN)}{\rho \uN}, 
\frac{\psi(\uM)}{\del\uM} \Big)
\mu({\set I})\rho N \del M \Del.
\end{align*}

Next, observe that
$$
\Del \leq 2 \widehat{CM} \widehat{CN} 
\max \Big(\frac{\psi(\uN)}{\rho \widehat{CN}}, 
\frac{\psi(\uM)}{\del\widehat{CM}}\Big).
$$
It follows from \eqref{eq: f is locally constant} that
$\widehat{N} \asymp \widehat{CN}$
and $\widehat{M} \asymp \widehat{CM}$,
and so
\[
\min \Bigl( \frac{\psi(\uN)}{\rho \uN}, 
\frac{\psi(\uM)}{\del\uM} \Bigr) 
\ll 
\min \Big( \frac{\psi(\uN)}{\rho \widehat{CN}}, 
\frac{\psi(\uM)}{\del \widehat{CM}} \Big).
\]
The upshot is that 
$ R_{\gcd \leq }(M, N; \brho, \bdel)$
is at most a constant times
\begin{equation}\label{eq: overlap inequalities count not simplified}
\min \Big( \frac{\psi(\uN)}{\rho \widehat{CN}}, 
\frac{\psi(\uM)}{\del \widehat{CM}} \Big)
\rho N \del M \mu({\set I}) \widehat{CM} \widehat{CN} 
\max \Big(\frac{\psi(\uN)}{\rho \widehat{CN}}, 
\frac{\psi(\uM)}{\del\widehat{CM}}\Big).
\end{equation}
Applying the simple identity
\[
\min \Big( \frac{\psi(\uN)}{\rho \widehat{CN}}, 
\frac{\psi(\uM)}{\del \widehat{CM}} \Big)
\max \Big(\frac{\psi(\uN)}{\rho \widehat{CN}}, 
\frac{\psi(\uM)}{\del\widehat{CM}}\Big)
=
\frac{\psi(\uN)}{\rho \widehat{CN}}
\frac{\psi(\uM)}{\del\widehat{CM}},
\]
we see that 
\eqref{eq: overlap inequalities count not simplified}
simplifies to $\mu({\set I}) M N \psi(\uM) \psi(\uN)$.
\end{proof}

Next, we establish overlap estimates
in the regime in which the GCDs are relatively small.
Here congruence considerations and the structural theory of Bohr
sets will be decisive; a welcome circumstance
is that the bound works for any final shift $\gam$,
irrespective of its diophantine nature. 

\begin{prop}\label{prop: overlap lemma}
Let $M$ be large, and let $N \ge M$. Then
\begin{equation*} 
R_{\gcd \leq }(M,N; \brho, \bdel)  \ll 
\mu({\set I}) MN \psi(\uM)\psi(\uN),
\end{equation*}
uniformly for $\brho \in {\set W}(N)$
and $\bdel \in {\set W}(M)$.
\end{prop} 

\begin{proof}
By Lemma \ref{lem: overlaps and inequalities},
it remains to verify \eqref{eq: bound on DNM}. 
We distinguish several cases according to the size of
the of the greatest common divisor $d :=(\un,\um)$. Put
\[
\qquad \un = dx \quad \mathrm{and} \quad \um = dy.
\]
As $f$ is non-decreasing, we have
$$ 
d = (\un,\um) = (4^{f(n)}n,4^{f(m)}m) \geqslant 4^{f(M)}
\geqslant \mu({\set I})^{-1},
$$
where the last inequality comes from $M$ 
being large in terms of ${\set  I}$. 
We rewrite the overlap 
inequality \eqref{eq: overlap inequality} 
as
\begin{equation}\label{eq: third Bohr set}
|(x-y)\gam - (ya - xb)| \le \Del/d.
\end{equation}

Recall that $(q'_\ell)_\ell$ is
the sequence of continued fraction denominators of $\gam$.

\underline{Case 1: $\mu({\set I})^{-1}\le d \le 3\Del/\|q'_2 \gam \|$}

By Lemma \ref{lem: size of localised Bohr sets}, there are $O(\rho N \del M)$ possible choices of 
$n \in {\set B}_{\mathrm{loc}}(N; \brho)$
and $m\in {\set B}_{\mathrm{loc}}(M; \bdel)$.
Now suppose we are given $n$ and $m$, which then uniquely determine $d,x,y$. 

\underline{Case 1a: $\Del < \un/4$}

Since $x$ and $y$ are coprime, 
the inequality \eqref{eq: third Bohr set} 
restricts the integer $a$ to one of at most 
$2\Del/d + 1$ residue classes modulo $x$. 
Furthermore, the constraint $a+\gam \in \hat n {\set I}$ 
places the integer $a$ in an interval 
of length $dx \mu({\set I})$. Within this interval, 
and given $r \in \bZ / x \bZ$, there can be at most
$d \mu({\set I}) +1$ integers $a \equiv r \mmod x$. 
Thus, given $n$ and $m$, there are
at most 
\[
(d\mu({\set I})+1)(2\Del/d +1) \ll \mu({\set I}) \Del
\]
possibilities for $a$. Finally, the 
inequality \eqref{eq: third Bohr set} 
constrains the integer $b$ to an interval 
of length at most $2\Del/\hat n < 1/2$, so $b$ 
(if it exists at all) is determined by the other variables.

\underline{Case 1b: $\Del \ge \un/4$}

We may suppose that \eqref{eq: overlap inequality} has a solution $(a,b) = (a_0, b_0)$ with $a_0 + \gam \in n \cI$. By the triangle inequality, any solution to \eqref{eq: overlap inequality} for which $a+\gam \in n \cI$ satisfies
\begin{equation} \label{eq: region}
| \hat m a' - \hat n b'| \le 2\Del, \qquad a' \in \un (\cI - \cI),
\end{equation}
where 
\[
a' = a - a_0, \qquad b' = b-b_0,
\qquad \cI - \cI = \{ r_1 - r_2: r_1, r_2 \in \cI \}.
\]
Denote by $\cR$ the closure of the set of $(a', b') \in \bR^2$ satisfying \eqref{eq: region}. We apply Theorem \ref{DavenportBound} to the region $\cR$ and the lattice $\bZ^2$. Observe using the triangle inequality that $\cR$ is contained in the rectangle
\[
\left \{ (a',b') \in \bR^2: a' \in \un(\cI - \cI),
\quad |b'| \le \frac{2\Del}\un + 2\mu(\cI) \um
\right \},
\]
and in particular the projection of $\cR$ onto any line has length at most
\[
2 \mu(\cI) \un + \frac{4\Del}\un + 4\mu(\cI) \um \ll \un \mu(\cI),
\]
where for the final inequality we have used that $\psi(\hat N), \psi(\hat M) \le \psi(1) \ll 1$, as well as that $\brho \in \cW(N)$ and $\bdel \in \cW(M)$.
The area of $\cR$ is $O(\Del \mu(\cI))$, and we glean
that the number of solutions is \[
O(\Del \mu(\cI) + \un \mu(\cI) + 1) = O(\Del \mu(\cI)).
\]

We conclude that there are 
$O(\mu({\set I}) \rho N \del M \Del)$ 
solutions in total coming from Case 1,
uniformly in $\brho \in \Rone$
and $\bdel \in {\set W}(M)$.

\underline{Case 2: $3\Del/\|q'_2 \gam \| \le d \le N_1$}

By Lemma \ref{lem: size of localised Bohr sets}, there are $O(\del M)$ possibilities for 
$m \in {\set B}_{\mathrm{loc}}(M; \bdel)$. 
Next, we choose $d \in [\Del, N_1]$ dividing $m$, in one of $M^{o(1)}$ ways. By Lemma \ref{lem: outer}, we have
\[
dx = \un = A_1 n_1 + \cdots + A_k n_k + s,
\]
where $|n_i| \le N_i$ ($1 \le i \le k$) and
\[
N_1 \cdots N_k \asymp \rho \uN.
\]
As $d \le N_1$, wherein we recall \eqref{eq: N1min}, Lemma \ref{lem: divisible numbers in GAPS}
assures us that the congruence
\[
A_1 n_1 + \cdots + A_k n_k + s \equiv 0 \mmod d
\]
has $O(\rho \hat N / d)$ solutions $n_1,\ldots,n_k$, and these variables determine $x$.

Now suppose that we are given $d,x,y \in \bN$ with $d \ge \Del$ and $\gcd(x,y) = 1$. Then, as $d \ge \Del$, by the same argument as in Case 1a there are at most
\[
(d \mu(\cI) + 1)(2\Del/d + 1) \ll d \mu(\cI)
\]
many choices for the pair $(a,b) \in \bZ^2$ subject to \eqref{eq: third Bohr set} as well as  $a+\gam \in \hat n \cI$. In view of \eqref{eq: lower bound on Delta}, the total is again $O(\mu({\set I})\rho N \del M \Del)$. 

We have established \eqref{eq: bound on DNM}, completing the proof.
\end{proof}

\subsection{Large GCDs}

Let $M$ be large, $N \ge M$, $\brho \in \cW(N)$, and $\bdel \in \cW(M)$.
On account of the preceding subsections, our final sub-task for this section is to establish
\eqref{eq: gcd large}. We will choose $\eps_0$ and $\eta$ according to how well $\gam$ can be approximated by rationals.

\subsubsection{Diophantine final shift}\label{subsection: Diophantine shifts}

Throughout the present subsubsection, the final shift $\gamma \in \bR$ is diophantine,
i.e. there exists $\lam \ge 2$ such that 
\begin{equation}\label{eq: non-Liouville gamma}
\| n \gamma \| \gg n^{-\lam/2} \qquad (n \in \bN).
\end{equation}
This is equivalent to $\gam$ being irrational and non-Liouville. We fix such a value of $\lam$ throughout the current subsubsection.

The approximation sets are as in 
\eqref{def: approximation sets for diophantine shifts},
however now we are more specific about $\eps_0$ and $\eta$. Let $\tilde \eps$ be the minimum of its values when Lemmata \ref{lem: inner} and \ref{lem: outer} are applied with $\vartheta = 20k$, and fix
a small positive real number $\eps_{0}$ such that
\begin{equation}\label{def: eps}
1000 k^2 \eps_{0} < \min\{ \lam^{-2}, \tilde{\eps}\}.
\end{equation}
In this diophantine case, the value of $\eta$ is of no importance, and we arbitrarily choose $\eta = 1/2$.\\

Recall that $D_{\gcd >}$ denotes the number of quadruples 
$(m,n,a,b) \in \bN^4$ satisfying
\[
\frac{a+\gam}{\un}, \frac{b+\gam}{\um}\in {\set  I},
\quad
n \in {\set B}_{\mathrm{loc}}(N; \brho) \cap {\set G}, 
\quad
m\in {\set B}_{\mathrm{loc}}(M; \bdel) 
\cap {\set G} \setminus \{ n \},
\]
as well as $d:= \gcd(\un,\um) > \max \{ 3\Del/\|q'_2 \gam \|, N_1 \}$, \eqref{eq: overlap inequality}, and \eqref{eq: shifty}.
Put $\un = dx$, $\um = dy$, and note that 
\eqref{eq: overlap inequality} entails 
\eqref{eq: third Bohr set} and hence
\begin{equation} \label{ReducedBohr}
\| (x-y) \gam \| < \Del/d \quad 
\mathrm{and} \quad |x-y| < \widehat{CN}/d.
\end{equation}

Combining the 
diophantine assumption \eqref{eq: non-Liouville gamma}
with \eqref{ReducedBohr} yields
\[
\Del / d \gg \vert x-y\vert^{-\lam} \gg (\uN/d)^{-\lam},
\]
and hence 
\begin{equation} \label{eq: d_upper} d \le \Del^{1/(1+\lam)} N^{\lam/(1+\lam)+o(1)}.
\end{equation}
We enlarge the Bohr set implicit in \eqref{ReducedBohr} to
\[
{\set B}'= \Big\{ u \le \frac{\widehat{CN}}{d}:  \| u\gam \| \le
\Big(\frac{\uN}{d}\Big)^{-1/\lam}
+ \frac{\Delta}{d} 
\Big\},
\]
observing that $u := |x-y| \in \cB'$. 

We claim that
\begin{equation} \label{eq: glean}
\# {\set B}' \le N^{o(1)}\Big( \Big(\frac{N}{d}\Big)^{1- \frac{1}{ \lam}} 
+ \frac{\Delta N}{d^2}\Big).
\end{equation}
This is clearly true if $d \gg \hat N$, so let us now assume that $d \le c \hat N$ for some small constant $c > 0$. By \eqref{eq: non-Liouville gamma}, we have
\[
\ome(\gam) \le \lam/2,
\]
so by Lemma \ref{BHV11} there exists $\ell \in \bN$ such that
\[
(\hat N/d)^{1/\lam} \le q'_\ell \le \hat N/d.
\]
Thus, we may apply Lemma \ref{lem: BHV size bound for Bohr sets} to the enlarged Bohr set $\cB'$, reaping \eqref{eq: glean}.

We begin by choosing $m \in {\set B}_{\mathrm{loc}}(M;\bdel)$, and by Lemma \ref{lem: size of localised Bohr sets} there are at most $O(\del M)$ choices. Next, we choose $d \mid m$ with $d > \max \{ 3\Del/\|q'_2 \gam \|, N_1\}$, and by the standard divisor function bound there are $M^{o(1)}$ possible choices of $d$.

Given an element $u$ from the Bohr set $\cB'$
as well as a choice of $y$, 
the value of $x$ is then determined in at most two ways.
Furthermore, we claim that for each choice
of $x,y$ the number of possible choices 
of $a,b$ is $O(d \mu (\cI))$. To see this, let $v$ be the integer closest to
$(x-y)\gam$, and note from \eqref{eq: third Bohr set} that $a,b$ must satisfy $ya-xb = v$.
All integer solutions to this linear diophantine equation have the form 
\[
a= a_0 + tx, \qquad b=b_0 +ty,
\]
for a specific solution $a_0,b_0$ and a parameter $t\in \bZ$. 
Therefore the number of such $a,b$ constrained by
$a+\gam \in \un \cI$ is 
\[
O\left(
\frac{\un\mu(\cI)}{x} + 1 \right) = O(d \mu (\cI)).
\]
Hence, by \eqref{eq: glean}, 
the number of choices for $u,a,b$ is
at most 
\begin{align*}
N^{o(1)} \mu({\set I})
(  d^{1/\lam} N^{1-1/\lam} 
+ \Del N /d )
\end{align*}
which, by \eqref{eq: d_upper} and the inequality $d > N_1$, is at most
\begin{align} \label{eq: bound on Bohr set in high gcd case}
N^{o(1)} \mu({\set  I})
(\Del^{1/(\lam + \lam^2)} 
N^{1/(1+\lam) +1 - 1/\lam}
+ \Del N/N_1).
\end{align} 

The contribution to $D_{\gcd >}$ coming from the second term
in \eqref{eq: bound on Bohr set in high gcd case} is at most
\begin{align*}
\del MN^{o(1)}\mu(\cI)  \Del N/ N_1
& \le \mu(\cI) \del MN^{1+o(1)-20k \eps_0} \Del
\leq \mu({\set  I}) \rho \del MN \Del.
\end{align*}
Finally, the contribution to $D_{\gcd >}$ corresponding 
to the first term on the right hand side of \eqref{eq: bound on Bohr set in high gcd case} 
is at most
$$
\del 
M \mu({\set I}) \Big(\frac{\Del}{N}\Big)^{\frac{1}{ \lam +  \lam^2}} N^{1+o(1)}.
$$ 
This quantity being at most $ \mu({\set I}) \rho \del MN \Del$ is equivalent to
\begin{equation}\label{eq: alternative inequality to verify for large gcd}
\rho \Del^{1- 1/( \lam +  \lam^2)}
N^{1/( \lam +  \lam^2)} \ge N^{o(1)}.
\end{equation}
To confirm \eqref{eq: alternative inequality to verify for large gcd}, we first recall from
 \eqref{eq: lower bound on Delta} that
$\Del > 1$. Next, the fact that  $\brho \in {\set W}(N)$, together with the bound \eqref{def: eps} on $\eps_0$, give
\[
\rho \ge N^{-5k \eps_0} \geq N^{-1/(9\lam^2)}.
\]
These considerations bestow \eqref{eq: alternative inequality to verify for large gcd} and thence \eqref{eq: gcd large}.\\

We now prove Theorems \ref{thm2} and \ref{thm: SpecialCaseWeak} in the remaining situation $\gam \in \bQ \cup {\set L}$, recalling that
it suffices to establish
\eqref{eq: gcd large}. The underpinning mechanisms are easier to grasp 
when $\gam$ is rational, so we begin with this case.

\subsubsection{Rational final shift}\label{subsection: rational shift Gallagher}

Throughout this subsubsection we assume that $\gam$ 
is a rational number, given in lowest terms by
$\gam = c_0/d_0$. We continue to use the notation introduced 
in the previous subsubsection,
and alert the reader to any deviation that we make.
Let $\tilde \eps$ be the minimum of its values when Lemmata \ref{lem: inner} and \ref{lem: outer} are applied with $\vartheta = 20k$, and fix
a small positive real number $\eps_{0} < \min \{ \tilde \eps, (99k)^{-1} \}$. 

In this case $\eta$ is again of little importance, and so we once more choose $\eta = 1/2$ arbitrarily. Observe that if $n$ is large then $c_t = c_0$ and $q'_t = d_0$, whereupon
${\set E}_n = {\set E}_n^{{\set I},\gam}$ is the set 
of $\alpha \in [0,1]$ for which
there is an integer $a$ satisfying
\begin{equation}\label{def: approximation sets rational case}
\vert \hat n \alpha - \gamma - a \vert < \Psi(n),
\quad 
\frac{a+\gamma}{\hat n}\in {\set I},
\quad 
\mathrm{and} 
\quad
(d_0a + c_0, \hat n) = 1.
\end{equation}

The overlap inequality \eqref{eq: overlap inequality} is equivalent to 
$$
\vert c_0 (\un - \um ) - d_0 (\um a - \un b)\vert \le d_0 \Del. 
$$
Dividing the above inequality by $d$ gives
\[
\vert c_0 (x - y ) - d_0 (y a - x b)\vert \le \frac{\Del d_0}{d}.
\]

Assume for a contradiction that $c_0 (x - y ) = d_0 (y a - x b)$. Then
\[
x(c_0 + d_0 b) = y (c_0 + d_0 a).
\]
As $n \ne m$, we have $\max \{x,y\} \ge 2$. Let us assume for simplicity that $x \ge 2$; a similar argument handles the case $y \ge 2$. Let $p$ be a prime divisor of $x$.
By Euclid's lemma $p$ divides $y$ or $c_0 + d_0 a$. 
The former is excluded by the coprimality of $x$ and $y$, and the 
latter by the shift-reduction. This contradiction means that
\[
1 \le \vert c_0 (x - y ) - d_0 (y a - x b)\vert \le \frac{\Del d_0}{d},
\]
so $d \le d_0 \Del$.

On the other hand, we have $d> \Del$ by assumption. So let us fix a non-zero integer $h\in[-d_0, d_0]$ and count integer
quadruples $(n,m,a,b)$ with $n\neq m$ satisfying  
\begin{equation}\label{eq: overlap rational shift fixed integ}
c_0 (x - y ) - d_0 (y a - x b) = h
\end{equation}
as well as the 
constraints \eqref{eq: constraints numer in overlap est}.
There are $O(\rho N \del M)$
many viable choices for the two integers $n\neq m$.
Then by \eqref{eq: overlap rational shift fixed integ}
there are $O(\mu({\set I})d)=O(\Del \mu(\cI))$ many choices
for the numerators $a,b$. Summing this bound over the $O(1)$ many choices for $h$
verifies \eqref{eq: gcd large}.\\

It remains to consider the case in which $\gam$
is a Liouville number. Informally,
this is a careful interpolation
between the diophantine 
and rational cases.

\subsubsection{Liouville final shift}\label{subsection: Liouville shift Gallagher}

Throughout the present subsubsection, we 
fix a Liouville number $\gam$. In a nutshell,
we need to find an way to balance between the regimes in which $\gam$ 
behaves like a rational number, and those in which $\gam$ 
behaves like a diophantine number.
This is incarnated in the definition
of the approximation sets, and shift-reduced fractions play a crucial role. 

Fix a non-empty interval ${\set I}$ in $[0,1]$, and a large positive constant
$C$. 
Further, let 
\begin{equation}\label{def: choice of eta}
\eta = 9(k-1)\eps_0.
\end{equation}
Here, as before, the positive constant $\eps_0$ is at most the lower value of $\tilde \eps$ when Lemmata \ref{lem: inner} and \ref{lem: outer} 
are applied with $\vartheta = 20k$, and moreover $\eps < (99k)^{-1}$. Recall that the approximation set
${\set E}_n = {\set E}_n^{{\set I},\gam}$
is the set of
$\alpha \in [0,1]$ for which
there exists $a \in \bZ$ satisfying 
\begin{equation}
\label{eq: nga}
\vert \hat n \alpha - \gamma - a \vert < \Psi(n)
\end{equation}
and
\begin{equation}\label{def: approximation sets Liouville case}
\frac{a+\gamma}{\hat n}\in {\set I},
\,\,
\mathrm{and\, the\, pair} 
\,\,
(a,\un)\,\, \mathrm{is}\,\,
(\gam,\eta)-\text{shift-reduced}.
\end{equation}
As in the previous subsubsections, we write $d= \gcd (\un, \um)$, $dx = \hat n$, $dy = \hat m$. 

\bigskip

Recall that $q'_t$ is the greatest continued fraction denominator of $\gam$ not exceeding $n^\eta$, and that $c_t$ is the corresponding numerator. We separate our argument according to the size of the subsequent denominator $q'_{t+1}$.

\underline{Case 1: $q'_{t+1}\geq 10\widehat{CN}/d$}

This inequality and the continued fraction approximation
\eqref{eq: standard estimate} yield 
$$
\left \vert \gam - \frac{\numgam}{\demgam} \right \vert
\leq \frac{d}{10 \demgam \widehat{CN}}.
$$
Moreover, the inequality \eqref{eq: overlap inequality}
can be written in the form 
\[
\left \vert  \un (b+ \gam) 
- \um (a+\gam)
\right \vert \leq \Del,
\]
or equivalently
$$
\vert  x (\demgam b+ \demgam\gam)
- y (\demgam a+\demgam \gam) \vert 
\leq \frac{\demgam \Del}{d}. 
$$
Note that 
\[
x \vert \demgam  \gam  -  \numgam \vert 
\leq x \frac d{10 \widehat{CN}} \leq \frac1{10},
\qquad
y \vert \demgam \gam - \numgam  \vert 
\leq 
\frac{\widehat{CM}}{d} \: \frac d{10 \widehat{CN}} \le \frac1{10},
\]
wherefore
\[
\vert x (\demgam b + \numgam)  - y (\demgam a + \numgam ) \vert 
\leq \frac{\demgam \Del}{d} + \frac{1}{5}.
\]

Recall from \eqref{eq: shifty} that the pairs $(a,\un)$ and $(b,\um)$ are $(\gam,\eta)$-shift-reduced, from which we now deduce that
$x (\demgam b + \numgam) - y (\demgam a + \numgam )$ is a non-zero integer.
Indeed, suppose for a contradiction that $x (\demgam b + \numgam) = y (\demgam a + \numgam )$. As $n \ne m$, we have $\max \{x,y\} \ge 2$. Let us assume for simplicity that $x \ge 2$; a similar argument handles the case $y \ge 2$. Let $p$ be a prime divisor of $x$.
Then $p$ divides $y$ or $q_t'a + c_t$. The former is excluded by the coprimality of $x$ and $y$, and the 
latter by the shift-reduction. This contradiction means that
\[
1\leq 
\vert x (\demgam b + \numgam) - y (\demgam a + \numgam ) \vert 
\leq \frac{\demgam \Del}{d} + \frac{1}{5},
\]
and consequently $d\leq 2 \demgam \Del$.
The upshot is that we are in the rather special scenario that 
$\Del \leq d \leq 2 \demgam \Del$. 

Setting
$h=  x (\demgam b + \numgam) - y (\demgam a + \numgam )$,
we see that there 
are $O(\demgam)$ many realisable values for the integer $h$. 
Now we are in the position to derive an acceptable
bound on the contribution from this case to $D_{\gcd >}$. There are $O(\del M)$
many options for $m$ and at most
$N^{o(1)}$ many ways to choose $d >N_1$ dividing $\um$.
It follows from
Lemmata \ref{lem: outer} and \ref{lem: divisible numbers in GAPS} 
that there are
$$ 
O\Big( \rho \widehat{CN} \Big( \frac{1}{d} + \frac{1}{N_1} \Big) \Big)
\le \frac{\rho N^{1+o(1)}}{N_1}  
$$ many choices for $n \in \cB_{\mathrm {loc}}(N;\brho)$ divisible by $d$.
For each choice of $h \ll q'_t$,
the parameters $a,b$ are determined up to 
$O(\mu({\set I})d)= O(\mu({\set I}) \demgam \Del)$ 
many possibilities. 
By summing over all 
choices of $h$, we obtain
\[
D_{\gcd >} = O\left(\del M  \frac{\rho N^{1+o(1)}}{N_1} (\demgam)^2 \mu({\set I}) \Del \right).
\]
Since $(\demgam)^2 N^{o(1)} \leq N^{2\eta + o(1)} < N^{20 (k-1) \eps_0}\leq N_1$,
this bound is acceptable.

\bigskip

\underline{Case 2: $q'_{t+1} < 10\widehat{CN}/d$}

By definition, we have
$q_{t+1}' > N^{\eta}$, and therefore
\begin{equation}\label{eq: range for d Liouville case}
d< 10 \widehat{CN} N^{-\eta}.
\end{equation}
Now, akin to the proof from Subsubsection \ref{subsection: Diophantine shifts},
we work with an enlarged Bohr set, namely
$$
{\set B} := \{u\leq 10 \widehat{C N}/d:
\Vert u \gam \Vert \le L \},
$$
where $L = L(d) = \max(\Del/d, N^{-\eta})$.

Note that $q'_{t+1} \in [1/(2L), 10 \widehat{CN}/d]$, where the lower bound comes from maximality in the definition of shift-reduction. The existence of a continued fraction denominator in this range is a key technical ingredient.
As $M$ is large and
$d > 3\Del/ \| q'_2 \gam \|$, we have
$L < \| q'_2 \gam \|/2$.
Therefore
Lemma \ref{lem: BHV size bound for Bohr sets} 
is applicable, and hence 
$$
\# {\set B} \ll \frac{N^{1+o(1)}}{d} L(d).
$$

First choose $d > \max(3\Del/ \| q'_2 \gam \|, N_1)$.
There are $O(\hat{M}/d)$ many conceivable options for $m$ divisible by $d$. Then $u = |x-y|$ lies in $\cB$, and therefore admits at most $\# \cB$ possibilities, and then
$x$ is determined up at most two choices. The number of viable choices 
of $a,b$ is then $O(d \mu ({\set I}))$. So, 
for a fixed choice of $d$, the number of valid
possibilities for $(m,n,a,b)$ is
at most
$$
\frac{M}{d} N^{o(1)}\frac{N}{d} L(d) d \mu ({\set I})
= N^{o(1)} \mu ({\set I}) MN \frac{L(d)}{d}.
$$
Upon summing over $d$, we infer that
$D_{\gcd > }$ is at most
\[
N^{o(1)} \mu ({\set I}) M N \sum_{\max(\Del, N_1)< d \le \widehat{CN}}\frac{L(d)}{d} .
\]

To conclude, it suffices to show that
for any $d$ in the range of interest we have
\begin{equation} \label{eq: WinnerWinnerChickenDinner}
N^{o(1)}L(d) \leq \Del \del \rho.
\end{equation}
If $L(d) = N^{-\eta}$ then
$$
\rho \del \geq N^{-4.2(k-1)\eps_0} N^{-4.2(k-1)\eps_0} 
= N^{-8.4 (k-1)\eps_0}  > N^{o(1)-\eta} = N^{o(1)} L(d).
$$
Since $\Del > 1$, from \eqref{eq: lower bound on Delta},
we conclude that $N^{o(1)}L(d) \leq \del \rho \Del$.
If on the other hand $L(d) = \Del/d$,
then
\[
d > N_1 \geq N^{20 k \eps_0}
\]
and therefore 
\[
L(d) N^{o(1)}/\Del = d^{-1}N^{o(1)} \leq N^{o(1) - 20 k \eps_{0}} \leq  \del \rho.
\]
We have \eqref{eq: WinnerWinnerChickenDinner} in both cases, completing the proofs of Theorems  \ref{thm2} and \ref{thm: SpecialCaseWeak}.

\subsection{A convergence statement}
\label{ConvergencePart}

In this subsection, we prove the convergence side of Corollary \ref{FIG}. Assume that
\begin{equation} \label{ConvergenceAssumption}
\sum_{n=1}^\infty \psi(n) (\log n)^{k-1} < \infty.
\end{equation}
Replacing $\psi(n)$ by $\max(\psi(n), n^{-2})$, we may suppose that 
\[
\psi(n) \ge n^{-2} \qquad (n \in \bN).
\]
We wish to show that $\mu_k(\cW^\times) = 0$. By 1-periodicity, we may assume that
\[
-1 < \gam_1,\ldots,\gam_k \le 0.
\]

We abbreviate $\tilde \balp = (\alp_1, \ldots, \alp_k)$ and $\balp = (\alp_1, \ldots, \alp_{k-1})$ for the remainder of this section.
Observe that $\cW^\times = \displaystyle \limsup_{n \to \infty} \cA_n$, where
\[
\cA_n = \{ \tilde \balp \in [0,1]^k:
\| n \alp_1 - \gam_1\| \cdots \| n\alp_k - \gam_k \| < \psi(n) \} \qquad (n \in \bN).
\]
For $\ba = (a_1,\ldots, a_k) \in \bZ^k$, let
\[
\cA_{n,\ba} = \left \{ \tilde \balp \in [0,1]^k:
| n \alp_k - \gam_k - a_k|
< \min (1, Q(\balp)) \right \},
\]
where
\[
Q(\balp) = Q(\balp; a_1,\ldots,a_{k-1})
= \frac{\psi(n)}
{\displaystyle \prod_{i \le k-1} | n \alp_i - \gam_i - a_i|}.
\]

\begin{lemma}
For $n \in \bN$, we have 
\[
\cA_n = \bigcup_{0 \le a_1, \ldots, a_k \le n+1}
\cA_{n,\ba}.
\]
\end{lemma}

\begin{proof} First suppose $\tilde \balp \in \cA_{n,\ba}$, for some $\ba \in \{0,1,\ldots,n+1\}^k$. Then
\[
\prod_{i \le k} \| n \alp_i - \gam_i \| \le \prod_{i \le k} | n \alp_i - \gam_i - a_i| < \psi(n),
\]
and so $\tilde \balp \in \cA_n$.
Therefore $\displaystyle \bigcup_{0 \le a_1, \ldots, a_k \le n+1}
\cA_{n,\ba} \subseteq \cA_n$.

Next, suppose $\tilde \balp \in \cA_n$, and for $i = 1,2,\ldots,k$ let $a_i$ be an integer for which
$\| n \alp_i - \gam_i \| = |n \alp_i - \gam_i - a_i|$. Now $|n \alp_k - \gam_k - a_k| < 1$ and $|n \alp_k - \gam_k - a_k| < Q(\balp)$. Moreover, the triangle inequality yields
\[
-1/2 \le -\gam_i - 1/2 \le a_i \le n - \gam_i + 1/2 \le n + 3/2 \qquad (1 \le i \le k),
\]
and therefore
$\cA_n \subseteq \displaystyle \bigcup_{0 \le a_1, \ldots, a_k \le n+1}
\cA_{n,\ba}$.
\end{proof}

By the union bound, if $n \in \bN$ then
\[
\mu_k(\cA_n) \le \sum_{a_1,\ldots,a_k = 0}^{n+1} \mu_k(\cA_{n,\ba}).
\]
Further, if $n \in \bN$ and $0 \le a_1,\ldots,a_k \le n+1$ then
\begin{align*}
\mu_k(\cA_{n,\ba})
&\ll n^{-1} 
\int_{[0,1]^{k-1}} \min (1, Q(\balp)) \d \balp \le n^{-1} (I_1 + I_2),
\end{align*}
where
\[
I_1 = \mu_{k-1}
(\cR), \qquad
\cR = \{ \balp \in [0,1]^{k-1}:
\displaystyle \min_{1\le i \le k-1} |n \alp_i - \gam_i - a_i| \le n^{-k} \},
\]
and
\[
I_2 =
\int_{ [0,1]^{k-1} \setminus \cR} Q(\balp) \d \balp.
\]
Here $\mu_{k-1}$ denotes $(k-1)$-dimensional Lebesgue measure. 

We have
$I_1 \ll n^{-1-k}$ by inspection. 
To estimate $I_2$, we cover
$
[0,1]^{k-1} \setminus \cR 
$
by $O((\log n)^{k-1})$ dyadically-restricted regions
\[
\{ \balp \in [0,1]^{k-1}:
 \del_i < |n \alp_i - \gam_i - a_i| \le 2 \del_i \: (1 \le i \le k-1) \}.
\]
The integral of $Q$ over such a region is $O(\psi(n)/n^{k-1})$, and so
\[
I_2 \ll \frac{\psi(n) (\log n)^{k-1}}{n^{k-1}}.
\]

Recalling that $\psi(n) \ge n^{-2}$, we therefore have
\[
\mu_k(\cA_{n,\ba}) \ll \frac{I_1 + I_2}n \ll \frac{\psi(n) (\log n)^{k-1}}{n^k}
\qquad (0 \le a_1,\ldots,a_k \le n+1),
\]
and finally
\[
\mu_k(\cA_n) \ll \psi(n) (\log n)^{k-1}.
\]
In view of \eqref{ConvergenceAssumption}, we now have
\[
\sum_{n=1}^\infty \mu_k(\cA_n) < \infty,
\]
and the first Borel--Cantelli lemma completes the proof that $\mu_k(\cW^\times) = 0$.

\section{Liouville fibres}\label{Sec: Gallagher on Liouville fibres}

In this section, we establish Theorem \ref{thm: Liouville}.

\subsection{A special case}

For expository purposes, we begin with the `partially homogeneous' case $\gam_2 =0$. Here it is simplest to invoke the resolution of the Duffin--Schaeffer conjecture by Koukoulopoulos and Maynard, though this is by no means an essential ingredient. By Theorem \ref{DS}, it suffices to show that
\[
\sum_{n =1}^\infty \frac{\varphi(n)}n \: \cdot \: \frac1 {n (\log n)^2 \| n \alp - \gam \|} = \infty,
\]
where $\alp = \alp_1$ and $\gam = \gam_1$.

By Lemma \ref{BHV11}, there are infinitely many positive integers such that $q_t > q_{t-1}^{9}$, where the $q_\ell$ are continued fraction denominators of $\alp$. Let $\cT$ be a sparse, infinite set of integers $t \ge 9$ for which $q_t > q_{t-1}^{9}$. 

Given $t \in \cT$, we begin by using Theorem \ref{LargestGap}, with 
\[
m = q_t = a_t q_{t-1} + q_{t-2} + 0,
\]
to find a small positive integer $b_t$ such that $\| b_t \alp - \gam \|$ is small. Recall that this concerns the gaps $d_{i+1} - d_i$, where
\[
\{ d_1, \ldots, d_m \} = \{ i \alp - \lfloor i \alp \rfloor: 0 \le i \le m \}
, \qquad
0 = d_0 < \cdots < d_{m+1} = 1.
\]
By \eqref{eq: standard estimate}, Theorem \ref{LargestGap} tells us that
\[
\max \{ d_{i+1} - d_i: 0 \le i \le m \} \le  |D_t| + |D_{t-1}| \le 2/q_t,
\]
where we have employed the standard notation \eqref{def: Dk}. For some $i \in \{ 0,1,\ldots,m \}$, the fractional part of $\gam$ lies in $[d_i, d_{i+1}]$, and for some $b_t \in \{1,2,\ldots,m \}$ the fractional part of $b_t \alp$ is either $d_i$ or $d_{i+1}$. The upshot is that we have a positive integer $b_t \le q_t$ such that $\| b_t \alp - \gam \| \le 2/q_t$.

For $t \in \cT$, denote by $\cD_t$ the set of integers of the form
\[
n = b_t + q_{t-1}x + q_t y,
\]
where $x,y \in \bZ$ satisfy 
\[
q_t^{1/4} \le y \le q_t^{1/3},
\qquad
q_t^{2/3} \le x \le q_t^{3/4}.
\]
If $n$ is as above then
\[
n \asymp q_t y, \qquad \log n \asymp \log q_t, \qquad \| n \alp - \gam \| \asymp \frac {x}{q_t}.
\]
Indeed, for the final estimate, observe using \eqref{eq: standard estimate} that
\[
\| b_t \alp - \gam \| \le \frac{2}{q_t},
\qquad 
\| x q_{t-1} \alp \| \asymp \frac x {q_t}, 
\qquad 
\| y q_t \alp \| \ll \frac {y}{q_t},
\]
and apply the triangle inequality. The sets $\cD_t$ ($t \in \cT$) are disjoint, since $\cT$ is sparse. Moreover, for $n \in \cT$ the representation above is unique, since $q_{t-1} q_t^{3/4} < q_t$. Fix $t \in \cT$, and let
\[
S_t = \sum_{n \in \cD_t} \frac{\varphi(n)}{n} v_n,
\]
where $v_n^{-1} =  n(\log n)^2 \| n \alp - \gam \|$. As $\cT$ is infinite, it suffices to prove that $S_t \gg 1$. Note that in this section our implicit constants do not depend on $t$.

As a first step, we compute that
\[
\sum_{n \in \cD_t} v_n \gg (\log q_t)^{-2} \sum_{\substack{q_t^{1/4} \le y \le q_t^{1/3} \\
q_t^{2/3} \le x \le q_t^{3/4}}}  (xy)^{-1} \gg 1.
\]
Let $C_t = \left( \sum_{n \in \cD_t} v_n \right)^{-1}$, and for $n \in \cD_t$ let $w_n = C_t v_n \ll v_n$, so that we have $\sum_{n \in \cD_t} w_n = 1$. The weighted AM--GM inequality gives
\[
S_t \gg \sum_{n \in \cD_t} w_n \varphi(n)/n \ge \prod_p (1-1/p)^{\tau_p},
\]
where
\[
\tau_p = \sum_{\substack{n \in \cD_t \\ n\equiv 0 \mmod p}} w_n.
\]
Now
\[
-\ln S_t \le O(1) - \sum_p \tau_p \ln(1-1/p) \ll 1 + \sum_p \tau_p/p, 
\]
so it remains to show that
\begin{equation} \label{eq: SumOverPrimes}
\sum_p \tau_p/p < \infty.
\end{equation}

Let $p$ be prime, and let $X,Y$ be parameters in the ranges
\[
q_t^{1/4} \le Y \le q_t^{1/3}/2,
\qquad
q_t^{2/3} \le X \le q_t^{3/4}/2.
\]
Denote by $N_p(X,Y)$ the number of integer solutions $(x,y) \in [X,2X] \times [Y,2Y]$ to
\[
b + q_{t-1}x + q_t y \equiv 0 \mmod p.
\]
Let us assume that this congruence has a solution in $[X,2X] \times [Y,2Y]$, forcing $p \le q_t^2$. Then, by Lemma \ref{lem: divisible numbers in GAPS}, we have
\begin{equation} \label{eq: NpBound}
N_p(X,Y) \ll XY/p + X + Y \ll XY p^{-1/8}.
\end{equation}
If 
\[
X \le x \le 2X, 
\qquad Y \le y \le 2Y,
\qquad n = b + q_{t-1} x + q_t y
\]
then $w_n^{-1} \asymp XY (\log q_t)^2$. Using \eqref{eq: NpBound}, and summing over $X,Y$ that are powers of $2$ or the endpoints of their allowed ranges, gives
\[
\tau_p \ll p^{-1/8},
\]
and in particular \eqref{eq: SumOverPrimes}.

\subsection{Diophantine second shift}

In this subsection, we prove Theorem \ref{thm: Liouville} in the case that $\gam_2$ is diophantine. Let $\lam \ge 2$ satisfy
\begin{equation} \label{eq: non-Liouville gamma2}
\| n \gam_2 \| \gg n^{-\lam/2} \qquad (n \in \bN).
\end{equation}
Then $\lam \ge 2 \ome(\gam_2)$, where
\[
\ome(\gam_2) = \sup \{ w > 0: \exists^\infty q \in \bN \: \: \| q \gam_2 \| < q^{-w} \}.
\]
Fix a constant $c_1 > 0$, and for $n \in \bN$ let 
\[
\Psi(n) = \frac{c_1}{n (\log n)^2 \| n \alp_1 - \gam_1 \|} \in (0, +\infty],
\]
recalling our notational conventions from Section \ref{OrgNot}. By $1$-periodicity of $\| \cdot \|$, our task is to show that for almost all $\alp_2 \in [0,1]$ the inequality
\[
\| n \alp_2 - \gam_2 \| < \Psi(n)
\]
has infinitely many solutions $n \in \bN$. Let $\cI$ be a non-empty subinterval of $[0,1]$. Let $\tet_3,\tet_4$ be constants satisfying
\[
\tet_3 = 1 - \lam^{-3} < \tet_4 < 1.
\]
Let $T_0$ be large in terms of $\cI$, and let $\cT$ be a sparse, infinite set of integers $t \ge T_0$ for which $q_t^{1-\tet_4} > q_{t-1}$.

For $t \in \cT$, let $\cD_t$ be the set of integers of the form
\begin{equation} \label{Repn}
n = b_t + q_{t-1} x_1 + q_t y_1
\end{equation}
with
\[
q_t^{1/4} \le y_1 \le q_t^{1/3}, \qquad q_t^{\tet_3} \le x_1 \le q_t^{\tet_4},
\]
where $b_t$, $q_{t-1}$, and $q_t$ are as in the previous subsection.
Let $t \in \cT$, and let $n \in \cD_t$. The representation above is unique, since $q_t > q_t^{\tet_4} q_{t-1}$.
Moreover, we have
\[
\| n \alp_1 - \gam_1\| \ge \| x_1 q_{t-1} \alp_1 \|- \frac{2}{q_t} - y_1 \| q_t \alp_1 \|
\ge \frac{x_1}{2q_{t}} - \frac2{q_t} - \frac{y_1}{q_{t+1}} \ge \frac{x_1}{3q_t},
\]
so
\[
n \asymp q_t y_1, \qquad \log n \asymp \log q_t,
\qquad \| n \alpha_1 - \gamma_1 \| \asymp x_1/q_t,
\]
and in particular
\[
\Psi(n)  \asymp  \frac{q_t}{x_1} (q_t y_1)^{-1} (\log q_t)^{-2} \asymp \frac1{x_1 y_1 (\log q_t)^2}.
\]

For $n \in \bN$ and $a \in \bZ$, define
\[
\cA_{n,a} = \{ \bet \in \cI: | n \bet - \gam_2 -a| < \Psi(n) \}, \qquad \cA_n = \bigcup_{a \in \bZ} \cA_{n,a},
\]
and observe that if $n \mu(\cI) \ge 1$ then
\[
\mu(\cA_n) \asymp \mu(\cI) \Psi(n).
\]
Let $F(1), F(2), \ldots$ be a sequence of powers of 4, non-increasing, satisfying $F(t) \le \log \log t$ for all $t$, and such that
\[
\sum_{t \in \cT} F(t)^{-1} = \infty.
\]

For $t \in \cT$, let
\[
\cG_t = \{ n \in \cD_t: n \equiv 0 \mmod F(t) \}.
\]
We will show that if $t,s \in \cT$ then
\begin{equation} \label{eq: LN1}
\sum_{n \in \cG_t} \mu(\cA_n) \gg \frac{\mu(\cI)}{F(t)} 
\end{equation}
and
\begin{equation} \label{eq: LN2}
\sum_{n \in \cG_t} \sum_{\substack{m \in \cG_s \\ m < n}} \mu(\cA_n \cap \cA_m) \ll \frac{\mu(\cI)}{F(t)F(s)}.
\end{equation}
Then, with
\[
\cX:= \bigcup_{t \in \cT} \cG_t, \qquad \cX_N := \bigcup_{\substack{t \in \cT \\ t \le N}} \cG_t,
\]
we would have
\[
\sum_{n \in \cX} \mu(\cA_n) = \infty
\]
and
\begin{align*}
\mu(\cI) \sum_{n,m \in \cX_N} \mu(\cA_n \cap \cA_m) &\ll \left(
\sum_{n \in \cX_N} \mu(\cA_n) \right)^2 + \sum_{n \in \cX_N} \mu(\cA_n) \\
&\ll
\left(\sum_{n \in \cX_N} \mu(\cA_n)\right)^2.
\end{align*}
At that stage
Lemmata \ref{lem: Borel-Cantelli} and \ref{lem: density} would give 
\[
\mu(\limsup_{n \in \cX} \{ \bet \in [0,1]: \| n \bet - \gam_2 \| < \Psi(n) \}) = 1,
\]
which would finish the proof. The upshot is that it remains to establish \eqref{eq: LN1} and \eqref{eq: LN2}.

Let $t \in \cT$, and let $X$ and $Y$ be parameters in the ranges
\begin{equation} \label{XYranges}
q_t^{\tet_3} \le X \le q_t^{\tet_4}/2,
\qquad q_t^{1/4} \le Y \le q_t^{1/3}/2.
\end{equation}
Then
\begin{align*}
\mu(\cI)^{-1} \sum_{\substack{n = b_t + q_{t-1}x_1 + q_t y_1 \\ n \equiv 0 \mmod F(t) \\ X < x_1 \le 2X \\ Y < y_1 \le 2Y}} \mu(\cA_n) 
&\gg (\log q_t)^{-2} (XY)^{-1} \sum_{\substack{
X < x_1 \le 2X\\
Y < y_1 \le 2Y
\\ b_t + q_{t-1}x_1 + q_ty_1 \equiv 0 \mmod F(t)}} 1 \\
& \gg (\log q_t)^{-2} F(t)^{-1},
\end{align*}
by Lemma \ref{lem: divisible numbers in GAPS}. Summing over $X$ and $Y$ that are powers of $2$ in the above ranges, we obtain
\eqref{eq: LN1}.\\

It remains to prove \eqref{eq: LN2}. Writing 
\begin{equation} \label{2d}
n = b_t + q_{t-1} x_1 + q_t y_1, \qquad m =  b_s + q_{s-1} x_2 + q_s y_2,
\end{equation}
where
\begin{equation} \label{xyranges}
q_t^{1/4} \le y_1 \le q_t^{1/3}, \quad 
q_t^{\tet_3} \le x_1 \le q_t^{\tet_4},
\quad
q_s^{1/4} \le y_2 \le q_s^{1/3},
\quad
q_s^{\tet_3} \le x_2 \le q_s^{\tet_4},
\end{equation}
we have
\begin{align*}
\mu(\cA_{n,a} \cap \cA_{m,b}) &\ll \min \left\{ \frac{\Psi(n)}n, \frac{\Psi(m)}m \right\} \\
&\asymp \min \left\{ \frac{q_t}{ x_1 (q_t y_1 \log q_t)^2}, \frac{q_s}{ x_2 (q_s y_2 \log q_s)^2} \right\}
\end{align*}
for $a,b \in \bZ$.

Let $X_1, Y_1, X_2, Y_2$ be parameters in the ranges
\begin{align*}
&q_t^{1/4} \le Y_1 \le q_t^{1/3}/2, \qquad
&q_t^{\tet_3} \le X_1 \le q_t^{\tet_4}/2, \\
 &q_s^{1/4} \le Y_2 \le q_s^{1/3}/2, \qquad
&q_s^{\tet_3} \le X_2 \le q_s^{\tet_4}/2.
\end{align*}
Suppose $\cA_{n,a} \cap \cA_{m,b}$ is non-empty, and that we have \eqref{2d} and
\begin{equation} \label{eq: dyadic12}
X_i \le x_i \le 2X_i, \quad Y_i \le y_i \le 2Y_i \qquad (i=1,2).
\end{equation}
Then
\begin{equation} \label{d1}
\dist(a+\gam_2, n \cI) < 1/2, \qquad \dist(b + \gam_2, m \cI) < 1/2,
\end{equation}
and there exists $\bet \in \cA_{n,a} \cap \cA_{m,b}$, so that
\[
| n \bet - \gam_2 - a| < \Psi(n), \qquad | m \bet - \gam_2 - b| < \Psi(m).
\]
The triangle inequality gives
\begin{equation} \label{d2}
|(n-m) \gam_2 - (ma - nb)| < m \Psi(n) + n \Psi(m) < \Del,
\end{equation}
where
\[
\Del \asymp \frac{q_s Y_2 }{X_1 Y_1 (\log q_t)^2}
+ \frac{q_t Y_1 }{X_2 Y_2 (\log q_s)^2}.
\]
Note that
\begin{equation} \label{BoundOnDelta}
\Del^2 \gg \frac{q_t q_s} {X_1 X_2 (\log q_t)^2 (\log q_s)^2}
\ge q_t^{1-\tet_4 - o(1)} q_s^{1-\tet_4-o(1)}.
\end{equation}

Given $t,s \in \cT$, denote by $N(t,s)$ the number of solutions 
\[
(n,m,a,b) \in \cG_t \times \cG_s \times \bZ^2
\]
to the diophantine system given by \eqref{d1}, \eqref{d2} and $n > m$ for which we have \eqref{2d} for some $x_1,x_2,y_1,y_2$ in the ranges \eqref{eq: dyadic12}. Our goal is to show that 
\begin{equation} \label{eq: Nts}
N(t,s) \ll X_1 X_2 Y_1 Y_2 \Del \mu(\cI) F(t)^{-1} F(s)^{-1}.
\end{equation}
Note that the number of solutions with $n=m$ is at most $X_1Y_1$, which is negligible. Assuming for the time being that we can achieve \eqref{eq: Nts}, the contribution to $\sum_{n \in \cD_t} \sum_{m \in \cD_s} \mu(\cA_n \cap \cA_m)$ from $x_1, x_2, y_1, y_2$ in these dyadic ranges is 
\begin{align*}
O\left(\frac{X_1 X_2 Y_1 Y_2 \Del \mu(\cI)}{ F(t)F(s)}
 \min \left\{ \frac{q_t}{X_1 (q_t Y_1 \log q_t)^2}, \frac{q_s}{X_2 (q_s Y_2 \log q_s)^2} \right\} \right),
\end{align*}
which is 
\[
O\Bigl( \frac{\mu (\cI)}{F(t)F(s) (\log q_t)^2 (\log q_s)^2} \Bigr).
\]
Summing over $X_1, X_2, Y_1, Y_2$ that are powers of $2$ or endpoints of the prescribed ranges, we would thereby deduce \eqref{eq: LN2}. The upshot is that it remains to prove \eqref{eq: Nts}.

We partition our solutions according to the value of $d = \gcd(m,n)$, and write $n=dx$, $m = dy$, so that $\gcd(x,y) = 1$ and $x > y$. Then \eqref{d2} becomes
\begin{equation} \label{d3}
|(x-y) \gam_2 - (ya - xb)| < \Del/d.
\end{equation}
Note that $d \ge \min \{F(t), F(s) \} > \mu(\cI)^{-1}$.
Let $(q'_\ell)_\ell$ be the sequence of continued fraction denominators of $\gam_2$.

\underline{Case 1: $\mu(\cI)^{-1} < d < 3 \Del / \| q'_2 \gam_2 \|$}

Choose $m,n$ with $\gcd(m,n) < 3 \Del / \| q'_2 \gam_2 \|$ in $O(F(t)^{-1} F(s)^{-1} X_1 X_2 Y_1 Y_2)$ ways, via Lemma \ref{lem: divisible numbers in GAPS}. Then there are $O(\Del/d)$ possibilities for
$h = ya - xb$, by \eqref{d3}. Given $h$, the value of $a$ is determined modulo $x$, and lies in an interval of length $O(dx \mu(\cI))$, so there are $O(d \mu(\cI) + 1)$ possibilities for $a$, whereupon $b$ is determined. As $d > \mu(\cI)^{-1}$, the contribution to $N(t,s)$ from this case is $O(X_1 X_2 Y_1 Y_2 \Del \mu(\cI) F(t)^{-1} F(s)^{-1})$.

\underline{Case 2: $3 \Del / \| q'_2 \gam_2 \| \le d < Y_1$} 

Choose $m$ in $O(X_2Y_2)$ ways, and then choose $d \mid m$ with $d \in [3 \Del / \| q'_2 \gam_2 \|, Y_1)$ in $Y_2^{o(1)}$ ways. Then, by Lemma \ref{lem: divisible numbers in GAPS},  choose $n \in \cD_t$ such that
\[
n \equiv 0 \mmod d
\]
in $O(X_1 Y_1 /d)$ ways. Finally there are $d \mu (\cI)$ possibilities for $a$ and $b$ so, using \eqref{BoundOnDelta}, the contribution to $N(t,s)$ from this case is bounded above by
\[
X_2 Y_2^{1+o(1)} X_1 Y_1 \mu(\cI) \ll X_1 Y_1 X_2 Y_2 \Del \mu (\cI) F(t)^{-1} F(s)^{-1}.
\]

\underline{Case 3: $d \ge \max\{ Y_1, 3 \Del / \| q'_2 \gam_2 \| \}$}

Choose $m$ in $O(X_2Y_2)$ ways and $d \ge \max\{ Y_1, 3 \Del / \| q'_2 \gam_2 \| \}$ dividing $m$ in $q_t^{o(1)}$ ways. Put
\[
N = q_t Y_1.
\]
Set $x-y = u \in \bN$ and $ya - xb = v \in \bZ$, so that
\begin{equation} \label{uv}
u \le 3N/d, \qquad | u \gamma_2 - v| < \Del/d.
\end{equation}
From our choice of $\lam$, we have
\[
\Del/d \gg (N/d)^{-\lam},
\]
so $d \ll (\Del N^{\lam})^{1/(1+\lam)}$. We relax the inequalities above, giving
\begin{equation} \label{relax}
u \le CN/d, \qquad \| u \gamma_2 \| \le (N/d)^{-1/\lam} + \Del/d,
\end{equation}
where $C$ is a large, positive constant.

We claim that
\eqref{relax} has
$O((N/d)((N/d)^{-1/\lam}+\Del/d))$ solutions $u \in \bN$. This is clearly true if $d \gg N$, so let us now assume that $d \le c N$ for some small constant $c > 0$. By \eqref{eq: non-Liouville gamma2}, we have $\ome(\gam_2) \le \lam/2$,
so by Lemma \ref{BHV11} there exists $\ell \in \bN$ such that
\[
(N/d)^{1/\lam} \le q'_\ell \le CN/d.
\]
Thus, we may apply Lemma \ref{lem: BHV size bound for Bohr sets} to the Bohr set defined by \eqref{relax},
establishing the claim.

Consequently, there are $O((N/d)((N/d)^{-1/\lam} + \Del/d))$ pairs $(u,v) \in \bN \times \bZ$ satisfying \eqref{uv}. Hence, given $m$ and $d$ as above, the number of possibilities for $u,v,a,b$ is at most a constant times
\begin{align*}
(N/d)((N/d)^{-1/\lam} + \Del/d) d 
&= N^{1-(1/\lam)} d^{1/\lam}  + N \Del/d \\
&\ll \Del^{1/(\lam+\lam^2)} N^{1-(1/\lam) + 1/(1+\lam)} + N \Del / Y_1.
\end{align*}
Recalling that 
\[
N = q_t Y_1, \qquad X_1 \ge q_t^{\tet_3} = q_t^{1-\lam^{-3}},
\qquad Y_1 \ge q_t^{1/4},
\]
we find that the contribution to $N(t,s)$ from this case is $O(N_1(t,s) + N_2(t,s))$, where
\begin{align*}
N_1(t,s) &\ll  X_2 Y_2 q_t^{o(1)} (\Del / N)^{1/(\lam + \lam^2)} N
\ll X_2 Y_2 \Del q_t^{1 - 1/(\lam+\lam^2) + o(1)} Y_1 \\
& \ll X_1 Y_1 X_2 Y_2 \Del \mu (\cI) F(t)^{-1} F(s)^{-1}
\end{align*}
and
\begin{align*}
N_2(t,s) &\ll X_2 Y_2 q_t^{o(1)} N \Del / Y_1 = X_2 Y_2 q_t^{1+o(1)}  \Del  \\
&\ll X_1 Y_1 X_2 Y_2 \Del \mu (\cI) F(t)^{-1} F(s)^{-1}.
\end{align*}

We have considered all possible size ranges for $d$, confirming \eqref{eq: Nts}. We conclude that Theorem \ref{thm: Liouville} holds in the case that $\gam_2$ is diophantine.

\subsection{Liouville second shift}

Next, we prove Theorem \ref{thm: Liouville} in the case that $\gam_2 \in \cL$. This is a more sophisticated variant of the proof given in the previous subsection that works whenever $\gam_2 \in \bR \setminus \bQ$. As we discuss in the next subsection, a simpler version of it works when $\gam_2 \in \bQ.$

Fix a constant $c_1 > 0$, and for $n \in \bN$ let 
\[
\Psi(n) = \frac{c_1}{n (\log n)^2 \| n \alp_1 - \gam_1 \|} \in (0, +\infty].
\]
By $1$-periodicity of $\| \cdot \|$, it suffices to show that for almost all $\alp_2 \in [0,1]$ the inequality
\[
\| n \alp_2 - \gam_2 \| < \Psi(n)
\]
has infinitely many solutions $n \in \bN$. Indeed, the latter would imply that
\[
\liminf_{n \to \infty} n(\log n)^2 \| n \alp_1 - \gam_1 \| \cdot \| n \alp_2 - \gam_2 \| \le c_1,
\]
and $c_1 > 0$ is arbitrary.

Let $\cI$ be a non-empty subinterval of $[0,1]$, let $T_0$ be large in terms of $\cI$, and let $\cT$ be a sparse, infinite set of integers $t \ge T_0$ for which $q_t > q_{t-1}^9$, where again $q_1, q_2,\ldots$ are the denominators of the continued fraction convergents to the Liouville number $\alp_1$. 

Fix $\lam \ge 2$. For $t \in \cT$, we define $\cD_t$ as in the previous subsection. The sets $\cD_t$ ($t \in \cT$) are disjoint, because $\cT$ is sparse.
For $t \in \cT$ and $n \in \cD_t$, the representation \eqref{Repn} is unique, and moreover
\[
n \asymp q_t y_1, \qquad \log n \asymp \log q_t,
\qquad \| n \alpha_1 - \gamma_1 \| \asymp \frac{x_1}{q_t}, \qquad 
\Psi(n)  \asymp  \frac1{x_1 y_1 (\log q_t)^2}.
\]
We also define $\cG_t$ via $\cD_t$ as in the previous subsection, for $t \in \cT$, using the arithmetic function $F$.

For $t \in \cT$, let $c_{t'}/q'_{t'}$ be the continued fraction convergent to $\gam_2$ for which $q'_{t'} < q_t$ is maximal.  For $n \in \cD_t$ and $a \in \bZ$, we again define
\[
\cA_{n,a} =
\{ \bet \in \cI: | n \bet - \gam_2 - a| < \Psi(n) \},
\]
but now we let $\cA_n$ be the union of $\cA_{n,a}$ over integers $a$ for which 
\[
(q'_{t'} a + c_{t'}, n) = 1.
\]
As in the previous subsection, it remains to establish \eqref{eq: LN1} and \eqref{eq: LN2}.

\begin{lemma} Let $n \in \cD_t$. Then
\[
\frac{\varphi(n)}n \mu(\cI) \Psi(n) \ll
\mu(\cA_n) \ll \mu(\cI) \Psi(n).
\]
\end{lemma}

\begin{proof} For the upper bound, observe that there are $O(n \mu(\cI))$ integers $a$ for which $\cA_{n,a}$ is non-empty, and for each of these $\mu(\cA_{n,a}) \ll \Psi(n)/n$. 
For the lower bound, 
it suffices to show that 
\[
\# \{ a \in n \cI: (q'_{t'} a + c_{t'}, n) = 1 \} \gg \frac{\varphi(n)}n \mu(\cI).
\]
This follows routinely from the fundamental lemma of sieve theory, in the same way as \eqref{eq: lower bound of fake phi}.
\end{proof}

Let $t \in \cT$, and let $X$ and $Y$ be parameters in the ranges
\eqref{XYranges}.
Then
\begin{align*}
&\mu(\cI)^{-1} \sum_{\substack{n = b_t + q_{t-1}x_1 + q_t y_1 \\ n \equiv 0 \mmod F(t) \\ X < x_1 \le 2X \\ Y < y_1 \le 2X}} \mu(\cA_n) 
\\
&\gg (\log q_t)^{-2} (XY)^{-1} \sum_{\substack{
X < x_1 \le 2X\\
Y < y_1 \le 2Y
\\ b_t + q_{t-1}x_1 + q_t y_1 \equiv 0 \mmod F(t)}} 
\frac{\varphi(b_t + q_{t-1}x_1 + q_t y_1)} {b_t + q_{t-1}x_1 + q_t y_1}.
\end{align*}

We claim that
\[
S :=
\sum_{\substack{
X < x_1 \le 2X\\
Y < y_1 \le 2Y
\\ b_t + q_{t-1}x_1 + q_ty_1 \equiv 0 \mmod F(t)}} 
\frac{\varphi(b_t + q_{t-1}x_1 + q_t y_1)} {b_t + q_{t-1}x_1 + q_t y_1} \gg XY/F(t).
\]
To show this, we write
\[
\cU = 
\{ n = b_t + q_{t-1}x_1 + q_t y: X < x_1 \le 2X, \: Y < y_1 \le 2Y, \: n \equiv 0 \mmod F(t) \},
\]
and apply the AM--GM inequality to give
\[
S \ge |\cU| \left(\prod_{n \in \cU} \frac{\varphi(n)}n \right)^{1/|\cU|}
= |\cU| \prod_p \prod_{\substack{n \in \cU \\ n \equiv 0\mmod p}} (1-1/p)^{1/|\cU|}.
\]
By Lemma \ref{lem: divisible numbers in GAPS}, we have
\[
|\cU| \gg XY/F(t),
\]
so for the claim it suffices to show that
\[
S' := \prod_p 
\prod_{\substack{n \in \cU \\ n \equiv 0\mmod p}}
(1-1/p)^{1/|\cU|} \gg 1.
\]
Next, observe that
\[
-\ln(S') = -\sum_p \tau_p \ln(1-1/p) \ll \sum_p \tau_p/p,
\]
where
\[
\tau_p = |\cU|^{-1} \# \{ n \in \cU: n \equiv 0 \mmod p \},
\]
so for the claim it remains to prove that
\begin{equation} \label{LLdensity}
\sum_{p \ge 3} \tau_p/p \ll 1.
\end{equation}
Let $p \ge 3$, and note that $p \nmid F(t)$ because $F(t)$ is a power of $4$. As $\tau_p = 0$ for $p > q_t^2$, let us also assume that $p \le q_t^2$. By Lemma \ref{lem: divisible numbers in GAPS}, we have
\[
\tau_p \ll F(t) \left(\frac1{pF(t)} + \frac1X + \frac1Y\right) \ll \frac1p + Y^{-1/2}
\ll p^{-1/16},
\]
giving \eqref{LLdensity}.

Thus, we have the claim, and so
\[
\mu(\cI)^{-1} \sum_{\substack{n = b_t + q_{t-1}x_1 + q_t y_1 \\ n \equiv 0 \mmod F(t) \\ X < x_1 \le 2X \\ Y < y_1 \le 2X}} \mu(\cA_n) \gg (\log q_t)^{-2} F(t)^{-1}.
\]
Summing over $X$ and $Y$ that are powers of $2$ in the ranges \eqref{XYranges}, we obtain
\eqref{eq: LN1}.

It remains to prove \eqref{eq: LN2}. To this end, it again suffices to prove \eqref{eq: Nts}, where $X_1, Y_1, X_2, Y_2$ are parameters in the ranges
\begin{align*}
&q_t^{1/4} \le Y_1 \le q_t^{1/3}/2, \qquad
&q_t^{\tet_3} \le X_1 \le q_t^{\tet_4}/2, \\
 &q_s^{1/4} \le Y_2 \le q_s^{1/3}/2, \qquad
&q_s^{\tet_3} \le X_2 \le q_s^{\tet_4}/2,
\end{align*}
but now in the count $N(t,s)$ we impose the additional restrictions
\begin{equation} \label{eq: thrust}
(q'_{t'} a + c_{t'}, n) = 1, \qquad (q'_{s'} b + c_{s'}, m) = 1.
\end{equation}
Cases 1 and 2 from the previous subsection are unaffected, so our task is to count solutions for which
\[
d = (m,n) \ge \max\{ Y_1, 3 \Del / \| q'_2 \gam_2 \| \}
\qquad (\text{Case }3).
\]
As $\cT$ is sparse, this is only possible if
\[
s=t.
\]
Put $N = q_t Y_1.$ Let us again write $n = dx$ and $m=dy$, so that $x > y$ and $(x,y)=1.$ Let $C$ be a large, positive constant.

\underline{Case 3a: $\gam_2$ has a continued fraction denominator in $[(N/d)^{1/\lam}, CN/d]$, or $d \ge N$} 

In the case the proof from the previous subsection carries through, for in this case we may apply Lemma \ref{lem: BHV size bound for Bohr sets} therein.

\underline{Case 3b: $\gam_2$ has no continued fraction denominator in $[(N/d)^{1/\lam}, CN/d]$, and $d < N$}

In this case, as
\[
q'_{t'+1} \ge q_t \ge q_t Y_1/d = N/d > (N/d)^{1/\lam},
\]
we must have
\[
q'_{t'+1} > CN/d,
\]
where $q'_{t'+1}$ is the continued fraction denominator of $\gam_2$ subsequent to $q'_{t'}$. Therefore
\[
|q'_{t'} \gam_2 - c_{t'}| < (q'_{t'+1})^{-1} < d/(CN).
\]
As
\[
|(n-m) q'_{t'} \gam_2 - q'_{t'} (ma-nb)| < q'_{t'} \Del,
\]
the triangle inequality now confers
\[
|(n-m) c_{t'} - q'_{t'} (ma - nb)| < q'_{t'} \Del + d/2.
\]
We thus have
\[
1 \le |x(q'_{t'}b + c_{t'}) - y(q'_{t'} a + c_{t'})| < q'_{t'} \Del/d + 1/2,
\]
owing to the coprimality restrictions \eqref{eq: thrust}
that we have thrust upon the problem. Whence
\begin{equation} \label{CongruenceConstraint}
d < 2q'_{t'} \Del, 
\qquad
c_{t'}(x-y) \equiv O(q'_{t'} \Del/d) \mmod q'_{t'}.
\end{equation}

Recall that in this case we have $s=t.$
We begin our count by choosing $d$ so that
\[
\max\{ Y_1, 3 \Del / \| q'_2 \gam_2 \| \} \le d < 2q'_{t'} \Del.
\]
Next, we choose $x$ in $O(q_t Y_1 /d)$ ways. Then
\[
y \ll q_t Y_2/d
\]
lies in one of $O(q'_{t'} \Del/d)$ residue classes modulo $q'_{t'}$, according to \eqref{CongruenceConstraint}, so there are
\[
O \left(
\frac{q'_{t'} \Del}d \left(\frac{q_t Y_2}{d q'_{t'}} + 1 \right)
\right)
\]
possibilities for $y$. After that, the variable $a$ is then determined modulo $x$ from \eqref{d3}, and so there are at most $d \mu(\cI) \le d$ possibilities for $a$ and then finally $b$ is uniquely determined. 
Recalling that $X_1, X_2 \ge q_t^{\tet_3} = q_t^{1- \lam^{-3}}$ and $Y_1, Y_2 \ge q_t^{1/4}$, our total count from this case is at most a constant times
\[
\sum_{\max\{ Y_1, 3 \Del / \| q'_2 \gam_2 \| \} \le d < 2q'_{t'} \Del} \frac{q_t Y_1 q'_{t'} \Del }{d}\left(\frac{q_t Y_2}{d q'_{t'}} + 1 \right) \le N_1(t,s) + N_2(t,s),
\]
where
\[
N_1(t,s) = \sum_{d > Y_1} \frac{q_t^2 Y_1 Y_2  \Del }{d^2} \ll q_t^2  \Del Y_2 \ll
\frac{X_1 Y_1 X_2 Y_2 \Del \mu (\cI)}{F(t) F(s)}
\]
and
\begin{align*}
N_2(t,s) &=
\sum_{d < 2 q'_{t'} \Del} \frac{q_t Y_1 q'_{t'}\Del }{d}
\ll q_t Y_1 q'_{t'} (\log q'_{t'}) \Del \le q_t^{2+o(1)} Y_1 \Del \\
&\ll
\frac{X_1 Y_1 X_2 Y_2 \Del \mu (\cI)}{F(t) F(s)}.
\end{align*}

We have considered all cases, confirming \eqref{eq: Nts}. We conclude that Theorem \ref{thm: Liouville} holds in the case that $\gam_2 \in \cL$.

\subsection{Rational second shift}

Finally, we prove Theorem \ref{thm: Liouville} in the case that $\gam_2 \in \bQ$. Let $\gam_2 = c_0/d_0$, where $c_0 \in \bZ$ and $d_0 \in \bN$ are fixed and coprime. We follow the previous subsection, but this time we replace $q'_{t'}$ and $c_{t'}$ by $d_0$ and $c_0$, respectively, for all $t \in \cT$, and the proof carries through.

We have covered all possibilities for $\gam_2$, completing the proof of Theorem \ref{thm: Liouville}.

\section{Obstructions on Liouville fibres}\label{Sec: log square sharp}

Our proof of Theorem \ref{thm: log square is sharp}, via an Ostrowski expansion construction, rests upon the following technical lemma. In the following lemma and its proof, the implied constants are allowed to depend on $\alp$.

\begin{lemma}
\label{lem: Lioville log square sharp} Suppose $\alp,\gam$ satisfy \eqref{eq: DandyAndy}.
Let $m(n)$
denote the least $i \geq 0$ such that $\delta_{i+1}(n)\neq0$.
Define
\[
\W_{u,d}=\left\{ n\in\mathbb{N}:\,m(n)=u,\,\left|\delta_{u+1}(n)\right|=d\right\},
\]
whenever $1 \le d \le a_{u+1}-b_{u+1}$, as well as
\begin{align*}
S_{u,d} & =
\sum_{n\in\W_{u,d}}\frac{1}{n(\log n)^{2}\left\Vert n\alpha-\gamma\right\Vert }.
\end{align*}
Then $\min \W_{u,d} \gg q_{u}$, uniformly in $d$. 
Moreover, we have
\[
S_{u,d}
\ll \begin{cases}
\frac{1}{d\log q_{u+1}}, &\text{if } d > b_{u+1} \\
\frac1{\log q_u}, &\text{if }d = b_{u+1}\\
\frac{q_{u+1}}{\left(b_{u+1}-d\right)q_{u}(\log(\left(b_{u+1}-d\right)q_{u}))^{2}d}+\frac{1}{d\log q_{u+1}}, &\text{if } d < b_{u+1}.
\end{cases}
\]
\end{lemma}

\begin{proof}
For $n \in \cW_{u,d}$, Lemma \ref{lem: size of inhomogeneous distance} implies that
\begin{align*}
\left\Vert n\alpha-\gamma\right\Vert &\gg
\min\left(\left(d-1\right)\left|D_{u}\right|+a_{u+2}\left|D_{u+1}\right|,a_{1}\left|D_{0}\right| + a_2 |D_1| \right) \\
&= \left(d-1\right)\left|D_{u}\right|+a_{u+2}\left|D_{u+1}\right|.
\end{align*}
By (\ref{eq: standard estimate}), we have
\[
\left(d-1\right)\left|D_{u}\right|\gg\frac{d-1}{q_{u+1}}
\]
and
\[
a_{u+2}\left|D_{u+1}\right|\gg \frac{a_{u+2}}{q_{u+2}}=\frac{a_{u+2}}{a_{u+2}q_{u+1}+q_{u}} \gg \frac{1}{q_{u+1}}.
\]
Therefore
\begin{equation}
\left\Vert n\alpha-\gamma\right\Vert \gg\frac{d}{q_{u+1}}\label{eq: lower bound on recipical} \qquad (n\in\W_{u,d}).
\end{equation}
Using the notation of Lemma \ref{lem: gaps}, observe that
\[
\W_{u,d}=\A\left(b_{1},\ldots,b_{u},b_{u+1}+d\right)\cup\A\left(b_{1},\ldots,b_{u},b_{u+1}-d\right),
\]
where $\A\left(b_{1},\ldots,b_{u},b_{u+1}+d\right)$ is understood to be empty if $b_{u+1}+d>a_{u+1}$,
and $\A\left(b_{1},\ldots,b_{u},b_{u+1}-d\right)$ is empty if $b_{u+1}-d<0$.  

Observe that
\[
S_{u,d} = S^+_{u,d} + S^-_{u,d},
\]
where
\[
S^{\pm}_{u,d} = \sum_{n \in \cA(b_1,\ldots,b_u, b_{u+1} \pm d)} \frac1{n(\log n)^2 \| n \alp - \gam\|}.
\]

\underline{Case 1: $n\in\A(b_{1},\ldots,b_{u},b_{u+1}+d)$}

Then
\[
n\geq\sum_{0\leq k<u}b_{k+1}q_{k}+\left(b_{u+1}+d\right)q_{u}+\sum_{k>u}c_{k+1}q_{k}\geq b_{u+1}q_{u}\gg q_{u+1}.
\]
From Lemma \ref{lem: gaps}, applied with $m = u$ and $d_{u+1}=b_{u+1}+d>0$, any two distinct elements of $\A\left(b_{1},\ldots,b_{u},b_{u+1}+d\right)$
differ by at least $q_{u+1}$. So the $r^{th}$ smallest element $n_{r}$
of $\A\left(b_{1},\ldots,b_{u},b_{u+1}+d\right)$ satisfies 
\[
n_{r}\geq\min\A\left(b_{1},\ldots,b_{u},b_{u+1}+d\right)+(r-1)q_{u+1}\gg rq_{u+1},
\]
uniformly in $d$. 
Combining this with (\ref{eq: lower bound on recipical}), we deduce that
\begin{align*}
S_{u,d}^+ & \ll\sum_{r\geq1}\frac{1}{rq_{u+1}(\log(rq_{u+1}))^{2}\frac{d}{q_{u+1}}}
\\&=
\frac{1}{d}\left(\sum_{r<q_{u+1}}\frac{1}{r(\log(rq_{u+1}))^{2}}+\sum_{r\geq q_{u+1}}\frac{1}{r(\log(rq_{u+1}))^{2}}\right)\\
 & \leq \frac{1}{d}\left(\sum_{r<q_{u+1}}\frac{1}{r(\log q_{u+1})^{2}}+\sum_{r\geq q_{u+1}}\frac{1}{r(\log r)^{2}}\right).
\end{align*}
The first sum is $O(1/\log q_{u+1})$, and the second
sum is bounded by a constant times
\[
\sum_{j\geq\log q_{u+1}}e^{j}\frac{1}{e^{j}j^{2}}\ll\int_{\log q_{u+1}}^{\infty}\frac{\mathrm{d}x}{x^{2}}\ll\frac{1}{\log q_{u+1}}.
\]
Therefore
\[
S^+_{u,d}\ll\frac{1}{d\log q_{u+1}}.
\]

\underline{Case 2: 
$n\in\A(b_{1},\ldots,b_{u},b_{u+1}-d)$}

For the aforementioned set to be non-empty, following our earlier convention, we must have $1\leq d\leq b_{u+1}$. We distinguish two sub-cases.

\underline{Case 2a: $1\leq d<b_{u+1}$} 

Note that
\begin{align*}
n & \geq\sum_{0\leq k<u}b_{k+1}q_{k}+\left(b_{u+1}-d\right)q_{u}+\sum_{k>u}c_{k+1}q_{k}\geq\left(b_{u+1}-d\right)q_{u}.
\end{align*}
Let $n_0 = \min \cW_{u,d}$, and for $r \ge 1$ denote by $n_{r}$ the $r^{th}$ smallest element of $\W_{u,d} \setminus \{ n_0 \}$.
It follows from Lemma \ref{lem: gaps} that 
\[
n_{r}\geq rq_{u+1}+\left(b_{u+1}-d\right)q_{u} \qquad (r \ge 0).
\]
Together with \eqref{eq: lower bound on recipical}, this gives
\begin{align*}
S_{u,d}^- & \ll \sum_{r\geq0}\frac{1}{(rq_{u+1}+\left(b_{u+1}-d\right)q_{u})(\log(rq_{u+1}+\left(b_{u+1}-d\right)q_{u}))^{2}\frac{d}{q_{u+1}}}\\
 & \le \frac{q_{u+1}}{\left(b_{u+1}-d\right)q_{u}(\log(\left(b_{u+1}-d\right)q_{u}))^{2}d}
+\sum_{r\geq1}\frac{1}{rq_{u+1}(\log(rq_{u+1}))^{2}\frac{d}{q_{u+1}}}.
\end{align*}
As in Case 1, we have
\[
\sum_{r\geq1}\frac{1}{rq_{u+1}(\log(rq_{u+1}))^{2}\frac{d}{q_{u+1}}}\ll\frac{1}{d\log q_{u+1}},
\]
and so
\begin{align*}
S_{u,d}^-
\ll
\frac{q_{u+1}}{\left(b_{u+1}-d\right)q_{u}(\log(\left(b_{u+1}-d\right)q_{u}))^{2}d} 
+
\frac{1}{d\log q_{u+1}}.
\end{align*}

\underline{Case 2b: $d=b_{u+1}$}

Note that
\[
n\geq\sum_{0\leq k<u}b_{k+1}q_{k}\gg a_{u}q_{u-1}\gg q_{u}.
\]
Let $n_0 = \min \cW_{u,d}$, and for $r \ge 1$ denote by $n_{r}$ the $r^{th}$ smallest element of $\W_{u,d} \setminus \{ n_0 \}$. By \eqref{eq: lower bound on recipical}, we have
\begin{align*}
S_{u,d}^-
\ll T_1 + T_2,
\end{align*}
where
\[
T_j = \sum_{r \ge 0}
\frac1{n_{j+2r} (\log n_{j+2r})^2 \: \frac{b_{u+1}}{q_{u+1}}} \qquad (j=1,2).
\]
Let $j \in \{1,2\}$. We infer from Lemma \ref{lem: gaps} that if $r \ge 0$ then $n_{j+2r} \gg rq_{u+1} + q_u$. Whence
\begin{align*}
T_{j} & \ll\sum_{r\geq0}\frac{1}{(rq_{u+1}+q_{u})(\log(rq_{u+1}+q_{u}))^{2} \: \frac{b_{u+1}}{q_{u+1}}}\\
 & \ll\frac{1}{q_{u}(\log q_{u})^{2}\frac{b_{u+1}}{q_{u+1}}}+\sum_{r\geq1}\frac{1}{r(\log(rq_{u+1}))^{2}b_{u+1}}\\
 & \ll\frac{1}{(\log q_{u})^{2}}+\frac{1}{\log q_{u+1}} \ll \frac1{\log q_u}.
\end{align*}
\end{proof}

We are now in the position to prove the main result of this section.

\begin{proof}[Proof of Theorem \ref{thm: log square is sharp}]

Let $A=(a_{n})_{n=1}^\infty$ 
be a sequence in $64\mathbb{N}$, sufficiently rapidly-increasing that
\begin{itemize}
    \item The sequence defined by
    \[
    q_0 = 1,
    \qquad q_1 = a_1,
    \qquad 
    q_{u+1}=a_{u+1} q_u + q_{u-1}  \quad (u\geq 1)
    \]
    satisfies
\begin{equation} \label{eq: SuperConvergence}
\sum_{u\geq0}\left(\frac{1}{\log q_{u}}+\frac{1}{\grow(q_{u})}\right)<\infty
\end{equation}
\item $a_{u+1} \ge q_u!$ ($u \ge 0$).
\end{itemize} 
Let $\cV$ be the collection of all such sequences $A=(a_i)_i$.
For $A\in \cV$, define 
$\alpha(A) := \left[0;a_{1},a_{2},\ldots\right]$.
For $\sigma=\left(\sigma_{i}\right)_{i}
\in\left\{ 0,1\right\} ^{\mathbb{N}}$ and $A= (a_i)_i \in \cV$, 
we define a sequence $(b_{i}(A, \sigma))_i$ by
\[
b_{i}(A, \sigma)=\frac{a_{i}}{2^{1+\sigma_{i}}} \qquad (i \in \bN),
\]
as well as a real number 
\[
\gam(A,\sigma) := \sum_{k\geq 0} b_{k+1}(A, \sigma) D_{k}(A),
\]
where $D_k(A) = q_{k}\alp(A) -p_k$ and $p_k/q_k$ is the $k^{th}$
convergent to $\alp(A)$.
Note that we have \eqref{eq: DandyAndy}, so by Lemma \ref{lem: DandyAndy} we have
$\gam(A,\sig) \in [0,1-\alp(A))$ and
\[
\| n \alp(A) - \gam(A,\sig)\| \ne 0 \qquad (n \in \bN).
\]
There are continuum many $\alpha(A)$, and they are Liouville by Lemma \ref{BHV11}.
For the rest of the proof, we fix a pair $(\alpha(A), \gam(A,\sigma))$, where $A \in \cV$ and $\sig \in \{0,1\}^\bN$,
and abbreviate
$\alpha=\alpha(A)$, $\gam=\gam(A,\sigma)$.

Fix $\gam_2 \in\mathbb{R}$, and consider $\alp_1 = \alp$ and $\gam_1 = \gam$. By the Borel--Cantelli lemma,
it remains to prove that 
\begin{equation}
\sum_{n\geq1}\frac{\psio(n)}{\left\Vert n\alpha-\gamma\right\Vert }<\infty \label{eq: convergence},
\end{equation} 
where $\psio$ is as in \eqref{def: psi_1}. 
To this end, observe that 
\begin{align*}
\sum_{n\geq1}
\frac{\psio(n)}{\left\Vert n\alpha-\gamma\right\Vert } & 
\le\sum_{u\geq0}\sum_{d\neq b_{u+1}}\max_{n\in\W_{u,d}}
\frac{1}{\grow(n)}\,S_{u,d}
+\sum_{u\geq0}\max_{n\in\W_{u,b_{u+1}}}
\frac{1}{\grow(n)}\,S_{u,b_{u+1}},
\end{align*}
where here and henceforth $d\neq b_{u+1}$ means that 
\[
d \in \{1,\ldots,a_{u+1}-b_{u+1}\}\setminus \{b_{u+1}\}.
\]
Since $\grow$ non-decreasing and unbounded, we infer from Lemma \ref{lem: Lioville log square sharp}
that
$$
\max_{n\in\W_{u,d}}
\frac{1}{\grow(n)}= 
\frac{1}{\grow (\min \W_{u,d} )}
\ll \frac{1}{\grow(q_{u-1})},
$$
where here $q_{-1} = 1$. By Lemma \ref{lem: Lioville log square sharp}, we now have
$$
\sum_{n\geq 1}
\frac{\psio(n)}{\left\Vert n\alpha-\gamma\right\Vert }
\ll T_1 + T_2 + T_3,
$$
where
\begin{align*}
T_1 &= \sum_{u\geq 0}
\frac{1}{\grow(q_{u-1})}
\: \frac{1}{\log q_{u}},
\\
T_2 &= \sum_{u\geq 0} \frac{1}{\grow(q_{u-1})} \sum_{d\neq b_{u+1}} \frac{1}{d\log q_{u+1}},
\\
T_3 &=
\sum_{u\geq 0} \frac{1}{\grow(q_{u-1})} \sum_{d < b_{u+1}} \frac{q_{u+1}}{\left(b_{u+1}-d\right)
q_{u}(\log(\left(b_{u+1}-d\right)q_{u}))^{2}d}.
\end{align*}

We see from \eqref{eq: SuperConvergence} that $T_1 < \infty$. The convergence of $T_2$ also follows straightforwardly from \eqref{eq: SuperConvergence}, since
\[
T_2 \ll \sum_{u \ge 0}
\frac1{\xi(q_{u-1})} \: \frac{\log a_{u+1}}{\log q_{u+1}} \le \sum_{u \ge 0}
\frac1{\xi(q_{u-1})} < \infty.
\]
Our final task is to establish the convergence of the series defining $T_3.$ We begin with the observation that
\begin{align*}
T_3
&\le \sum_{u \ge 0}
\frac{q_{u+1}}{\xi(q_{u-1}) q_u } \sum_{d < b_{u+1}}
\frac1{d(b_{u+1} - d) (\log (b_{u+1}-d))^2} \\
&\ll
\sum_{u \ge 0}
\frac{b_{u+1}}{\xi(q_{u-1})}
\sum_{d < b_{u+1}}
\frac1{d(b_{u+1} - d) (\log (b_{u+1}-d))^2}.
\end{align*}
The inner sum is at most $X_u + Y_u$, where
\[
X_u = 
\sum_{d \le b_{u+1}/2} \frac1{d(b_{u+1} - d) (\log (b_{u+1}-d))^2} 
\]
and
\[
Y_u =
\sum_{t \le b_{u+1}/2}
\frac1{t(b_{u+1} - t) (\log t)^2} \ge X_u.
\]
Finally, we have
\begin{align*}
T_3 \ll \sum_{u \ge 0}
\frac{b_{u+1}}{\xi(q_{u-1})} Y_u \ll \sum_{u \ge 0}
\frac{1}{\xi(q_{u-1})}
\sum_{t \ge 1}
\frac1{t (\log t)^2} \ll \sum_{u \ge 0}
\frac{1}{\xi(q_{u-1})},
\end{align*}
which converges.
\end{proof}

\appendix

\section{Pathology} \label{pathology}

In this appendix, we establish Theorem \ref{thm: SpecialCase}. If we have \eqref{eq: PsiSmall} then the proof of Theorem \ref{thm: SpecialCaseWeak} prevails. We also need to consider the pathological situation in which $\Psi(n) > 1/2$ for infinitely many $n \in \bN$. We will slightly alter this dichotomy.

\bigskip

Let $c^*$ be a small, positive constant. Our implicit constants will not depend on $c^*$ unless otherwise stated. The first idea is to restrict the support of $\Psi$ to
\[
\cG^* = \left \{n \in \cG: 
\frac{\varphi(\un)}{\un} \ge c^* \right \},
\]
where $\cG$ and $\eta = \eta(\gam)$ are as in Section \ref{Sec: Inhomog. Gallagher}. 
That is, we introduce $\Psi^* = \Psi 1_{\cG^*}$
and
\[
\cE_n^* = \displaystyle
\left \{ \alpha \in [0,1]: 
\exists a\in \bZ\,
\mathrm{s.t.} \quad
\displaystyle \substack{
\displaystyle
a+\gamma \in \hat n {\set I}, \\
\displaystyle
\vert \hat n \alpha - \gamma - a \vert < \Psi^*(n), \\
\displaystyle
(a,\un) \text{ is } (\gam,\eta)\text{-shift-reduced}} \right \}.
\]

\underline{Case: $\Psi^*(n) \le 1/2$ for large $n$}

We commence by discussing \eqref{eq: Positive2}. Our modification can only reduce the left hand side of \eqref{eq: indepen of localised sets} so, by the reasoning of Proposition \ref{prop: pot}, it remains to justify
\eqref{eq: divergence of approx. localised sets} with $\cE_n^*$ in place of $\cE_n$. The upper bound is immediate from the inequality $\mu(\cE_n^*) \le \mu(\cE_n)$, leaving us to deal with the lower bound. By \eqref{eq: assumed measure bound approximation set}, we have
\begin{equation} \label{StarMeasure}
\sum_{n \le X} \mu(\cE_n^*) \gg \mu(\cI)
\sum_{\substack{n \in \cG^* \\ C_1 < n \le X}}
\frac{\varphi(\un)}{\un} \:
\frac{ \psi(\hat n)}
{ \| \hat n \alp_1 - \gam_1 \| \cdots \| \hat n \alp_{k-1} - \gam_{k-1} \| },
\end{equation}
where $C_1$ is a large, positive constant. We compute that
\begin{align*}
&\sum_{\substack{n \in \cG \setminus \cG^* \\ n \le X}}
\frac{\varphi(\un)}{\un} \:
\frac{ \psi(\hat n)}
{ \| \hat n \alp_1 - \gam_1 \| \cdots \| \hat n \alp_{k-1} - \gam_{k-1} \| }
\\
&< c^* \sum_{\substack{n \in \cG \\ n \le X}}
\frac{ \psi(\hat n)}
{ \| \hat n \alp_1 - \gam_1 \| \cdots \| \hat n \alp_{k-1} - \gam_{k-1} \| }
\\
&\ll c^* \sum_{n\leq X} \psi(\un) (\log n)^{k-1},
\end{align*}
where the final inequality follows from the calculations within the proof of Lemma~\ref{lem: divergence lemma}. As $c^*$ is small, combining this with \eqref{eq: show1} and \eqref{eq: show2} yields
\[
\sum_{\substack{n \in \cG^* \\ n \le X}}
\frac{\varphi(\un)}{\un} \:
\frac{ \psi(\hat n)}
{ \| \hat n \alp_1 - \gam_1 \| \cdots \| \hat n \alp_{k-1} - \gam_{k-1} \| } \gg \sum_{n\leq X} \psi(\un) (\log n)^{k-1}.
\]
Substituting this into \eqref{StarMeasure} gives
\[
\sum_{n \le X} \mu(\cE_n^*)
\gg \mu(\cI) \sum_{n \le X} \psi(\un) (\log n)^{k-1}.
\]
The upshot is that we have \eqref{eq: divergence of approx. localised sets} with $\cE_n^*$ in place of $\cE_n$, which is the last remaining ingredient needed for \eqref{eq: Positive2}. 

Let $N$ be large, and let us now specialise $\cI = [0,1]$. By \eqref{eq: HatDivergence} and the above, we have
\[
\sum_{n=1}^\infty \mu(\cE_n^*) = \infty.
\]
By \eqref{eq: lower bound of fake phi}, we have
\[
\frac{\varphi_{\gam,\eta}(\hat n)}{\hat n} \gg
\frac{\varphi(\hat n)}{\hat n} 
\ge c^* 
\qquad (n \in \cG^*, \quad n > N),
\]
and note also that
\[
\mu(\cE_n^*) \le \mu(\cE_n) \ll \Psi(n)
\le \Phi(\un)
\qquad (n > N).
\]
Therefore 
\[
\sum_{n=1}^\infty \frac{\varphi_{\gam,\eta}(\un)}{\un} \Phi(\un)
\gg c^* \sum_{\substack{n \in \cG^*\\n > N}} \mu(\cE_n^*) = c^* \sum_{n > N} \mu(\cE_n^*) = \infty,
\]
which implies \eqref{eq: Positive1}. Having established \eqref{eq: Positive1} and \eqref{eq: Positive2}, we have completed the proof of the theorem in this case.

\underline{Case: $\Psi^*(n) > 1/2$ infinitely often}

Let $N$ be large, and put 
\[
\cS = \{ n \in \bN: \Psi^*(n) > 1/2, \quad n > N \}.
\]
The inequalities
\[
\frac{\varphi_{\gam,\eta}(\un)}{\un} \gg c^*,
\quad
\Phi(\un) \ge \Psi^*(n) > 1/2
\qquad (n \in \cS),
\]
together with the infinitude of $\cS$, yield 
\[
\sum_{n \in \cS} \frac{\varphi_{\gam,\eta}(\un)}{\un} \Phi(\un) = \infty,
\]
which implies \eqref{eq: Positive1}.

For $n \in \bN$, define
\[
\Psi^\dagger(n)
= \begin{cases}
1/2, &\text{if } n \in \cS \\
0, &\text{if } n \notin \cS.
\end{cases}
\]
and
\[
\cE_n^\dagger = \displaystyle
\left \{ \alpha \in [0,1]: 
\exists a\in \bZ\,
\mathrm{s.t.} \quad
\displaystyle \substack{
\displaystyle
a+\gamma \in \hat n {\set I}, \\
\displaystyle
\vert \hat n \alpha - \gamma - a \vert < \Psi^\dagger(n), \\
\displaystyle
(a,\un) \text{ is } (\gam,\eta)\text{-shift-reduced}} \right \}.
\]
Let $\cI$ be a non-empty subinterval of $[0,1]$. Henceforth, our implied constants will be allowed to depend on $c^*$ but not $\cI$. By \eqref{eq: lower bound of fake phi}, we have
\[
\mu(\cE_n^\dagger) \gg \frac{\varphi_0(\un)}{\un} \gg \mu(\cI)
\qquad (n \in \cS),
\]
and clearly
\[
\mu(\cE_n^\dagger \cap \cE_m^\dagger) 
\le \mu(\cE_n^\dagger) \ll \mu(\cI) \qquad (n,m \in \cS).
\]
Applying Lemmata  \ref{lem: Borel-Cantelli} and
\ref{lem: density}, as in Proposition \ref{prop: pot}, furnishes 
\[
\mu \left(\limsup_{n \to \infty} \cA_n \right) = 1,
\]
where
\[
\cA_n = 
\left \{ \alpha \in [0,1]: 
\exists a\in \bZ\,
\mathrm{s.t.} \quad
\displaystyle \substack{
\displaystyle
\vert \hat n \alpha - \gamma - a \vert < \Psi^\dagger(n), \\
\displaystyle
(a,\un) \text{ is } (\gam,\eta)\text{-shift-reduced}} \right \}.
\]
This implies \eqref{eq: Positive2}, since $\Phi(\un) \ge \Psi^{\dagger}(n)$ for all $n$, and completes the proof of Theorem \ref{thm: SpecialCase}.

\begin{remark}
In the context of Remark \ref{rem: QuestionRemark}, our proof works for any sufficiently small $\eta$, by scaling $\eps_0$ accordingly.
\end{remark}

\providecommand{\bysame}{\leavevmode\hbox to3em{\hrulefill}\thinspace}


\begin{thebibliography}{50}

\bibitem{ANL2018}
F. Adiceam, E. Nesharim and F. Lunnon,
\emph{On the $t$-adic Littlewood conjecture}, Duke Math. J., to appear, DOI: 10.1215/00127094-2020-0077, arXiv:1806.04478.

\bibitem{Aist2014} 
C. Aistleitner, \emph{A note on the Duffin--Schaeffer conjecture with slow divergence}, Bull. Lond. Math. Soc. \textbf{46} (2014), 164--168.

\bibitem{Fufu} 
C. Aistleitner, T. Lachmann, M. Munsch, N. Technau and A. Zafeiropoulos, \emph{The Duffin--Schaeffer conjecture with extra divergence}, Adv. Math. \textbf{356} (2019).

\bibitem{Bad2013}
D. Badziahin, \emph{On multiplicatively badly approximable numbers}, Mathematika \textbf{59} (2013), 31--55.

\bibitem{BV2011}
D. Badziahin and S. Velani, \emph{Multiplicatively badly approximable numbers and generalised Cantor sets}, Adv. Math. \textbf{228} (2011), 2766--2796.

\bibitem{BW2014}
F. Barroero and M. Widmer, 
\emph{Counting lattice points and o-minimal structures}, 
Int. Math. Res. Not. \textbf{2014,} 4932--4957. 

\bibitem{Ber2012}
V. Beresnevich, \emph{Rational points near manifolds and metric Diophantine approximation}, Ann. of Math. (2), \textbf{175} (2012), 187--235.

\bibitem{BDV2006} V. Beresnevich, D. Dickinson and S. Velani, \emph{Measure theoretic laws for lim sup sets}, Mem. Amer. Math. Soc. \textbf{179} (2006).

\bibitem{BDV2007}
V. Beresnevich, D. Dickinson and S. Velani, \emph{Diophantine approximation on planar curves and the distribution of rational points}, Ann. of Math. (2) \textbf{166} (2007), 367--426.

\bibitem{BHHV2013}
V. Beresnevich, G. Harman, A. Haynes and S. Velani, \emph{The Duffin-Schaeffer conjecture with extra divergence II}, Math. Z. \textbf{275} (2013), 127--133. 

\bibitem{BHV2016} 
V. Beresnevich, A. Haynes and S. Velani, \emph{Sums of reciprocals of fractional parts and multiplicative Diophantine approximation}, Mem. Amer. Math. Soc. \textbf{263} (2020).

\bibitem{BRV2016} 
V. Beresnevich, F. Ram\'irez and S. Velani, \emph{Metric Diophantine Approximation: some aspects of recent work}, Dynamics and Analytic Number Theory, London Math. Soc. Lecture Note Ser. (N.S.) \textbf{437,} Cambridge University Press, 2016, pp. 1--95.

\bibitem{BV2010}
V. Beresnevich and S. Velani, \emph{An inhomogeneous transference principle and Diophantine approximation}, Proc. Lond. Math. Soc. (3) \textbf{101} (2010), 821--851. 

\bibitem{BV2015}
V. Beresnevich and S. Velani, \emph{A note on three problems in metric Diophantine approximation}, Recent Trends in Ergodic Theory and Dynamical Systems, Contemp. Math. \textbf{631} (2015), 211--229.

\bibitem{Bugeaud2004}
Y. Bugeaud, \emph{Approximation by algebraic numbers}, Cambridge Tracts in Mathematics \textbf{160,} Cambridge University Press, Cambridge, 2004. 

\bibitem{BL2005}
Y. Bugeaud and M. Laurent, \emph{On exponents of homogeneous and inhomogeneous Diophantine approximation}, Mosc. Math. J. \textbf{5} (2005), 747--766. 

\bibitem{BL2010}
Y. Bugeaud and M. Laurent, \emph{On transfer inequalities in Diophantine approximation, II}, Math. Z. \textbf{265} (2010), 249--262.

\bibitem{Cas1997}
J. W. S. Cassels, \emph{An introduction to the geometry of numbers}, Springer, 1997.

\bibitem{Cho2018}
S. Chow, \emph{Bohr sets and multiplicative diophantine approximation}, Duke Math. J. \textbf{167} (2018), 1623--1642.

\bibitem{CGGMS}
S. Chow, A. Ghosh, L. Guan, A. Marnat and D. Simmons, \emph{Diophantine transference inequalities: weighted, inhomogeneous, and intermediate exponents}, Ann. Sc. Norm. Super. Pisa Cl. Sci. (5) \textbf{XXI} (2020), 643--671.

\bibitem{CT2019}
S. Chow and N. Technau,
\emph{Higher-rank Bohr sets and 
multiplicative diophantine approximation},
Compositio Math. \textbf{155} (2019), 2214--2233.

\bibitem{CY2020}
S. Chow and L. Yang, 
\emph{Effective equidistribution for multiplicative diophantine approximation on lines}, arXiv:1902.06081.

\bibitem{Dan1985}
S. G. Dani, \emph{Divergent trajectories of flows on homogeneous spaces and Diophantine
approximation}, J. Reine Angew. Math. \textbf{359} (1985), 55--89.

\bibitem{D1951}
H. Davenport, 
\emph{On a principle of {L}ipschitz}, 
J. London Math. Soc. \textbf{1} (1951), 179--183.

\bibitem{DS1941}
R. J. Duffin and A. C. Schaeffer, 
\emph{Khintchine's problem in metric Diophantine approximation}, Duke Math. J. \textbf{8} (1941), 243--255.
 
\bibitem{EKL2006}
M. Einsiedler, A. Katok, and E. Lindenstrauss,
\emph{Invariant measures and the 
set of exceptions to Littlewood's conjecture},
Ann. of Math. (2) \textbf{164} (2006), 513--560.

\bibitem{Erd1970}
P. Erd\H{o}s, 
\emph{On the distribution of convergents 
of almost all real numbers}, 
J. Number Theory \textbf{2} (1970), 425--441.

\bibitem{Erd1962}
P. Erd\H{o}s, 
\emph{Representations of real numbers as sums and products of Liouville numbers}
Michigan Math. J. \textbf{9} (1962), 59--60.

\bibitem{Fal2004}
K. Falconer, \emph{Fractal Geometry: Mathematical Foundations and Applications},
John Wiley \& Sons, 2004.

\bibitem{opera} J. Friedlander and H. Iwaniec, \emph{Opera de cribro}, American Mathematical Society Colloquium Publications, vol. 57, American Mathematical Society, Providence, RI, 2010. 

\bibitem{Gal1962}
P. X. Gallagher, \emph{Metric simultaneous diophantine approximation}, J. Lond. Math. Soc. \textbf{37} (1962), 387--390.

\bibitem{GM2019}
A. Ghosh and A. Marnat, \emph{On diophantine transference principles}, Math. Proc. Camb. Phil. Soc. \textbf{166} (2019), 415--431.

\bibitem{GP2018}
A. Gorodnik and P. Vishe,
\emph{Diophantine approximation for products 
of linear maps—logarithmic improvements},
Trans. Amer. Math. Soc. \textbf{370} (2018), 487--507.

\bibitem{Har1998} 
G. Harman, \emph{Metric number theory}, London Math. Soc. Lecture Note Ser. (N.S.), vol. 18, Clarendon Press, Oxford 1998.

\bibitem{Har1990} 
G. Harman, \emph{Some cases of the Duffin and Schaeffer conjecture}, 
Quart. J. Math., \textbf{41} (1990), 395--404.

\bibitem{HPV2012}
A. Haynes, A. Pollington and S. Velani, \emph{The Duffin-Schaeffer Conjecture with extra divergence}, Math. Ann. \textbf{353} (2012), 259--273.

\bibitem{Hua2015}
J.-J. Huang, \emph{Rational points near planar curves and Diophantine approximation}, Adv. Math. \textbf{274} (2015), 490--515.

\bibitem{Hua2020}
J.-J. Huang, \emph{The density of rational points near hypersurfaces}, Duke Math. J. \textbf{169} (2020), 2045--2077.

\bibitem{HS2018}
M. Hussain and D. Simmons, \emph{The Hausdorff measure version of Gallagher's theorem --- closing the gap and beyond}, J. Number Theory \textbf{186} (2018), 211--225.

\bibitem{Khi1926} A. Ya. Khintchine, \emph{\"Uber eine Klasse linearer diophantischer Approximationen}, Rendiconti Circ. Mat. Palermo \textbf{50} (1926), 170--195.

\bibitem{KM1998}
D. Y. Kleinbock and G. A. Margulis, \emph{Flows on homogeneous spaces and Diophantine approximation on manifolds}, Ann. of Math. (2) \textbf{148} (1998), 339--360.

\bibitem{DimitrisBook}
D. Koukoulopoulos, \emph{The distribution of prime numbers},
Graduate Studies in Mathematics, vol. 203, American Mathematical Society, Providence, RI, 2019.

\bibitem{KM2020} D. Koukoulopoulos and J. Maynard,
\emph{On the Duffin-Schaeffer conjecture}, Ann. of Math. (2) \textbf{192} (2020), 251--307.

\bibitem{LV2015} T.-H. L\^e and J. Vaaler, \emph{Sums of products of fractional parts}, Proc. Lond. Math. Soc. (3) \textbf{111} (2015), 561--590.

\bibitem{Mar2000}
G. Margulis, \emph{Problems and conjectures in rigidity theory},
  Mathematics: frontiers and perspectives, 2000, 161--174.
 
\bibitem{MK1998}
M. Mukherjee and G. Karner, \emph{Irrational numbers of constant type --- a new characterization}, New York J. Math. \textbf{4} (1998), 31--34.

\bibitem{Peck1961}
L. G. Peck, \emph{Simultaneous rational approximations to algebraic numbers}, Bull. Amer. Math. Soc. \textbf{67} (1961), 197--201.

\bibitem{PV1990}
A. D. Pollington and R. C. Vaughan,
\emph{The k-dimensional Duffin and Schaeffer conjecture},
Mathematika (2) \textbf{37} (1990), 190--200.

\bibitem{PV2000}
A. D. Pollington and S. L. Velani,
\emph{On a problem in simultaneous Diophantine 
approximation: Littlewood's conjecture}, Acta Math.
\textbf{185} (2000), 287--306.

\bibitem{Ram2016}
F. Ram\'{i}rez, \emph{Counterexamples, covering systems, and zero-one laws for inhomogeneous approximation}, Int. J. Number Theory \textbf{13} (2017), 633--654.

\bibitem{Ram2017}
F. Ram\'{i}rez, \emph{Khintchine's theorem with random fractions}, Mathematika \textbf{66} (2020) 178--199.

\bibitem{RS1992}
A. Rockett, and P. Sz\"{u}sz, \emph{Continued Fractions}.
World Scientific (1992), Singapore.

\bibitem{Sha2011}
U. Shapira,
\emph{A solution to a problem of Cassels and Diophantine properties of cubic numbers}, Ann. of Math. (2) \textbf{173} (2011), 543--557.

\bibitem{Szu1958}
P. Sz\"{u}sz, \emph{\"{U}ber die metrische Theorie der Diophantischen Approximation}, 
Acta. Math. Sci. Hungar. \textbf{9} (1958), 177--193.

\bibitem{TV2006}
T. Tao and V. Vu, \emph{Additive combinatorics}, Cambridge Stud. Adv. Math., vol. 105, Cambridge University Press, Cambridge, 2006.

\bibitem{TV2008}
T. Tao and V. Vu, \emph{John-type theorems for generalized arithmetic progressions and iterated sumsets}, Adv. Math. \textbf{219} (2008), 428--449.

\bibitem{T1993} J. L. Thunder, \emph{The number of solutions of bounded height to a system of linear equations}, J. Number Theory \textbf{43} (1993), 228--250.

\bibitem{Vaa1978}
J. Vaaler, \emph{On the metric theory of Diophantine approximation},
Pacific J. Math. \textbf{76} (1978), 527--539.

\bibitem{VV2006}
R. C. Vaughan and S. Velani, \emph{Diophantine approximation on planar curves: the convergence theory}, Invent. Math. \textbf{166} (2006), 103--124.

\bibitem{YuFourier} 
H. Yu, \emph{A Fourier-analytic approach to inhomogeneous Diophantine approximation}, Acta Arith.
\textbf{190} (2019), 263--292.

\bibitem{YuEV} 
H. Yu, \emph{On the metric theory of inhomogeneous Diophantine
approximation: An Erd\H{o}s-Vaaler type result}, J. Number Theory
\textbf{224} (2021), 243--273.

\end{thebibliography}
\end{document}